\documentclass{amsart}

\usepackage{datetime}
\usepackage{amsfonts}
\usepackage{amsmath}
\usepackage{amssymb}
\usepackage{amsthm}
\usepackage{eucal}
\usepackage{graphicx,color}
\usepackage[usenames,dvipsnames,svgnames,table]{xcolor}

\definecolor{duckgreen}{rgb}{0.0,0.309,0.153}
\definecolor{burntorange}{rgb}{0.801,0.332,0.0}

\usepackage{hyperref}
\hypersetup{colorlinks,
            filecolor=black,
            linkcolor=MidnightBlue,
            citecolor=NavyBlue,
            urlcolor=RoyalBlue,
            bookmarksopen=true}

\theoremstyle{plain}
\newtheorem{corollary}{Corollary}
\newtheorem{lemma}{Lemma}

\newtheorem{theorem}{Theorem}
\newtheorem*{thmA}{Theorem A}
\newtheorem*{thmB}{Theorem B}
\newtheorem*{main}{Main Theorem}
\newtheorem{definition}{Definition}
\newtheorem{remark}{Remark}
\theoremstyle{definition}
\newtheorem*{ack}{Acknowledgment}

\numberwithin{equation}{section}

\newcommand{\bP}{\mathbb P}
\newcommand{\bR}{\mathbb R}
\newcommand{\bS}{\mathbb S}
\newcommand{\cM}{\mathcal{M}}
\newcommand{\cN}{\mathcal{N}}
\newcommand{\cO}{\mathcal{O}}
\newcommand{\dt}{\partial_\tau}
\newcommand{\Dt}{\frac{d}{d\tau}}
\newcommand{\dx}{\partial_x}
\newcommand{\dy}{\partial_y}
\newcommand{\dz}{\partial_\theta}
\newcommand{\Dy}{\mathrm{d}y}
\newcommand{\Dz}{\mathrm{d}\theta}

\newcommand{\lp}{\langle}
\newcommand{\ls}{\lesssim}
\newcommand{\lan}{\langle}
\newcommand{\mcf}{\textsc{mcf}}
\newcommand{\mH}{\mathcal H_a}
\newcommand{\mL}{\mathcal L_a}
\newcommand{\opt}{\mathrm{opt}}
\newcommand{\pop}{\phi_{\opt}}
\newcommand{\ran}{\rangle}
\newcommand{\rd}{\mathrm{d}}
\newcommand{\rp}{\rangle_\mu}
\newcommand{\rs}{\rangle_{\bS^1}}
\newcommand{\spec}{\mathrm{spec}}
\newcommand{\ve}{\varepsilon}
\newcommand{\vp}{\varphi}

\newcounter{mnotecount}[section]
\let\oldmarginpar\marginpar
\renewcommand\marginpar[1]{\-\oldmarginpar[\raggedleft\footnotesize #1]%
{\raggedright\footnotesize #1}}

\newcommand{\V}[2]% Requires at least one nonzero argument to work properly.
    {\ifnum0=#1  %There are no y derivatives.
        \ifnum1=#2
            (v_{0,1,2})   %(v^{-2} \dz v)
        \else
            (v_{0,#2,#2}) %(v^{-#2} \dz^{#2} v)
        \fi
    \else   %There are y derivatives.
        \ifnum0=#2  %But there are no theta derivatives.
            \ifnum1=#1
                (v_{1,0,1})   %(v^{-1} \dy v)
            \else
                (v_{#1,0,0})  %(\dy^{#1} v)
            \fi
        \else   %There are both y and theta derivatives.
            (v_{#1,#2,#2})
        \fi
    \fi}

\begin{document}
\title[Universality in MCF neckpinches]
{Universality in mean curvature flow neckpinches}

\author{Zhou Gang}
\address[Zhou Gang]{California Institute of Technology}
\email{gzhou@caltech.edu}
%%\urladdr{???.caltech.edu}

\author{Dan Knopf} \address[Dan Knopf]{University of Texas at Austin}
\email{danknopf@math.utexas.edu}
\urladdr{http://www.ma.utexas.edu/users/danknopf/}

\thanks{
	ZG thanks NSF for support in DMS-1308985.
	DK thanks NSF for support in DMS-1205270.}
 
\begin{abstract}
We study noncompact surfaces evolving by mean curvature flow. Without
any symmetry assumptions, we prove that any solution that is $C^3$-close at some time
to a standard neck will develop a neckpinch singularity in finite time, will become
asymptotically rotationally symmetric in a space-time neighborhood of its singular set,
and will have a unique tangent flow.
\end{abstract}

\maketitle

\setcounter{tocdepth}{1}
\tableofcontents

\section{Introduction}

In this paper, we prove, without imposing any symmetry assumptions, that any complete
noncompact two-dimensional solution of mean curvature flow (\mcf) that is close to a standard
round neck at some time will \textsc{(i)} encounter a finite-time singularity, \textsc{(ii)} become
asymptotically rotationally symmetric  in a space-time neighborhood of the developing singularity,
\textsc{(iii)} satisfy an exact asymptotic profile in that neighborhood, and \textsc{(iv)} will as a
consequence have a unique tangent flow. All of these statements are made precise below. This
result extends our previous work on \mcf\ singularities: in that work \cite{GKS11}, we removed
the hypothesis of rotational symmetry but retained certain discrete symmetry hypotheses that
served to fix the limiting cylinder. The results in this paper, combined with significant work of
others, makes it reasonable to expect that singularities of generic \mcf\ solutions may be
constrained to a small selection of ``universal'' asymptotic profiles. Before giving the details,
of our results, we sketch the broad outlines of the emerging picture motivating this expectation.

One says a smooth one-parameter family $\cM^m_t\subset\bR^{m+1}$ of hypersurfaces moves by
\mcf\ if at every point $x\in\cM^m_t$, one has $\partial_t x = -H\nu$, where $H$ is the mean curvature
scalar, and $\nu$ is the unit normal at $x$. Denote the backward heat kernel at $(x_0,t_0)$ by
$\rho_{(x_0,t_0)}(x,t)=\{4\pi (t_0-t)\}^{-m/2}\exp\{-|x-x_0|^2/4(t_0-t)\}$.
For $t<t_0$, Huisken's \mcf\ monotonicity formula \cite{Huisken90} states that
\[
    \frac{d}{dt}\int_{\cM^m_t}\rho_{(x_0,t_0)}(x,t)\,\rd\mu
    =-\int_{\cM^m_t}\left|\frac{(x-x_0)^\perp}{2(t_0-t)}-H\nu\right|^2\rho_{(x_0,t_0)}(x,t)\,\rd\mu.
\]
A consequence of this is the following characterization of tangent flows (singularity models)
for Type-I \mcf\ singularities:

\begin{thmA}[Huisken]
Given any sequence of parabolic dilations at a Type-I \mcf\ singularity, there exists a subsequence
that converges smoothly to a nonempty immersed self-similarly shrinking solution.
\end{thmA}

In important subsequent work, Huisken and Sinestrari prove that any mean convex solution becomes
asymptotically convex at its first singular time \cite{HS99}, and they develop a complete surgery
program for solutions originating from two-convex initial data $\cM^m_0\subset\bR^{m+1}$
that are immersed images of compact manifolds of dimensions $m\geq3$ \cite{HS09}. In
the latter work, they show in particular that all singularities of such solutions are either
spherical $(\bS^m)$ or neckpinch $(\bR\times\bS^{m-1})$ singularities. Subsequently, Brendle
and Huisken \cite{BH13}  extend this result to \mcf\ with surgery of mean-convex initial surfaces
$\cM^2_0\subset\bR^3$. Very recently, Haslhofer and Kleiner \cite{HH14} obtain similar results for
\mcf\ with surgery, without dimension restrictions, using shorter proofs that rely on blow-up arguments.

Also recently, given $t_0>0$ and a hypersurface $\cM^m\subset\bR^{m+1}$, Colding and Minicozzi
consider a functional $F_{(x_0,t_0)}(\cM^m)=\int_{\cM^m}\rho_{(x_0,t_0)}(\cdot,0)\,\rd\mu$, and an
entropy $\lambda(\cM^m)=\sup_{(x_0,t_0)}F_{(x_0,t_0)}(\cM^m)$. The entropy $\lambda$ is invariant
under dilations and Euclidean motions, and $\lambda(\cM^m_t)$ is nonincreasing under \mcf.
Among other results, they prove a stability property for tangent flows  \cite{CM09}.
\begin{thmB}[Colding--Minicozzi] Let $\cM^m\subset\bR^{m+1}$ be a smooth\footnote{Recall that
by Ilmanen's $\bR^3$ blowup theorem \cite{Ilmanen95b}, a compact surface evolving by
\mcf\ has a smooth singularity model  at its first singular time. In dimensions $3\leq m\leq 6$,
Theorem~B also holds if $\cM^m$ is merely smooth away from a singular set of locally finite
$(m-2)$-dimensional Hausdorff measure.}
complete embedded self-shrinker with polynomial volume growth and without boundary.

If $\cM^m$ is not equal to any $\bS^k\times\bR^{m-k}$, $0\leq k\leq m$, then for any fixed
$r$, there is a graph $\cN^m$ over $\cM^m$ of a function with arbitrarily small $C^r$ norm
such that $\lambda(\cN^m)<\lambda(\cM^m)$.
%% Furthermore, if $\cM^m$ is not equal to $\bS^m$ and does not split off a line, then
%% the function above can be taken to have compact support.
\end{thmB}

A consequence of Theorem~B is that $\cM^m$ cannot occur as a tangent flow of a \mcf\ solution whose
initial data are any perturbations $\cN^m$ of $\cM^m$. It follows that spheres and cylinders are
the only generic self-similarly shrinking tangent flows.

\medskip

Left open by these important results is the question of whether singularity models are independent
of subsequence. This has been answered affirmatively by Schulze \cite{Schulze11} (using the
Simon--{\L}ojasiewicz inequality) if one tangent flow is a closed, multiplicity-one, smoothly embedded
self-similar shrinker, but the general case remains open. Progress toward resolving this
question was recently made by Colding, Ilmanen, and Minicozzi \cite{CIM13}, who prove that if one
tangent flow at a singularity is a multiplicity-one generalized cylinder, then every subsequential limit
is some generalized cylinder. Since this paper was written, its authors learned of important work
of Colding and Minicozzi \cite{CM14}, who prove that tangent cones are
unique for generic singularities of \mcf, and for all singularities of mean-convex \mcf.
Other recent progress towards classifying singularity models comes
from work of Wang \cite{Wang11}, who proves that  there is at most one smooth complete properly
embedded self-shrinker asymptotic at spatial infinity to any regular cone.

In this paper, we show that the tangent cylinders of certain \mcf\ neckpinches are independent of
subsequence by proving that these singularities have unique asymptotic behavior. With few
exceptions (i.e.~\cite{DS10} for logarithmic fast diffusion) asymptotic analysis results of this nature
(e.g.~\cite{King93} and \cite{DdP07} for logarithmic fast diffusion; \cite{AK07}, \cite{AIK11}, and
\cite{AIK12} for Ricci flow; and \cite{AV97} and \cite{GS09} for \mcf) require a restrictive
hypothesis of rotational symmetry.  In \cite{GKS11}, with I.M.~Sigal, we made partial progress
towards showing uniqueness by removing the
hypothesis of rotational symmetry for \mcf\ neckpinches, retaining only weaker discrete symmetries,
and then proving that neckpinches with these discrete symmetries asymptotically become rotationally
symmetric. (See \cite{GKS11} for precise statements of these assumptions and results.) Here,
we extend that work by removing symmetry assumptions altogether. 
We show that neckpinches originating from initial data sufficiently close to a formal
solution develop unique asymptotic profiles, modulo dilations and Euclidean motions.

At least if Ilmanen's multiplicity conjecture \cite{Ilmanen95a} is true, then combining our work here
with the results cited above makes it reasonable to conjecture that  \mcf\ solutions originating from
generic initial data $\cM^2_0\subset\bR^3$ are constrained to only two universal asymptotic profiles
if they become singular.

\medskip

Passing to higher dimensions, 
we strongly conjecture that our main results in this paper generalize to hypersurface neckpinches
$\cM^m_t\subset\bR^{m+1}$ in all dimensions. Indeed, the strategy of the proof,
as outlined in Section~\ref{Strategy}, rests on the fact that although the linearization of \mcf\ at
a cylinder $\bR\times\bS^1$ is formally unstable, the unstable eigenmodes are simply
``coordinate instabilities.'' By properly choosing seven coordinate parameters, a solution is
decomposed into a dominant component, and a rapidly-decaying component on which the
linearization is strictly stable. The only complications in extending these ideas to dimensions
$m>2$ are that one must deal with $m+5$ coordinate parameters, and with curvature terms
that arise from commuting covariant derivatives on $\bS^{m-1}$. The latter, however, are
harmless, because the coordinate parametrization controls the size of the cylinder.

\medskip

This paper is structured as follows.
As in \cite{GKS11}, we prove our results by means of two bootstrap machines. In Section~\ref{Preliminaries},
we establish notation, state our main assumptions, and outline the main ideas introduced here to generalize
our earlier work. In Section~\ref{Derived}, we derive the evolution equations for the quantities under analysis.
We construct the first bootstrap machine in Sections~\ref{FirstBootstrap}--\ref{InnerProof}
and the second in Sections~\ref{SecondBootstrap}--\ref{SlowlyDecompose}. We complete the proof of the
Main Theorem in Section~\ref{BigFinish}. The proofs of several supporting technical results appear in the appendices.

We now summarize the main results of this paper, using the following terminology.
We say that $(x,r,\theta)$ is a \emph{cylindrical coordinate system} if there exist orthonormal coordinates
$(x_0,x_1,x_2)$ for $\bR^3$ such that $x=x_0$ is the cylindrical axis and $(r,\theta)$ are polar coordinates
for the $(x_1,x_2)$-plane. In this notation, the standard cylinder with axis $x$ is the set $\{r=1\}$, and
a normal graph is determined by $r=u(x,\theta)$.

\medskip

\begin{main}
Suppose that a solution of \mcf\ satisfies Assumptions~[A1]--[A7] from Section~\ref{MainAssumptions}
for sufficiently small $b_0,c_0$.

\textsc{(i)} There exists $T<\infty$ such that the solution becomes singular at time $T$.

\textsc{(ii)} There exists a sequence $(x,r,\theta)_n$ of cylindrical coordinate systems 
such that for all times $t_{n-1}\leq t\leq t_n$, the surface can be written as the normal
graph of a positive function $u(x,\theta,t)$. Moreover, $t_n\nearrow T$ as $n\rightarrow\infty$,
and there exists a limiting $(x,r,\theta)_n\rightarrow(x,r,\theta)_\infty$ cylindrical coordinate system.
In these coordinates, the solution develops a neckpinch, with $u(x,\cdot,t)=0$ if and only if
$x=0$ and $t=T$.
    
\textsc{(iii)} In the cylindrical coordinate system $(x,r,\theta)_n$ constructed at time $t_n$,
the solution admits an ``optimal coordinate'' decomposition
\[
 \frac{u(x,\theta,t_n)}{\lambda_\opt(t_n)}
 =\sqrt{\frac{2+b_\opt(t_n)y^2}{1+\frac12 b_\opt(t_n)}}+\phi_\opt(y,\theta,t_n),
 \]
where $y=\lambda_\opt(t_n)^{-1}x$. The parameters $\lambda_\opt(t_n)$
and $b_\opt(t_n)$ have the asymptotic behavior that as $t_n\nearrow T$,
\begin{align*}
\lambda_\opt(t_n)  &=\{1+o(1)\}\sqrt{T-t_n},\\
b_\opt(t_n)             &=\{1+o(1)\}(-\log(T-t_n))^{-1}.
\end{align*}

\textsc{(iv)} The solution is asymptotically rotationally symmetric near the singularity --- in the
precise technical senses that there exist $\ve_1,\ve_2>0$ such that in each cylindrical coordinate
system $(x,r,\theta)_n$ at all times $t\in[t_{n-1},t_n]$, there exist $C^1$ functions $\lambda$, $a$,
$b$, and $\beta_0,\dots,\beta_4$ of time, with $\lambda(t_n)=\{1+o(1)\}\lambda_\opt(t_n)$,
$a(t_n)=\{1+o(1)\}\frac12$, $b(t_n)=\{1+o(1)\}b_\opt(t_n)$, and $|\beta_k(t)|\ls(-\log(T-t))^{-2}$
for $k=0,\dots,4$,
such that the $\theta$-dependent component $\phi$ of the solution defined by
\[
\phi:=\frac{u}{\lambda}-\sqrt\frac{2+by^2}{2-2a}-\beta_0y-\beta_1\cos\theta
	-\beta_2\sin\theta-\beta_3y\cos\theta-\beta_4y\sin\theta
\]
satisfies the derivative decay estimates
\begin{multline*}
v^{-2}|\dz\phi|+v^{-1}|\dy\dz\phi|+v^{-2}|\dz^2\phi|
+v^{-1}|\dy^2\dz\phi|+v^{-2}|\dy\dz^2\phi|+v^{-3}|\dz^3\phi|\ls b(t)^{2-\ve_1},
\end{multline*}
and the $C^0$ decay estimates
\[
    \frac{|\phi(y,\theta,t)|}{(1+y^2)^{\frac{3}{2}}}\ls b(t)^{2-\ve_2}
    \qquad\text{and}\qquad
    \frac{|\phi(y,\theta,t)|}{(1+y^2)^{\frac{11}{20}}}\ls b(t)^{1-\ve_1},
\]
all holding uniformly as $t_n\nearrow T$.\
\end{main}

\begin{ack}
The authors are deeply grateful to I.M.~Sigal for many helpful discussions related to extensions
of their work in \cite{GKS11}.
\end{ack}

\section{Preliminaries}\label{Preliminaries}

\subsection{Notation}\label{Notation}
We study the evolution of embedded graphs over a cylinder $\mathbb{S}^1\times\mathbb{R}$ in $\mathbb{R}^3$.
We consider initial data $\cM_0$ expressed in a cylindrical coordinate system by a positive function $u_0(x,\theta)$.
Then for as long as the flow remains a graph, all $\cM_{t}$ are determined by $r=u(x,\theta,t)$, where, as we showed
in \cite{GKS11}, $u$ satisfies the initial condition $u(x,\theta,0)=u_0(x,\theta)$ and evolves by
\begin{equation}\label{MCF-u}
    \partial_t u =
    \frac{\{1+(\frac{\dz u}{u})^2\}\partial_x^2 u
      +\frac{1+(\partial_x u)^2}{u^2}\dz^2 u
      -2\frac{(\partial_x u)(\dz u)^2}{u^3} \partial_x \dz u
      -\frac{(\dz u)^2}{u^3}}
    {1+(\partial_x u)^2 + (\frac{\dz u}{u})^2}
    -\frac{1}{u}.
 \end{equation}

As in \cite{GKS11}, we apply adaptive rescaling, transforming the original space-time variables
$x$ and $t$ into rescaled blowup variables\footnote{In the sequel, we abuse notation by using whichever
time scale ($t$ or $\tau$) is most convenient.}
\[
    y(x,t):=\lambda^{-1}(t)x \qquad\text{and}\qquad \tau(t):=\int_0^t \lambda^{-2}(s)\,\mathrm{d}s,
\]
where $\lambda>0$ is a scaling parameter to be chosen, with $\lambda(0)=\lambda_0$. The
equation's scaling symmetry allows us, without loss of generality, to fix $\lambda_0=1$. Note
that $y=0$ is only the approximate center of the developing neckpinch; we address this issue
below. We define a rescaled radius $v(y,\theta,\tau)$ by
\begin{equation}\label{Define-v}
    v\big(y(x,t),\theta,\tau(t)\big):=\lambda^{-1}(t)\,u(x,\theta,t).
\end{equation}
Then $v$ initially satisfies $v(y,\theta,0)=v_0(y,\theta)$,
where $v_0(y,\theta):=\lambda^{-1}_{0} u_{0}(\lambda_{0}y,\theta)$.

With respect to commuting $(y,\theta,\tau)$ derivatives, the quantity $v$ evolves by
\begin{equation}\label{MCF-v}
    \dt v = A_v v +av -v^{-1},
\end{equation}
where $A_v$ is the quasilinear elliptic operator
\begin{equation} \label{Define-A}
    A_v :=
    F_1(p,q)\dy^2 + v^{-2}F_2(p,q)\dz^2 + v^{-1}F_3(p,q)\dy\dz
    +v^{-2}F_4(p,q)\dz -ay\dy.
\end{equation}
The coefficients of $A_v$ are defined by
\begin{equation}\label{Define-Fi}
\begin{array}{llcllllc}
    F_1(p,q) &:= &\displaystyle\frac{1+q^2}{1+p^2+q^2}, &\quad &\quad
    F_2(p,q) &:= &\displaystyle\frac{1+p^2}{1+p^2+q^2},\\
    \\
    F_3(p,q) &:= &\displaystyle-\frac{2pq}{1+p^2+q^2}, &\quad &\quad
    F_4(p,q) &:= &\displaystyle\frac{q}{1+p^2+q^2},
\end{array}
\end{equation}
where
\begin{equation}\label{Define-apq}
    a:=-\lambda\partial_t \lambda,\qquad p:=\dy v,\qquad\text{and}\qquad q:=v^{-1}\dz v.
\end{equation}

Before stating our assumptions, we require some further notation. We denote the
formal solution of the adiabatic approximation to equation~\eqref{MCF-v} by
\begin{equation}\label{adiabatic}
V_{r,s}(y):=\sqrt{\frac{2+sy^2}{2-2r}},
\end{equation}
where $r$ and $s$ are positive parameters. We introduce a step function
\begin{equation}\label{Define-g}
    g(y,s):=
    \left\{\begin{array}{ccc}
    \frac{19}{20}\sqrt{2} & \text{if} & s y^{2}< 20,\\
    \\
    4 & \text{if} & s y^{2} \geq 20.
    \end{array}\right.
\end{equation}
This function differs from that used in \cite{GKS11} in that we here prescribe a slightly
sharper constant in the inner region $\{\beta y^2\leq20\}$, where
\[
    \beta(\tau):=\big(\kappa_0+\tau\big)^{-1}.
\]
This modification is not essential but simplifies some of the extra work required to
compensate for the lack of discrete symmetry assumptions in this paper.\footnote{Here,
$\kappa_0\gg1$ is a large constant to be fixed below. For simplicity, we may without loss of generality set
 $\kappa_0\equiv b_0^{-1}$, where $b_0\ll1$ is the constant introduced in Section~\ref{MainAssumptions}.}

We introduce a Hilbert space $L^2_\mu\equiv L^2(\bS^1\times\bR;\Dz\,\mu\,\Dy)$,
with weighted measure $\mu(y):=(M+y^2)^{-\frac{3}{5}}$ defined with respect to a
constant $M\gg 1$ to be fixed below.\footnote{In \cite{GKS11}, $\mu$ and $M$ were
denoted by $\sigma$ and $\Sigma$, respectively.} We denote the $L^2_\mu$ inner product by
\[
    \lp\vp,\psi\rp :=
    \int_{\bR}\int_{\bS^1} \vp \psi\,\Dz\,\mu\,\Dy,
\]
and its norm by $\|\cdot\|_\mu$. The inner product in the complex Hilbert space $L^2(\bS^1)$ is
denoted by
\[
    \lp\vp,\psi\rs := \int_{\bS^1}\vp \psi\,\Dz,
\]
and its norm is $\|\cdot\|_{\bS^1}$. The undecorated inner product $\lp\cdot,\cdot\rangle$ denotes the
standard (unweighted) inner product in $L^2(\bS^1\times\bR)$.

We define a norm $\|\cdot\|_{m,n}$ by
\[
    \|\phi\|_{m,n} := \left\| (1+y^2)^{-\frac m2} \dy^n \phi\right\|_{L^\infty}.
\]
We write $\vp\ls\psi$ if there exists a uniform constant $C>0$
such that $\vp\leq C\psi$, and we set $\langle x \rangle := \sqrt{1+|x|^2}$. Finally, for any
$f(y,\theta,\tau)$, we define
\begin{equation}    \label{DefinePM}
    f_\pm(y,\tau) := \frac{1}{2\pi}\lp f(y,\theta,\tau), e^{\pm i\theta}\rs.
\end{equation}

\subsection{Main Assumptions}\label{MainAssumptions}
There are small positive constants $b_0,c_0$ such that:

\begin{itemize}
\item[{[A1]}] The initial surface is a graph over $\mathbb{S}^1\times\mathbb{R}$
determined by a smooth function $u_0(x,\theta)>0$. The function $u_0$ is uniformly
$(2b_0)$-Lipschitz and satisfies $u_0(x,\cdot)\geq c_* V_{a_0,b_0}(x)$ for some $c_*>0$,
along with the estimates\,\footnote{Smallness of $(u_0)_\pm$ and $(\partial_xu_0)_\pm$
in the inner region suffice for the first bootstrap machine, but global smallness is needed
for certain propagator estimates in the second bootstrap machine.}
\begin{align*}
    |(u_0)_\pm| + |(\dx u_0)_\pm| &< b_0^2,\\
    \|\langle x\rangle^{-5}\dz u_0\|_{L^\infty} &<b_0^{21/10},\\
    \|\langle x\rangle^{-11/10}(u_0)_\pm\|_{L^\infty}
    +\|\langle x\rangle^{-11/10}(\dx u_0)_\pm\|_{L^\infty}&<b_0^{53/20}.
\end{align*}

\item[{[A2]}] The initial function satisfies $u_0(x,\cdot) > g(x, b_0)$.

\item[{[A3]}] The initial surface is a small deformation of a formal solution $V_{a_0,b_0}(x)$
in the sense that for $(m,n)\in\left\{(3,0),\ (11/10,0),\ (2,1), \ (1,1)\right\}$, one has
\[
    \|u_0(\cdot)-V_{a_0,b_0}(\cdot)\|_{m,n}<b_0^{\frac{m+n}{2}+\frac{1}{10}}.
\]

\item[{[A4]}] The parameter $a_0=a(0)$ obeys the bound $|a_0-1/2|<c_0$.

\item[{[A5]}] The initial surface obeys the further pointwise derivative bounds
\begin{align*}
    \textstyle\sum_{n\neq0,\, 2\leq m+n \leq3} u_0^{-n}|\dx^m\dz^n u_0|&<b_0^2,\\
    b_0u_0^{-1/2}|\dx u_0| + b_0^{1/2}|\dx^2u_0|+|\dx^3 u_0|&<b_0^{\frac32},\\
    |\dx\dz^2 u_0|+u_0^{-1}|\dz^3 u_0|&<c_0.
\end{align*}

\item[{[A6]}] The initial surface obeys the Sobolev bounds
\[
    b_0^{4/5}\|\dx^4 u_0\|_\mu + \|\dx^5u_0\|_\mu
    + \sum_{n\neq0,\, 4\leq m+n \leq5} \|u_0^{-n}\dx^m\dz^n u_0\|_\mu < b_0^4.
\]

\item[{[A7]}] The initial surface satisfies
$\|u_0^{-n}\dx^m\dz^n u_0\|_{L^\infty}<\infty$ for $4 \leq m+n \leq 6$.
\end{itemize}

Three remarks are in order here: \textsc{(i)} Our choice $\lambda_0=1$ guarantees that $x=y$ and $\tau=0$
both hold at $t=0$, hence that the assumptions above apply identically to $u(x,\theta,0)=u_0(x,\theta)$
and $v(y,\theta,0)=v_0(y,\theta)$. \textsc{(ii)} The lower bounds for $u_0$ in Assumptions~[A1] and [A2]
are clearly related, with [A1] being stronger for large $|x|$, and [A2] being stronger for small $|x|$.
The bounds are presented in this way so that [A1] is the only member of [A1]--[A6] that is significantly changed
from our earlier work \cite{GKS11}, other than the slight modification of the function $g$ used in~[A2].
The new inequalities in [A1] compensate for the removal of the discrete symmetries that were used in \cite{GKS11};
we explain this further below.  \textsc{(iii)} Assumption~[A7], which is also new, imposes slightly more control at
spatial infinity than we needed in \cite{GKS11}. This allows us to integrate by parts in $y$ without needing the
hypothesis of reflection symmetry that we used in that earlier paper. (See Remark~\ref{IBP} below.)

\subsection{Strategy of the proof}\label{Strategy}
In earlier work \cite{GKS11} (also see \cite{GS09}) we decomposed a solution $v(y,\theta,\tau)$
of equation~\eqref{MCF-v} into two terms: a point $V_{a(\tau),b(\tau)}(y)$ on a manifold $\mathfrak M$ of
approximate solutions, and a remainder term $\phi(y,\theta,\tau):=v(y,\theta,\tau)-V_{a(\tau),b(\tau)}(y)$
approximately orthogonal to that manifold. These terms represent the large, slowly-changing part of
the solution, and the small, rapidly-decaying part, respectively. The utility of this decomposition is
that permutations of $v$ tangent to $\mathfrak M$ are linearly unstable, whereas those orthogonal to
$\mathfrak M$ are linearly stable. We follow the same strategy here, except that our lack of symmetry
assumptions requires us to construct a more refined decomposition in order to compensate for coordinate
instabilities of \eqref{MCF-v}. The details of this refinement are as follows.

In \cite{GKS11}, we imposed discrete symmetries on the initial surface: $u_0(-x,\cdot)=u_0(x,\cdot)$
and $u_0(\cdot,\theta+\pi)=u_0(\cdot,\theta)$. These assumptions fixed the cylindrical axis and center $y=0$
of the developing neckpinch singularity. In the present paper, we remove these assumptions, which means that we
must determine the unique cylinder and center at which the neckpinch forms. To compensate, Assumption~[A1]
is strengthened from~\cite{GKS11}. Its revised form ensures that a neck aligned to a slightly perturbed cylinder
can still be written as a normal graph, and provides extra estimates (to be improved in the second
bootstrap machine) for the $\theta$-dependence of the solution. These estimates effectively replace the lost
symmetries, which implied inequalities (like $\|\dz^2v\|_{\bS^1}\geq 2\|\dz v\|_{\bS^1}$) that we cannot use here.

\medskip

We determine the unique cylinder and center of the neckpinch as follows.
First, at some fixed time $\tau_1$, we construct ``optimal coordinates'' for the developing solution.
As motivation, we observe that linearizing equation~\eqref{MCF-v} at the cylinder of radius
$a^{-1/2}$ leads to the elliptic operator
\[
    \mL = \mH - a\Delta \equiv -(\dy^2+ay\dy+2a) - a\Delta,
\]
where $\Delta$ is the Laplacian of $\bS^1$. The operator $-\mH=\dy^2+ay\dy+2a$ is
commonly seen in geometric evolution equations. It is self-adjoint in the weighted
Hilbert space $L^2(\bR,e^{-a y^2/2}\Dy)$ with spectrum $\spec(\mH)=\{a(j-2)\}_{j\geq0}$.
Its eigenfunctions are the Hermite polynomials $h_{a,j}$; in particular, its (weakly) unstable
eigenspaces are spanned by $h_{a,0}=1$, $h_{a,1}=y$, and $h_{a,2}=y^2-a^{-1}$. The spectrum
$\{k^2\}_{k\geq0}$ of $-\Delta$ is also well known, and its lowest nontrivial eigenspace is spanned
by the restriction to $\bS^1$ of the coordinate functions $x_1\big|_{\bS^1}=\cos\theta$ and
$x_2\big|_{\bS^1}=\sin\theta$. Hence $\spec(\mL)=\{a(j-2)+ak^2;\; j\geq0,\,k\geq0\}$, and we can
summarize the (weakly) unstable eigenvalues of $\mL$ and their associated eigenfunctions in this table:
\[
    \begin{tabular}[c]{cccc}
    \emph{Eigenvalue} & \emph{Multiplicity} & \emph{Eigenfunction(s)} & \emph{Geometric meaning}\\
    $-2a$ & $1$ & $1$ & rescale cylinder\\
    $-a$ & $3$ & $y=\lambda^{-1}x_0,\,x_1,x_2$ & translate center of neck\\
    $0$ & $3$ & $y^{2}-a^{-1};\;yx_1,yx_2$ & shape neck;\;\; tilt cylinder
    \end{tabular}
\]
In optimal coordinates, a surface is orthogonal to the span of these functions.

\begin{definition}	\label{OptimalCoordinates}
We say a cylindrical coordinate system $(x,r,\theta)$ is an \emph{optimal system} at time $\tau_1$ with
respect to positive parameters $\lambda_\opt\equiv\lambda_\opt(\tau_1)$ and
$b_\opt\equiv b_\opt(\tau_1)$ if
\[
    \pop(y,\theta,\tau_1):=v(y,\theta,\tau_1)-V_{b_\opt(\tau_1)/2,\;b_\opt(\tau_1)}
    \equiv v(y,\theta,\tau_1)-\sqrt{\frac{2+b_\opt(\tau_1)y^2}{1+\frac12 b_\opt(\tau_1)}}
\]
satisfies the orthogonality condition
\[
    \pop\perp\{1,y,y^2-a_1^{-1},\cos\theta,\sin\theta,y\cos\theta,y\sin\theta\}
\]
in $L^2(\bR\times\bS^1,\mathrm{d}\theta\,e^{-a_1y^2/2}\Dy)$, where $v=\lambda_\opt^{-1}u$,
$y=\lambda_\opt^{-1}\,x$, and $a_1:=\frac12-\frac14 b_\opt$.
\end{definition}

Having constructed optimal coordinates at time $\tau_1$, we then observe that by a
straightforward extension of the implicit function theorem argument developed in \cite{GS09}, 
there exists a (possibly small) earlier time interval $[\tau_0,\tau_1]$ and $C^1$ functions
$\lambda(\tau)$, $a(\tau)$, $b(\tau)$, and $\beta_0(\tau),\dots,\beta_4(\tau)$ defined for
$\tau\in[\tau_0,\tau_1]$, such that the quantity
\begin{equation}	\label{tilde-phi}
    \tilde\phi(y,\theta,\tau):=v(y,\theta,\tau)-V_{3/2-2a(\tau),b(\tau)}
    \equiv v(y,\theta,\tau)-\sqrt{\frac{2+b(\tau)y^2}{2-2a(\tau)}}
\end{equation}
can be parameterized in the form
\begin{equation}	\label{Whereisphi}
    \tilde\phi(y,\theta,\tau)=
     \beta_0y+\beta_1\cos\theta+\beta_2\sin\theta
    +\beta_3y\cos\theta+\beta_4y\sin\theta+\phi(y,\theta,\tau),
\end{equation}
where the orthogonality conditions
\[
    \phi\perp\{1,y,y^2-a^{-1},\cos\theta,\sin\theta,y\cos\theta,y\sin\theta\}
\]
hold in $L^2(\bR\times\bS^1,\mathrm{d}\theta\,e^{-a(\tau)y^2/2}\Dy)$ for $\tau\in[\tau_0,\tau_1]$,
with the boundary conditions $\lambda(\tau_1)=\lambda_\opt(\tau_1)$,
$a(\tau_1)=a_1$ (recall that $a=-\lambda\partial_t\lambda$), and $b(\tau_1)=b_\opt(\tau_1)$,
and with the new ``translate'' and ``tilt'' parameters satisfying the boundary conditions
\begin{equation}\label{eq:zeroEnd}
\beta_{0}(\tau_1)=\beta_{1}(\tau_1)=\beta_{2}(\tau_1)=\beta_3(\tau_1)=\beta_4(\tau_1)=0.
\end{equation}

The next step of the argument, accomplished in the second bootstrap machine, is to obtain sufficiently good
control on the parameters $\beta_0,\dots,\beta_4$ to conclude that the sequence of optimal coordinate
systems constructed at discrete times $\tau(t_1)<\tau(t_2)<\cdots<T$ converges. We accomplish this
by proving that the difference between an optimal system at time $t_1$ and an optimal system at a
later time $t_2$ is of order $b^2(t_1)$, where $b(t)=\{1+o(1)\}(-\log(T-t))^{-1}\rightarrow0$ as
$t\nearrow T$.

This concludes our heuristic outline of the proof. We now provide the details needed to make everything rigorous.

\section{How the solution evolves}	\label{Derived}
Given integers $m,n\geq0$ and a real number $k\geq0$, we define
\begin{equation} \label{Define-vmnk}
    v_{m,n,k}:= v^{-k}(\dy^m \dz^n v).
\end{equation}
The following result is an immediate consequence of Lemma~3.1 and Corollary~3.3 from
our earlier work in \cite{GKS11}:

\begin{lemma}
The quantity $v_{m,n,k}^2$ evolves by
\begin{equation} \label{Evolve-vmnk2}
    \begin{array}{lrl}
    \dt (v_{m,n,k}^2) &= &A_v (v_{m,n,k}^2) + 2\left[(k+1)v^{-2}-(m+k-1)a\right]v_{m,n,k}^2\\ \\
                    &  &- B_{m,n,k} + 2E_{m,n,k}v_{m,n,k},
    \end{array}
\end{equation}
where $B_{m,n,k}:= A_v(v_{m,n,k}^2)-2v_{m,n,k} A_v v_{m,n,k}$ satisfies the inequality
\begin{equation}\label{BisGood}
    B_{m,n,k} \geq2\frac{(\dy v_{m,n,k})^2+(v^{-1}\dz v_{m,n,k})^2}{1+p^2+q^2}\geq 0,
\end{equation}
and where the commutator terms $E_{m,n,k}=\sum_{\ell=0}^5 E_{m,n,k,\ell}$ are given by
\begin{equation} \label{Define-Emnkj}
    \begin{array} {lll}
    E_{m,n,k,0} & := &-kv^{-1}v_{m,n,k}(A_v v+ay\dy v),\\ \\
    E_{m,n,k,1} & := & v^{-k}\dy^m\dz^n(F_1\dy^2v)-F_1\dy^2v_{m,n,k},\\ \\
    E_{m,n,k,2} & := & v^{-k}\dy^m\dz^n(v^{-2}F_2\dz^2v)-v^{-2}F_2\dz^2v_{m,n,k},\\ \\
    E_{m,n,k,3} & := & v^{-k}\dy^m\dz^n(v^{-1}F_3\dy\dz v)-v^{-1}F_3\dy\dz v_{m,n,k},\\ \\
    E_{m,n,k,4} & := & v^{-k}\dy^m\dz^n(v^{-2}F_4\dz v)-v^{-2}F_4\dz v_{m,n,k},\\ \\
    E_{m,n,k,5} & := & -v^{-k}\dy^m\dz^n(v^{-1})-v^{-2}v_{m,n,k}.
    \end{array}
\end{equation}
\end{lemma}

The utility of estimate~\eqref{BisGood} is that once one has suitable first-order estimates
for $v$, one can bound $1+p^2+q^2$ from above, whereupon $B_{m,n,k}$ contributes useful
higher-order terms of the form $-\ve(v_{m+1,n,k}^2 + v_{m,n+1,k+1}^2)$ to the evolution
equation~\eqref{Evolve-vmnk2} satisfied by $v_{m,n,k}^2$.

In the bootstrapping arguments made in this paper, we must estimate the nonlinear ``error terms''
$E_{m,n,k,\ell}$ defined in display~\eqref{Define-Emnkj}. In Section~3 of \cite{GKS11}, we made
the following observations, which we freely use here in estimating these quantities.

\begin{remark}
For any $i,j\geq 0$ and $\ell=1,\dots,4$, there exist constants $C_{i,j,\ell}$ such that
\[
    |\partial_p^i\partial_q^j F_{\ell}(p,q)|\leq C_{i,j,\ell}
\]
for all $p,q\in\bR$.
Moreover, $E_{m,n,k,0}=0$ if $k=0$, and $E_{m,n,k,5}=0$ if $m+n=1$.
\end{remark}

\section{The first bootstrap machine}\label{FirstBootstrap}

\subsection{Inputs}\label{FBMI}
We now present the inputs to our first bootstrap machine, whose structure we describe below.
By standard regularity theory for quasilinear parabolic equations, if the initial data satisfy
the Main Assumptions in Section~\ref{MainAssumptions} for sufficiently small $b_0$ and $c_0$,
then solutions originating from such data will satisfy the properties below for a sufficiently
short time interval $[0,\tau_1]$.\footnote{In order to obtain some of the derivative bounds here, we use
a general interpolation result, Lemma~B.2 from \cite{GKS11}. Note that Lemmas~B.1 and B.2 in \cite{GKS11}
use only the $2\pi$-periodicity of $v(\cdot,\theta,\cdot)$, hence apply here as well.}

Most of these properties are global, while others are local in nature. In many of the arguments
that follow, we separately treat the \emph{inner region} $\{\beta y^2\leq20\}$ and the \emph{outer region}
$\{\beta y^2\geq20\}$, where $\beta(\tau):=(\kappa_0+\tau)^{-1}$. Note that the \emph{inner region} here
corresponds to the union of the \emph{parabolic} and \emph{intermediate} regions in \cite{AK07}.
Note also that our Main Assumptions imply that slightly stronger conditions hold for the inner region
than for the outer. This is natural, because it is in the inner region that one expects the
solution $v$ of \eqref{MCF-v} to be closest to the formal solution $V_{1/2,\beta}$ of the
equation $\frac12y\dy V-\frac12V+V^{-1}=0$ that is the adiabatic approximation of \eqref{MCF-v}.
\medskip

To state the global conditions, we decompose the solution into $\theta$-independent and
$\theta$-dependent parts $v_1,v_2$, respectively, defined by
\begin{equation}\label{Decomposition}
    v_1(y,\tau):=\frac{1}{2\pi}\int_0^{2\pi}v(y,\theta,\tau)\,\mathrm{d}\theta
    \quad\text{and}\quad
    v_2(y,\theta,\tau):=v(y,\theta,\tau)-v_1(y,\tau).
\end{equation}

Here are the global conditions:
\begin{itemize}
\item[{[C0]}] For $\tau\in[0,\tau_1]$, the solution has a uniform lower bound $v(\cdot,\cdot,\tau)\geq\kappa_0^{-1}$.

\item[{[C1]}] For $\tau\in[0,\tau_1]$, the solution satisfies the first-order estimates
\[
    |\dy v| \ls \beta^{\frac{2}{5}}v^{\frac12},\quad
    |\dz v| \ls \beta^{\frac{3}{2}}v^2,\quad\text{and}\quad
    |\dz v| \ls v.
\]

\item[{[C2]}] For $\tau\in[0,\tau_1]$, the solution satisfies the second-order estimates
\[
    |\dy^2 v|  \ls \beta^{\frac{3}{5}},\quad
    |\dy\dz v| \ls \beta^{\frac{3}{2}}v,\quad
    |\dy\dz v| \ls 1,\quad\text{and}\quad
    |\dz^2 v|  \ls \beta^{\frac{3}{2}}v^2.
\]

\item[{[C3]}] For $\tau\in[0,\tau_1]$ the solution satisfies the third-order decay estimates
\[
    |\dy^3 v|\ls\beta\quad\text{and}\quad
    v^{-n}|\dy^m\dz^n v|\ls\beta^{\frac{3}{2}}
\]
for $m+n=3$ with $n\geq 1$, as well as a ``smallness estimate'' that
\[
    \beta^{-\frac{11}{20}}\left(|\dy^3 v|+|\dy^2\dz v|\right)
    +|\dy\dz^2 v|+v^{-1}|\dz^3 v|\ls(\beta_0+\ve_0)^{\frac{1}{40}}
\]
holds for some $\ve_0=\ve_0(b_0,c_0)\ll1$.\footnote{The smallness estimate is only
used in the proof of Theorem~\ref{SmallnessEstimates}.}

\item[{[Ca]}] For $\tau\in[0,\tau_1]$, the parameter $a$ satisfies
%% \,\footnote{Proposition~\ref{IFT} will establish that $a$ is a $C^1$ function of time.}
\[
    \frac12-\kappa_0^{-1}\leq a \leq \frac12+\kappa_0^{-1}.
\]

\item[{[Cg]}] For $\tau\in[0,\tau_1]$, one has
$\lan y \ran^{-1}|\dy v|\leq M^{\frac{1}{4}}\beta$, where $M\gg1$ is the constant
introduced in Section~\ref{Notation}.

\item[{[Cr]}] There exists $0<\delta\ll 1$ such that the scale-invariant
bound $|v_2|\leq\delta v_1$ holds everywhere for $\tau\in[0,\tau_1]$.

\item[{[Cs]}] For $\tau\in[0,\tau_1]$, one has Sobolev bounds $\|v^{-n}\dy^m\dz^n v\|_{\mu}<\infty$
whenever $4\leq m+n\leq 7$.

\item[{[Ct]}] For $\tau\in[0,\tau_1]$, the solution is ``tame at infinity'' in the sense that for
$4\leq m+n\leq 7$, one has $\|v^{-n}\dy^m\dz^n v\|_{L^\infty}<\infty$.
\end{itemize}

Note that the global gradient Condition~[Cg] posits the $\beta$ decay rate
of the formal solution $V(y,\tau)=\sqrt{2+\beta y^2}$, but with a large constant
$M^{\frac{1}{4}}\gg1$. The effect of our bootstrap argument will be to
sharpen that constant. Remark~\ref{NoCircleHere} (below) shows that Condition~[Cb]
in Section~\ref{SecondBootstrap}, which is directly implied by our Main Assumptions in
Section~\ref{MainAssumptions}, in turn implies that the global gradient bound [Cg] holds,
along with extra properties that are local to the inner region.

\medskip
The extra conditions for the inner region are as follows.
\begin{itemize}
\item[{[C0i]}] For $\tau\in[0,\tau_1]$ and $\beta y^2\leq20$, the quantity $v$
is uniformly bounded from above and below: there exists $C_*$ such that
\[
    \frac{19}{20}\sqrt 2\equiv g(y,\beta)\leq v(y,\cdot,\cdot)\leq C_*.
\]

\item[{[C1i]}] For $\tau\in[0,\tau_1]$ and $\beta y^2\leq20$, the solution satisfies
the stronger first-order estimates
\[
    |\dy v| \ls \beta^{\frac12}v^{\frac12}\quad\text{and}\quad
    |\dz v|\ls\kappa_0^{-\frac12}v.
\]
\item[{[C2i]}] For $\tau\in[0,\tau_1]$ and $\beta y^2\leq20$, the $\theta$-dependent
part\,\footnote{Recall that $v_\pm$ and $(\dy v)_\pm$ are defined by \eqref{DefinePM}.}
of the solution satisfies
\[
    |v_\pm| + |(\dy v)_\pm| \ls \beta^2.
\]
\end{itemize}

\medskip
The only differences above from \cite{GKS11} are Condition~[Ct], which follows from the
new Assumption~[A7], and [C2i], which follows from the revised [A1]. As explained in
Section~\ref{Strategy} and Remark~\ref{IBP} (below), [Ct] and [C2i] allow us to compensate
for the discrete symmetries assumed in \cite{GKS11} that may not hold here. We make the
consequences of these changes clear in the proofs that follow.

\subsection{Outputs}\label{FirstOutput}
The output of this machine consist of the following estimates, which
collectively improve Conditions~[C0]--[C3], [Cs], [Cr], [Cg], and [C0i]--[C1i]:
\[\begin{array}{c}
    v(y,\theta,\tau)\geq g(y,\beta),\quad
    v=\mathcal{O}(\lan y \ran)\,\text{ as }\, |y|\rightarrow\infty;\\ \\
    v^{-\frac12}|\dy v|\ls \beta^{\frac12},\quad
    |\dy v|\leq\ve_0,\quad
    v^{-2}|\dz v|\ls\beta^{\frac{33}{20}},\quad
    v^{-1}|\dz v|\ls\kappa_0^{-\frac12};\\ \\
    \beta|\dy^2 v|+v^{-1}|\dy\dz v|+v^{-2}|\dz^2 v|\ls\beta^{\frac{33}{20}};\\ \\
    |\dy\dz v|+v^{-1}|\dz^2 v|\ls(\beta_0+\ve_0)^{\frac{1}{20}};\\ \\
    \beta^{\frac12}|\dy^3 v|+v^{-1}|\dy^2\dz v|+v^{-2}|\dy\dz^2 v|+v^{-3}|\dz^3v|
    \ls\beta^{\frac{33}{20}};\\ \\
    \beta^{-\frac{11}{20}}\left(|\dy^3 v|+|\dy^2\dz v|\right)
    +|\dy\dz^2 v|+v^{-1}|\dz^3 v|\ls(\beta_0+\ve_0)^{\frac{1}{20}}.
\end{array}\]
Here, $\beta_0\equiv\beta(0)$ and $\ve_0$ are independent of $\tau_1$. Therefore, these
improvements allow us to propagate the assumptions above forward in time in our proof of
the Main Theorem.

\subsection{Structure}

In the first step of the bootstrap argument, we derive improved estimates for $v$ and its first
derivatives in the outer region.

\begin{theorem}\label{FirstOrderEstimates}
Suppose a solution $v=v(y,\theta,\tau)$ of \eqref{MCF-v} satisfies Assumption~[A1] at $\tau=0$.
If for $\tau\in[0,\tau_1]$, Conditions~[Ca] and [C0]--[C2] hold in the outer region $\{\beta y^2\geq20\}$,
and Conditions~[C0i]--[C1i] hold on its boundary, then throughout the outer region $\{\beta y^2\geq20\}$,
one has estimates
\begin{equation}\label{v-below}
    v(y,\cdot,\cdot)\geq 4,
\end{equation}
\begin{equation}\label{v_y-est}
    |\dy v|\ls \beta^{\frac12}v^{\frac12},
\end{equation}
\begin{equation}\label{v_y-bound}
    |\dy v|\leq\ve_0,
\end{equation}
\begin{equation}\label{v_z-est}
    |\dz v|\leq C_0\kappa_0^{-\frac12} v,
\end{equation}
for all $\tau\in[0,\tau_1]$, where $\ve_0$ and $C_0$ depend only on the initial data and not on $\tau_1$.
\end{theorem}

\begin{proof}
Except for estimate~\eqref{v_y-bound}, the statements and their proofs are identical to those
stated in Theorem~4.2 and proved in Section~5 of \cite{GKS11}. Indeed, the reader may check that
the proofs there used neither the fact that $v(y,\cdot,\cdot)$ was assumed to be an even function
of $y$, nor that $v(y,\cdot,\cdot)$ was assumed to be orthogonal to the lowest two eigenspaces
$\mathrm{span}\{1,\sin\theta,\cos\theta\}\subset L^2(\{y\}\times\bS^1)$ of the Laplacian on $\bS^1$.

We now show how to improve the bound $|\dy v|\ls1$ proved in in Section~5 of \cite{GKS11} to the bound
$|\dy v|\leq\ve_0$ needed here. In our earlier work, we showed that $1+(\dy v)^4$ was bounded above for
all time by the solution of the \textsc{ivp}
\begin{align*}
    \Dt\vp&=C\beta^{\frac65}\vp,\\
    \vp(0)&=\sup_{\tau=0}\{1+(\dy v)^4\}.
\end{align*}
Noting that
\[
    \frac{\lim_{\tau\rightarrow\infty}\vp(\tau)}{\vp(0)}=\exp\left(\frac{5C}{\kappa_0^{1/5}}\right),
\]
where $C$ does not increase as we increase $\kappa_0$, we obtain the result by choosing $\kappa_0$
sufficiently large and $b_0$, hence $\vp(0)$, sufficiently close to $1$. 
\end{proof}

Note that estimate~\eqref{v_y-bound} provides $\ve_0$ such that
$|\dy v|\leq\ve_0$ in the outer region for $\tau\in[0,\tau_1]$. By
Condition~[C0i], one has $v(\pm\sqrt{20}\beta^{-\frac12})\leq c^2$
for some $c>0$. Hence by quadrature, one obtains the immediate corollary that
\begin{equation}\label{v-above}
    v=\mathcal{O}(\lan y \ran)\quad\text{as}\quad |y|\rightarrow\infty.
\end{equation}
\medskip

In the second step, we derive estimates for second and third derivatives of $v$.
In the outer region, these estimates are obtained by maximum principle arguments,
exactly as in \cite{GKS11}, using the noncompact maximum principle derived in
Appendix~C of that paper. In the inner region, we introduce and bound suitable
Lyapunov functionals and then employ Sobolev embedding. These arguments require
integration by parts.

\begin{remark}  \label{IBP}
For $1\leq m+n\leq7$, Condition~[Cs] implies that $\|v_{m,n,n}\|_{L^\infty}<\infty$ for
$\tau\in[0,\tau_1]$. Using this and \eqref{v-above}, it is straightforward to verify that
we may integrate by parts in $y$ as was done in \cite{GKS11} --- without using reflection
symmetry. Two special but important examples are the facts that for $0\leq r+s,\,m+n \leq5$,
one has
\[
    \lp\dy^2 v_{m,n,n},v_{r,s,s}\rp
    =-\lp v_{m+1,n,n},\,v_{r+1,s,s}+v_{r,s,s}\dy(\log\mu)\rp;
\]
and for $1\leq m+n\leq5$, one has
\begin{align*}
    -\lp y\dy v_{m,n,n},v_{m,n,n}\rp
    &= \frac12\|v_{m,n,n}\|_\mu^2+\frac12\lp v_{m,n,n},\, y\dy(\log\mu) v_{m,n,n}\rp\\
    &\leq\frac12\|v_{m,n,n}\|_\mu^2,
\end{align*}
because $y\dy(\log\mu)\leq0$.
\end{remark}

\begin{theorem}\label{InnerEstimates}
Suppose that a solution $v=v(y,\theta,\tau)$ of equation~\eqref{MCF-v}
satisfies Assumption~[A1] at $\tau=0$, and Conditions~[Ca], [C0]--[C3],
[Cs], [Cr], [Cg], and [C0i]--[C2i] for $\tau\in [0,\tau_1]$. Then for
the same time interval, the solution satisfies the following pointwise
bounds throughout the inner region $\{\beta y^2\leq20\}$:
\begin{equation}\label{InnerSecondOrder}
    \beta|\dy^2 v|+v^{-1}|\dy\dz v|+v^{-2}|\dz^2 v|\ls\beta^{\frac{33}{20}}
\end{equation}
and
\begin{equation}\label{InnerThirdOrder}
    \beta^{\frac12}|\dy^3 v|+v^{-1}|\dy^2\dz v|+v^{-2}|\dy\dz^2 v|
    +v^{-3}|\dz^3 v|\ls\beta^{\frac{33}{20}},
\end{equation}
\end{theorem}
The fact that we do not achieve the expected $\beta^2$ decay on the
\textsc{rhs} of estimates~\eqref{InnerSecondOrder}--\eqref{InnerThirdOrder}
is due to our use of Sobolev embedding with respect to the weighted measure
for $L^2_\mu$ introduced in Section~\ref{Notation}.

We prove Theorem~\ref{InnerEstimates} in Section~\ref{InnerProof}. This task requires
us to modify the arguments used in \cite{GKS11}: those arguments use evenness of
$v(y,\cdot,\cdot)$ and $\pi$-periodicity of $v(\cdot,\theta,\cdot)$, neither of which
are assumed here.

Using Theorem~\ref{InnerEstimates} to ensure that they hold on the boundary
of the inner region, we then extend its estimates to the outer region.

\begin{theorem}\label{OuterEstimates}
Suppose that a solution $v=v(y,\theta,\tau)$ of equation~\eqref{MCF-v}
satisfies Assumption~[A1] at $\tau=0$, and Conditions~[Ca] and [C0]--[C3]
for $\beta y^2\geq20$ and $\tau\in [0,\tau_1]$. Then the estimates
\begin{equation}\label{OuterSecondOrder}
    \beta|\dy^2 v|+v^{-1}|\dy\dz v|+v^{-2}|\dz^2 v|\ls\beta^{\frac{33}{20}}
\end{equation}
and
\begin{equation}\label{OuterThirdOrder}
\beta^{\frac12}|\dy^3 v|+v^{-1}|\dy^2\dz v|+v^{-2}|\dy\dz^2 v|+v^{-3}|\dz^3v|
    \ls\beta^{\frac{33}{20}}
\end{equation}
hold throughout the outer region $\{\beta y^2\geq20\}$ for the same time interval,
provided that they hold on the boundary $\{\beta y^2=20\}$.
\end{theorem}

As an easy corollary, we apply general interpolation results (c.f.~\cite[Lemma~B.2]{GKS11})
to get our final first-order estimate claimed in Section~\ref{FirstOutput}, namely
\begin{align*}
\frac{|\dz v|}{v^2}
    \leq \frac{1}{(1-\delta)^2}\frac{|\dz v_2|}{v_1^2}
    &\leq \frac{C_1}{(1-\delta)^2}\max_{\theta\in[0,2\pi]}\frac{|\dz^2 v_2|}{v_1^2}\\
    &\leq C_2\frac{(1+\delta)^2}{(1-\delta)^2}\max_{\theta\in[0,2\pi]}\frac{|\dz^2 v|}{v^2}
    \leq C_3 \beta^{\frac{33}{20}}.
\end{align*}

The proof of Theorem~\ref{OuterEstimates} is identical to that of Theorem~4.6 in Sections~7.1--7.2
of \cite{GKS11}, where neither evenness in $y$ nor $\pi$-periodicity in $\theta$ are used.

\medskip

In the final step, we improve the ``smallness estimates'' in Condition~[C3], producing improved
bounds for $|\dy\dz^2 v|$ and $v^{-1}|\dz^3 v|$ that serve as inputs to the second bootstrap machine.

\begin{theorem}\label{SmallnessEstimates}
Suppose that a solution $v=v(y,\theta,\tau)$ of equation~\eqref{MCF-v}
satisfies Assumption~[A1] at $\tau=0$, and Conditions [Ca], [C0]--[C3],
[Cs], [Cr], [Cg], and [C0i]--[C1i] for $\tau\in [0,\tau_1]$. Then for
the same time interval, the solution satisfies
\[
    \beta^{-\frac{11}{10}}\left[(\dy^3 v)^2+(\dy^2\dz v)^2\right]
    +(\dy\dz^2 v)^2+v^{-2}(\dz^3 v)^2
    \ls(\beta_0+\ve_0)^{\frac{1}{10}}.\]
\end{theorem}
As an easy corollary, we again apply general interpolation results to get our final second-order
estimates claimed in Section~\ref{FirstOutput}, namely
\[
    |\dy\dz v|+v^{-1}|\dz^2 v|\ls(\beta_0+\ve_0)^{\frac{1}{20}}.
\]

Once Theorems~\ref{InnerEstimates} and \ref{OuterEstimates} are established, the proof of
Theorem~\ref{SmallnessEstimates} is identical to that of Theorem~4.7 in Section~7.3 of \cite{GKS11}.

Collectively, Theorems~\ref{FirstOrderEstimates}--\ref{SmallnessEstimates} complete the
first bootstrap machine.

\section{Improved estimates for the inner region}\label{InnerProof}

In this section, we prove Theorem~\ref{InnerEstimates}.

Given integers $m,n\geq0$, we define Lyapunov functionals
$\Omega_{m,n}=\Omega_{m,n}(\tau)$ by
\begin{equation}
\Omega_{m,n}
    := \|v_{m,n,n}\|_\mu^2
    \equiv \int_{-\infty}^{\infty}\int_{0}^{2\pi}
        v^{-2n}(\dy^m\dz^n v)^2\,\Dz\,\mu\,\Dy.
\end{equation}
Note that $v_{m,n,k}$ is defined in equation~\eqref{Define-vmnk}, and
$\mu=\mu(y)$ is defined in Section~\ref{Notation}.
By Conditions~[C2], [C3], and [Cs], these functionals and their
$\tau$-derivatives are well defined if $m+n\leq5$.

We prove the second- and third-order derivative bounds in Theorem~\ref{InnerEstimates} in
two steps. \textsc{(i)} We bound certain weighted sums of $\Omega_{m,n}$ with $2\leq m+n\leq5$.
Here, as in \cite{GKS11}, it suffices to show that derivatives of orders three through five decay
at the same rates as the estimates we derive for those of second order, rather than at the faster
rates one would expect from parabolic smoothing. This somewhat reduces the work involved.
\textsc{(ii)} We apply Sobolev embedding to get pointwise estimates, using the facts that
$|y|\ls\beta^{-\frac12}$ in the inner region, and that $\mu\sim |y|^{-\frac{6}{5}}$ as
$|y|\rightarrow\infty$.

Our lack of discrete symmetry assumptions, in contrast to \cite{GKS11}, requires us to modify
some of the proofs from our earlier work. Here, we employ three kinds of argument: one for
functionals with only $y$-derivatives, one for functionals with ``many'' $y$-derivatives,
and one for functionals with ``few'' $y$-derivatives. These categories are made precise below.

\medskip

As immediate consequences of the proofs of Lemmas~6.1--6.3 in \cite{GKS11}, we find that the
following differential inequalities for the second-order functionals hold pointwise:

\begin{lemma}   \label{PointwiseSecondOrderBounds}
There exists $0<C<\infty$ such that:

\textsc{(i)} the quantity ${\V 20}^2=(\dy^2 v)^2$ satisfies
\begin{align*}
\frac12\dt {\V 20}^2
    \leq&{\V 20}A_v{\V 20}+(v^{-2}-a){\V 20}^2\\
    &+C\Big(\beta|\V 20|+\beta^{\frac{11}{10}}\big\{|\V 30|+|\V 21|+|\V 12|\big\}\Big);
\end{align*}
%% where $v^{-2}-a\leq\frac{1}{16}-a$ in the outer region $\{\beta y^2>20\}$;

\textsc{(ii)} the quantity ${\V 11}^2=v^{-2}(\dy\dz v)^2$ satisfies
\begin{align*}
    \frac12\dt{\V 11}^2
    \leq&{\V 11}A_v{\V 11}+(2v^{-2}-a){\V 11}^2\\
    &+C\Big(\beta^2\big\{|\V 11|+|\V 21|+|\V 12|+|\V 03|\big\}+\beta^3|\V 30|\Big);
\end{align*}
%% where $2v^{-2}-a\leq\frac{1}{8}-a$ in the outer region $\{\beta y^2>20\}$; and

\textsc{(iii)} the quantity $(v_{0,2,2})^2=v^{-4}(\dz^2 v)^2$ satisfies
\begin{align*}
    \frac12\dt{\V 02}^2
    \leq&{\V 02}A_v{\V 02}+(3v^{-2}-a){\V 02}^2\\
    &+C\Big(\beta^2\big\{|\V 12|+|\V 03|\big\}+\beta^{\frac{21}{10}}|\V02|+\beta^3|\V 21|\Big).
\end{align*}
%% where $3v^{-2}-a\leq\frac{3}{16}-a$ in the outer region $\{\beta y^2>20\}$.
\end{lemma}

We also import various pointwise estimates for derivatives of the coefficients $F_\ell$ of the quasilinear
elliptic operator $A_v$ defined in~\eqref{Define-A}--\eqref{Define-Fi}. They appear as estimates~(6.2)--(6.6)
in \cite{GKS11}, and follow from Theorem~\ref{FirstOrderEstimates}, Condition~[C2], and Conditions~[C0i] and
[C1i] for the inner region. We state them here as:

\begin{lemma}   \label{DifferentiateFi}
First-order derivatives of the coefficients $F_\ell$ obey the bounds
\[
    |\dy F_\ell|\ls\beta^{\frac35}\qquad\text{and}\qquad
    v^{-1}|\dz F_\ell|\ls\beta^{\frac32}.
\]
Second-order derivatives of the coefficients $F_\ell$ obey the bounds
\begin{align*}
    |\dy^2 F_\ell|  &\ls\beta^{\frac65}+|\V 30|+|\V 21|,\\
    v^{-1}|\dy\dz F_\ell|   &\ls\beta^2+|\V 21|+|\V 12|,\\
    v^{-2}|\dz^2 F_\ell|    &\ls\beta^3+|\V 12|+|v_{3,3,0}|.
\end{align*}
\end{lemma}

In practice, we use these estimates in the following forms, that follow
immediately from combining Lemma~\ref{DifferentiateFi} with Conditions~[C1]--[C3],
along with the fact that $\|\mu^{-1}\dy^2\mu\|_{L^\infty}=\cO(M^{-1})$ as
$M\rightarrow\infty$. The statement is:

\begin{lemma}   \label{ControlCoefficients}
One has the pointwise estimates
\begin{align*}
    \mu^{-1}|\dy^2(\mu F_1)|&\ls|\V 30|+|\V 21|+\beta^{\frac65}
        +M^{-\frac12}\beta^{\frac35}+M^{-1}\ls\beta+M^{-1},\\
    |\dz^2(v^{-2}F_2)|&\ls|v_{3,3,0}|+|\V 12|+\beta^{\frac32}\ls\beta^{\frac32},\\
    \mu^{-1}|\dy\dz(v^{-1}\mu F_3)|&\ls|\V 21|+|\V 12|+(1+M^{-\frac12})\beta^{\frac32}
        \ls(1+M^{-\frac12})\beta^{\frac32},\\
    |\dz(v^{-2}F_4)|&\ls\beta^{\frac32}.
\end{align*}
\end{lemma}

\subsection{Estimating the time evolution of the Lyapunov functionals}
In this section, we prove estimates for the evolution equations satisfied by $\Omega_{m,n}$.
The proofs and results are similar but not identical to those in Lemmas~6.5--6.9 and 6.11--6.12
of \cite{GKS11}. The arguments here come in three flavors: one for the case that there are no
$\theta$ derivatives ($\Omega_{m,0}$ with $2\leq m\leq5$), another for the case that there are
``many'' $\theta$ derivatives ($\Omega_{m,n}$ with $m\in\{0,1\}$ and $m+n\in\{3,4,5\}$), and a
third for the remaining cases (($\Omega_{m,n}$ with $(m,n)\in\{(2,1),(3,1),(2,2),(4,1),(3,2),(2,3)\}$).

\begin{lemma}   \label{Omega-all-y}
There exist constants $0<\ve<C<\infty$ and $\rho=\rho(M)>0$,
with $\rho(M)\searrow0$ as $M\rightarrow\infty$, such that
\begin{align*}
    \Dt\Omega_{2,0}&\leq
    -\ve\left(\Omega_{2,0}+\Omega_{3,0}+\Omega_{2,1}\right)
    +\rho\left\{\beta\Omega_{2,0}^{\frac12}
    +\beta^{\frac12}\Omega_{3,0}^{\frac12}
    +\beta^{\frac{11}{10}}
        \left(\Omega_{2,1}^{\frac12}+\Omega_{1,2}^{\frac12}\right)\right\},\\ \\
    \Dt\Omega_{3,0}&\leq
    -\ve\left(\Omega_{3,0}+\Omega_{4,0}+\Omega_{3,1}\right)
    +\rho\beta^{\frac{3}{2}}\left\{
    \Omega_{3,0}^{\frac12}
    +\beta^{\frac{1}{10}}\Omega_{4,0}^{\frac12}
    +\Omega_{3,1}^{\frac12}
    +\Omega_{2,2}^{\frac12}\right\}\\
    &\qquad+C\beta\left\{\Omega_{3,0}+\Omega_{2,1}+\Omega_{1,2}+\Omega_{0,3}\right\}
    +C\beta^{\frac12}\Omega_{3,0}^{\frac12}\Omega_{2,0}^{\frac12},\\ \\
    \Dt\Omega_{4,0}&\leq
    -\ve\left(\Omega_{4,0}+\Omega_{5,0}+\Omega_{4,1}\right)
    +C\beta^{\frac12}
        \Big(\sum_{4\leq i+j\leq 5}\Omega_{i,j}\Big)
    +C\beta^{\frac85}\Big(\sum_{4\leq i+j\leq 5}\Omega_{i,j}^{\frac12}\Big),\\ \\
    \Dt\Omega_{5,0}&\leq
    -\ve\left(\Omega_{5,0}+\Omega_{6,0}+\Omega_{5,1}\right)
    +C\beta^{\frac12}\Big(\sum_{5\leq i+j\leq 6}\Omega_{i,j}\Big)
    +C\beta^{\frac{21}{10}}\Big(\sum_{5\leq i+j\leq 6}\Omega_{i,j}^{\frac12}\Big).
\end{align*}
\end{lemma}

\begin{proof}
We provide a detailed argument for $\Omega_{2,0}$, and sketch the remaining cases.

\textsc{Argument for} $\Omega_{2,0}$:
Applying Lemma~\ref{PointwiseSecondOrderBounds} and Cauchy--Schwarz in $L^2_\mu$, we get
\begin{align*}
    \frac12\Dt\Omega_{2,0}\leq
    &\lp\V 20,A_v\V 20\rp+\lp(v^{-2}-a)\V 20,\V 20\rp\\
    &+\rho\Big\{\beta\Omega_{2,0}^{\frac12}
    +\beta^{\frac{11}{10}}\big(\Omega_{3,0}^{\frac12}+\Omega_{2,1}^{\frac12}
    +\Omega_{1,2}^{\frac12}\big)\Big\}.
\end{align*}
Define $I_1:=\lp\V 20,A_v\V 20\rp$ and $I_2:=\lp(v^{-2}-a)\V 20,\V 20\rp$.
Here $\rho$ is a constant that can be made as small as desired, because
\[
\rho=\int_{-\infty}^{\infty}\int_{0}^{2\pi} 
\Dz\,\mu(y)\,\Dy=2\pi \int_{-\infty}^{\infty}\big(M+y^2\big)^{-\frac{3}{5}}\ \Dy,
\]
which tends to zero as $M\nearrow\infty$.\footnote{Recall that $\mu(y):=(M+y^2)^{-\frac{3}{5}}$,
where $M$ is an arbitrarily large constant.}

By Remark~\ref{IBP}, we may integrate $I_2$ by parts in $y$, obtaining
\begin{align*}
    I_2&=-a\Omega_{2,0}+\lp (v_{2,0,2}),\V 20\rp\\
    &\leq-a\Omega_{2,0}
        +\lp (v_{1,0,2}),2\V 20\V 10-\V 30-\V 20\dy(\log\mu)\rp\\
    &\leq-a\Omega_{2,0}
        +\rho\big(\beta\Omega_{2,0}^{\frac12}+\beta^{\frac12}\Omega_{3,0}^{\frac12}\big).
\end{align*}
In the last step, we used the facts that $\|\dy\log\mu\|_{L^\infty}=\cO(M^{-\frac12})$
as $M\rightarrow\infty$ and that $\dy(\log\mu)=\cO(|y|^{-1})$ as $|y|\rightarrow\infty$,
along with Condition~[Cg] and estimate~\eqref{v_y-est}.

Again using Remark~\ref{IBP} and integrating by parts, we find that
\[
    I_1\leq I_*+\frac{a}{2}\Omega_{2,0},
\]
where $I_*:=\lp\V 20,F_1\V 40+F_2\V 22+F_3\V 31+F_4 (v_{2,1,2})\rp$. Then defining
$F_0:=(1+p^2+q^2)^{-1}$ lets us estimate
\begin{align*}
    I_*&\leq-\lp F_0,\V 30^2+\V 21^2\rp\\
    &\quad+\frac12\lp\V 20^2, \mu^{-1}\dy^2(\mu F_1)+\dz^2(v^{-2}F_2)
        +\mu^{-1}\dy\dz(v^{-1}\mu F_3)-\dz(v^{-2}F_4)\rp\\
    &\leq-(1-\ve)(\Omega_{3,0}+\Omega_{2,1})+\ve_1\Omega_{2,0}.
\end{align*}
The first inequality results from applying Cauchy--Schwarz pointwise to the coefficients of
the elliptic operator $A_v$ after integrating by parts. To get the second inequality, we use
estimates~\eqref{v_y-bound} and \eqref{v_z-est} from Theorem~\ref{FirstOrderEstimates} to see that
we can bound $F_0\geq(1-\ve)$. We use Lemma~\ref{ControlCoefficients} to control the remaining
terms that came from integrating by parts, noting that by taking $\kappa_0$ and $M$ sufficiently
large, we can by Condition~[Ca] ensure that $\ve_1\in\big(0,\frac{a}{2}-\ve\big)$. The result follows.

\textsc{Arguments for the remaining }$\Omega_{m,0}$: The proofs here are essentially identical
to the corresponding cases considered in Lemmas~6.9, 6.11, and 6.12 of \cite{GKS11}. It is important
to note that we can use the results of Appendix~D from \cite{GKS11} without modification: because
we did not need sharp estimates for derivatives of orders three and higher there (as is also true
here) we did not need to exploit the discrete symmetry assumptions of that paper to get better results
when integrating by parts in $\theta$. As a consequence, the conclusions of Lemmas~D.1--D.7, which
allow us to estimate the nonlinear terms in the $\Omega_{m,n}$ evolution equations, hold here without
modification. For the linear terms, the key estimate, which holds by Conditions~[Ca] and [C0i] and
Theorem~\ref{FirstOrderEstimates} when $3\leq m\leq5$, is
\[
        I_2:=\lp(v^{-2}-(m-1)a)(v_{m,0,0}),(v_{m,0,0})\rp\leq-\frac14\Omega_{m,0}.
\]
We omit further details.
\end{proof}

\begin{lemma}   \label{Omega-many-y}
There exist $0<\ve<C<\infty$ such that for $(m,n)=(2,1)$, one has
%% for $(m,n)\in\{(2,1),(3,1),(2,2),(4,1),(3,2)$, one has
\begin{align*}
    \Dt\Omega_{2,1}\leq
    &-\ve\left(\Omega_{2,1}+\Omega_{3,1}+\Omega_{2,2}\right)
    +C\beta^{\frac{2}{5}}\left(\Omega_{2,1}+\Omega_{1,2}+\Omega_{0,3}\right)\\
    &+C\beta^2\left\{\Omega_{2,1}^{\frac{1}{2}}
    +\beta\Omega_{4,0}^{\frac{1}{2}}+\Omega_{3,1}^{\frac{1}{2}}
    +\Omega_{2,2}^{\frac{1}{2}}+\Omega_{1,3}^{\frac{1}{2}}
    +\Omega_{0,4}^{\frac{1}{2}}\right\};
\end{align*}
and for $(m,n)\in\{(3,1),(2,2)\}$, one has
\begin{align*}
    \Dt\Omega_{m,n}\leq
    &-\ve\left(\Omega_{m,n}+\Omega_{m+1,n}+\Omega_{m,n+1}\right)
    +C\beta^{\frac{1}{2}}\left(\sum_{4\leq i+j\leq 5}\Omega_{i,j}\right)\\
    &+C\beta^2\left(\sum_{4\leq i+j\leq 5}\Omega_{i,j}^{\frac{1}{2}}\right);
\end{align*}
and for $(m,n)\in\{(4,1),(3,2),(2,3)\}$, one has
\begin{align*}
    \Dt\Omega_{m,n}\leq
    &-\ve\left(\Omega_{m,n}+\Omega_{m+1,n}+\Omega_{m,n+1}\right)
    +C\beta^{\frac{1}{2}}\left(\sum_{5\leq i+j\leq 6}\Omega_{i,j}\right)\\
    &+C\beta^{\frac{21}{10}}\left(\sum_{5\leq i+j\leq 6}\Omega_{i,j}^{\frac{1}{2}}\right).
\end{align*}
\end{lemma}

\begin{proof}
We provide a detailed argument for $\Omega_{2,1}$, and outline the remaining cases.
\textsc{Argument for} $\Omega_{2,1}$:
One has
\[
    \frac12\Dt\Omega_{2,1}=I_1+I_2+I_3,
\]
with terms $I_j$ that we now define. The first term on the \textsc{rhs} is
\[
    I_1:=\lp\V 21,A_v\V 21\rp\leq\frac{a}{2}\Omega_{2,1}+I_*,
\]
where
\[
    I_*:=\lp\V 21,F_1\dy^2\V 21+F_2v^{-2}\dz^2\V 21+F_3v^{-1}\dy\dz\V 21+F_4 v^{-2}\dz\V 21\rp.
\]
The second term in $\frac12\Dt\Omega_{2,1}$ is
\[
    I_2:=\lp(2v^{-2}-2a)\V 21,\V 21\rp.
\]
Below, we obtain a good estimate for $I_1+I_2$. Lemma~D.1 of \cite{GKS11}, which does not require
any symmetry assumptions, implies that the remaining contribution to $\frac12\Dt\Omega_{2,1}$, which is
\[
    I_3:=\sum_{\ell=0}^5\lp\V 21,E_{2,1,1,\ell}\rp,
\]
can be estimated by
\[
    |I_3|\leq C\beta^2\left\{\Omega_{2,1}^{\frac12}
    +\beta\Omega_{4,0}^{\frac12}
    +\Omega_{3,1}^{\frac12}
    +\Omega_{2,2}^{\frac12}
    +\Omega_{1,3}^{\frac12}
    +\Omega_{0,4}^{\frac12}\right\}
    +C\beta^{\frac{2}{5}}\left(\Omega_{2,1}+\Omega_{1,2}+\Omega_{0,3}\right).
\]

We now study $I_*$.
Define $F_0:=(1+p^2+q^2)^{-1}$, as in the proof of Lemma~\ref{Omega-all-y}. Then by
estimates~\eqref{v_y-bound} and \eqref{v_z-est} from Theorem~\ref{FirstOrderEstimates},
there exists $\ve_1>0$ such that $F_0\geq(1-\ve_1)$.
Let $\gamma\in(0,1)$ be a parameter to be chosen. By integrating $I_*$ by parts, applying
Cauchy--Schwarz pointwise to the coefficients of the quasilinear elliptic operator $A_v$,
and using Lemma~\ref{ControlCoefficients} to bound terms containing derivatives of the
coefficients of that operator, we find that there exists $\ve_2>0$ such that
\begin{align*}
    I_*
    &\leq-\lp F_0,\{\dy\V 21\}^2+\{v^{-1}\dz\V 21\}^2\rp + \ve_2\Omega_{2,1}\\
    &\leq-(1-\ve_1)\Big\{\gamma\big(\|\dy\V 21\|_\mu^2+\|v^{-1}\dz\V 21\|_\mu^2\big)
        +(1-\gamma)\|v^{-1}\dz\V 21\|_\mu^2\Big\}\\
    &\qquad+\ve_2\Omega_{2,1}.
\end{align*}
Integrating by parts in $y$ and using Conditions~[C1] and [C2] shows that
\[
    \|\dy\V 21\|_\mu^2\geq\Omega_{3,1}-C\beta^{\frac35}\Omega_{2,1}.
\]
Integrating by parts in $\theta$ again using Conditions~[C1] and [C2] shows that
\[
    \|v^{-1}\dz\V 21\|_\mu^2\geq\Omega_{2,2}-C\beta^{\frac32}\Omega_{2,1}.
\]
Then using Condition~[Cr] along with the fact that $\dy^2\dz v$ is orthogonal to
constants in $L^2(\bS^1)$, we integrate by parts in $\theta$ to get
\begin{align*}
    \|v^{-1}\dz\V 21\|_{\bS^1}^2
    &\geq\|v^{-2}(\dy^2\dz^2v)\|_{\bS^1}^2-C\beta^{\frac32}\|\V 21\|_{\bS^1}^2\\
    &\geq\left(\frac{1-\delta}{1+\delta}\right)^4\|v^{-2}(\dy^2\dz v)\|_{\bS^1}^2
        -C\beta^{\frac32}\|\V 21\|_{\bS^1}^2.
\end{align*}
Combining these estimates, we obtain $\ve_3,\ve_4$ such that
\[
    I_*\leq-(1-\ve_3)
    \Big\{\gamma\big(\Omega_{3,1}+\Omega_{2,2}\big)
    +(1-\gamma)\lp v^{-2}\V 21,\V 21\rp\Big\} + \ve_4\Omega_{2,1}.
\]
Hence taking $\gamma$ small and using Conditions~[Ca] and [C0i] with Assumption~[A2],
we obtain $\ve$ such that
\begin{align*}
    I_1+I_2
    &\leq\lp\big\{(1+\ve)v^{-2}+\ve-\frac32 a\big\}\V 21,\V 21\rp-\ve\big(\Omega_{3,1}+\Omega_{2,2}\big)\\
    &\leq-\ve\big(\Omega_{2,1}+\Omega_{3,1}+\Omega_{2,2}\big).
\end{align*}
The result follows.

\textsc{Arguments for the remaining }$\Omega_{m,n}$: As we note in the proof of Lemma~\ref{Omega-all-y},
the conclusions of Lemmas~D.1--D.7 from \cite{GKS11}, which allow us to estimate the nonlinear terms
in the $\Omega_{m,n}$ evolution equations, hold here without modification. For the linear terms, the key
estimates are
\begin{align*}
    I_1+I_2
    &\leq\lp\big\{\frac34-\frac{m}{2}+n\big(v^{-2}-\frac12\big)+\ve\big\}v_{m,n,n},v_{m,n,n}\rp
    -\ve\big(\Omega_{m+n,n}+\Omega_{m,n+1}\big)\\
    &\leq-\ve\big(\Omega_{m,n}+\Omega_{m+1,n}+\Omega_{m,n+1}\big).
\end{align*}
Condition~[Ca] implies the first estimate. The bound $v^{-2}<\frac59$, which follows
from Condition~[C0i] in the inner region and from estimate~\eqref{v-below} in the outer region, implies the
second estimate above. This concludes the proof for all remaining cases $(m,n)\in\{(3,1),(2,2),(4,1),(3,2),(2,3)\}$.
\end{proof}

Finally, we bound the evolutions of the $\Omega_{m,n}$ with relatively ``few'' $y$-derivatives,
namely $(m,n)\in\{(1,2),(0,3),(1,3),(0,4),(1,4),(0,5)\}$. The arguments to prove these estimates
are those that are most changed from our prior work \cite{GKS11}.

\begin{lemma}   \label{Omega-few-y}
There exist $0<\ve<C<\infty$ such that for $(m,n)\in\{(1,2),(0,3)\}$ one has
\begin{align*}
    \Dt\Omega_{m,n}\leq
    &-\ve\left(\Omega_{m,n}+\Omega_{m+1,n}+\Omega_{m,n+1}\right)
    +C\beta^{\frac{3}{5}}\left(\Omega_{2,1}+\Omega_{1,2}+\Omega_{0,3}\right)\\
    &+C\beta^2\left\{\Omega_{m,n}^{\frac{1}{2}}
    +\beta\Omega_{4,0}^{\frac{1}{2}}+\Omega_{3,1}^{\frac{1}{2}}
    +\Omega_{2,2}^{\frac{1}{2}}+\Omega_{1,3}^{\frac{1}{2}}
    +\Omega_{0,4}^{\frac{1}{2}}\right\};
\end{align*}
and for $(m,n)\in\{(1,3),(0,4)\}$, one has
\begin{align*}
    \Dt\Omega_{m,n}\leq
    &-\ve\left(\Omega_{m,n}+\Omega_{m+1,n}+\Omega_{m,n+1}\right)
    +C\beta^{\frac{1}{2}}\left(\sum_{4\leq i+j\leq 5}\Omega_{i,j}\right)\\
    &+C\beta^2\left(\sum_{4\leq i+j\leq 5}\Omega_{i,j}^{\frac{1}{2}}\right);
\end{align*}
and for $(m,n)\in\{(1,4),(0,5)\}$, one has
\begin{align*}
    \Dt\Omega_{m,n}\leq
    &-\ve\left(\Omega_{m,n}+\Omega_{m+1,n}+\Omega_{m,n+1}\right)
    +C\beta^{\frac{1}{2}}\left(\sum_{5\leq i+j\leq 6}\Omega_{i,j}\right)\\
    &+C\beta^{\frac{21}{10}}\left(\sum_{5\leq i+j\leq 6}\Omega_{i,j}^{\frac{1}{2}}\right).
\end{align*}
\end{lemma}

\begin{proof}
We begin with a detailed argument for $\Omega_{1,2}$.

\textsc{Argument for} $\Omega_{1,2}$:
As in Lemma~\ref{Omega-many-y}, we write
\[
    \frac12\Dt\Omega_{1,2}=I_1+I_2+I_3.
\]
Here, the first term on the \textsc{rhs} is
\[
    I_1:=\lp\V 12,A_v\V 12\rp\leq\frac{a}{2}\Omega_{1,2}+I_*,
\]
where
\[
    I_*:=\lp\V 12,F_1\dy^2\V 12+F_2v^{-2}\dz^2\V 12+F_3v^{-1}\dy\dz\V 12+F_4 v^{-2}\dz\V 12\rp
\]
is analyzed below. The second term in $\frac12\Dt\Omega_{1,2}$ is
\[
    I_2:=\lp(3v^{-2}-2a)\V 12,\V 12\rp.
\]
We estimate $I_1+I_2$ below. Lemma~D.1 of \cite{GKS11}, which does not require any symmetry
assumptions, implies that the remaining contribution to $\frac12\Dt\Omega_{1,2}$, which is
the quantity $I_3:=\sum_{\ell=0}^5\lp\V 12,E_{1,2,2,\ell}\rp$, can be estimated by
\[
    |I_3|\leq C\beta^2\left\{\Omega_{1,2}^{\frac12}
    +\beta\Omega_{4,0}^{\frac12}
    +\Omega_{3,1}^{\frac12}
    +\Omega_{2,2}^{\frac12}
    +\Omega_{1,3}^{\frac12}
    +\Omega_{0,4}^{\frac12}\right\}
    +C\beta^{\frac{3}{5}}\left(\Omega_{2,1}+\Omega_{1,2}+\Omega_{0,3}\right).
\]

To study $I_*$, we again define $F_0:=(1+p^2+q^2)^{-1}$. Then by estimates~\eqref{v_y-bound}
and \eqref{v_z-est} from Theorem~\ref{FirstOrderEstimates}, there exists $\ve_1>0$ such that
$F_0\geq(1-\ve_1)$. Let $\gamma\in(0,1)$ be a parameter to be chosen. By integrating $I_*$
by parts, applying Cauchy--Schwarz pointwise to the coefficients of the quasilinear elliptic
operator $A_v$, and using Lemma~\ref{ControlCoefficients} to bound terms containing derivatives
of the coefficients of that operator, we find that there exists $\ve_2>0$ such that
\begin{align*}
    I_*
    &\leq-\lp F_0,\{\dy\V 12\}^2+\{v^{-1}\dz\V 12\}^2\rp + \ve_2\Omega_{1,2}\\
    &\leq-(1-\ve_1)\gamma\Big\{\|\dy\V 12\|_\mu^2+\|v^{-1}\dz\V 12\|_\mu^2\Big\}\\
    &\quad-(1-\ve_1)(1-\gamma)\int_{\beta y^2\leq20}\|v^{-1}\dz\V 12\|_{\bS^1}^2\,\mu\,\Dy
    +\ve_2\Omega_{1,2}.
\end{align*}
Conditions~[C1] and [C2] let us integrate by parts in $y$ to estimate
$\|\dy\V 12\|_\mu^2\geq\Omega_{2,2}-C\beta^{\frac35}\Omega_{1,2}$, and then integrate by parts
in $\theta$ to estimate $\|v^{-1}\dz\V 12\|_\mu^2\geq\Omega_{1,3}-C\beta^{\frac32}\Omega_{1,2}$.
Thus we obtain $\ve_3$ such that
\begin{align*}
    I_* &\leq-(1-\ve_1)\gamma\big(\Omega_{2,2}+\Omega_{1,3}\big)+\ve_3\Omega_{1,2}\\
    &\qquad-(1-\ve_1)(1-\gamma)\int_{\beta y^2\leq20}\|v^{-3}\dy\dz^3v\|_{\bS^1}^2\,\mu\,\Dy.
\end{align*}

We now estimate $I_1+I_2$. Using Conditions~[Ca], [C0i], and estimate~\eqref{v-below}, we observe that
$I_2:=\lp(3v^{-2}-2a)\V 12,\V 12\rp$ may be estimated by\,\footnote{See below for the formula we use to
get this decomposition for general $(m,n)$.}
\begin{align*}
    I_2
    &=-\frac{2a}{3}\Omega_{1,2}+\frac{19}{10}\int_{\beta y^2\leq20}\|v^{-3}\dy\dz^2v\|_{\bS^1}^2\,\mu\,\Dy\\
    &\qquad+\int_{\beta y^2\leq20}\int_{\bS^1}\Big(\frac{11}{10}v^{-2}-\frac{4a}{3}\Big){\V 12}^2\,\Dz\,\mu\,\Dy\\
    &\qquad+\int_{\beta y^2>20}\int_{\bS^1}\Big(3v^{-2}-\frac{4a}{3}\Big){\V 12}^2\,\Dz\,\mu\,\Dy\\
    &\leq-\frac{2a}{3}\Omega_{1,2}+\frac{19}{10}\int_{\beta y^2\leq20}\|v^{-3}\dy\dz^2v\|_{\bS^1}^2\,\mu\,\Dy.
\end{align*}
Combining this with the estimates above, we obtain
\[
    I_1+I_2
    \leq\left(\ve_3-\frac{a}{6}\right)\Omega_{1,2}
        -(1-\ve_1)\gamma\big(\Omega_{2,2}+\Omega_{1,3}\big)
        +\int_{\beta y^2\leq20}\Xi\,\mu\,\Dy,
\]
where the quantity that remains to be controlled is
\[
    \Xi:=\frac{19}{10}\|v^{-3}\dy\dz^2v\|_{\bS^1}^2
    -(1-\ve_1)(1-\gamma)\|v^{-3}\dy\dz^3v\|_{\bS^1}^2.
\]

To estimate $\Xi$, let $\bP$ denote orthogonal projection onto
$\mathrm{span}\{e^{\pm i\theta}\}\subset L^2(\bS^1)$. Then using Condition~[Cr], it is easy
to see that there exists $\ve_4$ depending only on $\delta$ such that in the inner region,
\[
    \|v^{-3}\dy\dz^3v\|_{\bS^1}^2\geq(1-\ve_4)
    \big\{\|v^{-3}\bP(\dy\dz^3 v)\|_{\bS^1}^2+\|v^{-3}(1-\bP)(\dy\dz^3 v)\|_{\bS^1}^2\big\}.
\]
Again using [Cr] together with the fact that $(1-\bP)(\dy\dz^3 v)\perp\mathrm{span}\{1,e^{\pm i\theta}\}$,
we integrate by parts as in \cite{GKS11} to estimate that
\begin{align*}
    \|v^{-3}(1-\bP)(\dy\dz^3 v)\|_{\bS^1}^2
    &=\|v^{-3}\dz\{(1-\bP)(\dy\dz^2 v)\}\|_{\bS^1}^2\\
    &\geq 4(1-\ve_4)\|v^{-3}(1-\bP)(\dy\dz^2 v)\|_{\bS^1}^2.
\end{align*}
Then using the fact that $\dy\dz^2v\perp\mathrm{span}\{1\}$ in $L^2(\bS^1)$, we obtain
\begin{align*}
    \Xi
    &\leq\Big\{\frac{19}{10}(1+\ve_4)-(1-\ve_4)(1-\ve_1)(1-\gamma)\Big\}\|v^{-3}\bP(\dy\dz^2v)\|_{\bS^1}^2\\
    &+\Big\{\frac{19}{10}(1+\ve_4)-4(1-\ve_4)(1-\ve_1)(1-\gamma)\Big\}\|v^{-3}(1-\bP)(\dy\dz^2v)\|_{\bS^1}^2.
\end{align*}
It is easy to choose $\gamma\in(0,1)$ so that the second quantity in braces above is negative.
To control the remaining term, we observe that the new Condition~[C2i] implies that
$\|\bP(\dy\dz^2v)\|_{\bS^1}\ls\beta^2$, and then use Cauchy-Schwarz to see that
\[
    \int_{\beta y^2\leq20}\|v^{-3}\bP(\dy\dz^2v)\|_{\bS^1}^2\,\mu\,\Dy
    \leq\rho\beta^2\Omega_{1,2}^{\frac12},
\]
where $\rho(M)\searrow0$ as $M\rightarrow\infty$. The result follows.

\textsc{Arguments for the remaining }$\Omega_{m,n}$: The nonlinear terms are estimated
as above, \emph{mutatis mutandis.} To control the linear terms, we first observe that
\begin{align*}
    I_2&:=\int_{\bR}\int_{\bS^1}\big\{(n+1)v^{-2}-(m+n-1)a\big\}(v_{m,n,n})^2\,\Dz\,\mu\,\Dy\\
    &=-\frac{2a}{3}\Omega_{m,n}+\frac{26+n-9m}{10}\int_{\beta y^2\leq20}\|v^{-(n+1)}\dy^m\dz^n v\|_{\bS^1}^2\,\mu\,\Dy\\
    &\qquad+\int_{\beta y^2\leq20}\left\{\frac{9(m+n)-16}{10}v^{-2}-\left(m+n-\frac53\right)a\right\}(v_{m,n,n})^2\,\Dz\,\mu\,\Dy\\
    &\qquad+\int_{\beta y^2>20}\left\{(n+1)v^{-2}-\left(m+n-\frac53\right)a\right\}(v_{m,n,n})^2\,\Dz\,\mu\,\Dy\\
    &\leq-\frac{2a}{3}\Omega_{m,n}+\frac{26+n-9m}{10}\int_{\beta y^2\leq20}\|v^{-(n+1)}\dy^m\dz^n v\|_{\bS^1}^2\,\mu\,\Dy.
\end{align*}
The estimate here again follows from Conditions~[Ca], [C0i], and estimate~\eqref{v-below}.
Thus we obtain
\[
    I_1+I_2
    \leq\left(\ve_3-\frac{a}{6}\right)\Omega_{m,n}
        -(1-\ve_1)\gamma\big(\Omega_{m+1,n}+\Omega_{m,n+1}\big)
        +\int_{\beta y^2\leq20}\Xi\,\mu\,\Dy,
\]
where the quantity that remains to be controlled is now
\[
    \Xi:=\frac{26+n-9m}{10}\|v^{-(n+1)}\dy^m\dz^n v\|_{\bS^1}^2
    -(1-\ve_1)(1-\gamma)\|v^{-(n+1)}\dy^m\dz^{(n+1)}v\|_{\bS^1}^2.
\]
We separate $\Xi$ by using orthogonal projection as above, using Condition~[C2i] to control the term
involving $\bP(\dy^m\dz^n v)$, and integrating by parts in $\theta$ to control the term involving
$(1-\bP)(\dy^m\dz^n v)$, using the fact that $\frac{26+n-9m}{10}\leq\frac{31}{10}<4$ for all $(m,n)$
under consideration. This completes the proof.
\end{proof}

\subsection{Second-order estimates in the inner region}
With Lemmas~\ref{Omega-all-y}--\ref{Omega-few-y} in hand, we are ready to prove
Theorem~\ref{InnerEstimates}.

We begin by deriving $L^2_\mu$ bounds on all derivatives of orders two through five.

\begin{lemma}\label{SharpLyapunov}
Suppose that a solution $v=v(y,\theta,\tau)$ of equation~\eqref{MCF-v}
satisfies Assumption~[A1] at $\tau=0$, as well as [Ca], [C0]--[C3],
[C0i]--[C2i], [Cg], [Cr], and [Cs] for $\tau\in [0,\tau_1]$.
Then for the same time interval, one has
\[
    \beta^2\Omega_{2,0}+\Omega_{1,1}+\Omega_{0,2}+\beta\Omega_{3,0}
    +\Omega_{2,1}+\Omega_{1,2}+\Omega_{0,3}\ls\beta^4
\]
and
\[
    \beta^{\frac{4}{5}}\Omega_{4,0}+\Omega_{3,1}+\Omega_{2,2}+\Omega_{1,3}+\Omega_{0,4}
    +\sum_{m+n=5}\Omega_{m,n}\ls\beta^4.
\]
\end{lemma}

\begin{proof}
The proof of Proposition~6.10 in \cite{GKS11} applies without change to establish the first estimate.

To establish the second estimate, we define
\[
    \Upsilon:=\beta^{\frac{4}{5}}\Omega_{4,0}+\Omega_{3,1}+\Omega_{2,2}+\Omega_{1,3}+\Omega_{0,4}
    +\sum_{m+n=5}\Omega_{m,n},
\]
noting the $\beta$-weight imposed on the first term.
By Lemmas~\ref{Omega-all-y}--\ref{Omega-few-y} and Cauchy--Schwarz, there exist $0<\ve<C<\infty$ such that
\[
    \frac{d}{d\tau}\Upsilon\leq-\ve\Upsilon+C\left(\beta^2\Upsilon^{\frac{1}{2}}+\beta^4\right).
\]
Assumption~[A6] bounds $\Upsilon$ at $\tau=0$. It follows that $\Upsilon\ls\beta^4$. The result follows.
\end{proof}

Using these estimates, we complete the proof of Theorem~\ref{InnerEstimates}.

\begin{proof}[Proof of Theorem~\ref{InnerEstimates}]
Lemma~\ref{SharpLyapunov} provides $L^2_\mu$ bounds on derivatives of orders two through five.
To complete the proof, we use Sobolev embedding to obtain pointwise bounds on derivatives of orders two and three.
Because of the weighted norm $\|\cdot\|_\mu$ these pointwise bounds are not uniform in $y$. Using the facts
that $\mu\sim\lan y \ran^{-\frac{6}{5}}$ as $|y|\rightarrow\infty$ and that
$|y|\ls\beta^{-\frac{1}{2}}$ in the inner region, one obtains
\[
    \beta^2(\dy^2 v)^2+v^{-2}(\dy\dz v)^2+v^{-4}(\dz^2 v)^2\ls
    \beta^{4-\frac{1}{10}-\frac{3}{5}}
\]
and
\[
    \beta(\dy^3 v)^2+v^{-2}(\dy^2\dz v)^2+v^{-4}(\dy\dz^2 v)^2+v^{-6}(\dz^3 v)^2
    \ls\beta^{4-\frac{1}{10}-\frac{3}{5}}
\]
in the inner region. These inequalities are equivalent to estimates~\eqref{InnerSecondOrder}--\eqref{InnerThirdOrder}.
\end{proof}

\section{The second bootstrap machine}\label{sec:asymp}\label{SecondBootstrap}

In this section and those that follow, we analyze the asymptotic behavior of solutions. Specifically, we
show that a solution $v$  may be decomposed into a slowly-changing ``main component'' and a rapidly-decaying
``small component.'' The main component is controlled by finitely many parameters, some of which encode 
information about the optimal coordinate system at a given time.

We accomplish this by building a second bootstrap machine, following \cite{GKS11} and making modifications
where necessary. Many of the arguments of \cite{GKS11} apply here and so do not have to be repeated in
detail. Significant new arguments are needed in two areas. \textsc{(i)} The discrete symmetry assumptions
imposed in  \cite{GKS11} effectively fixed the cylindrical axis once and for all, which allowed us to track the
behavior of the main component of the solution using only two parameters. Here, we must deal with seven
parameters, corresponding to the seven-dimensional weakly unstable eigenspace of the  linearization of \mcf,
as explained in Section~\ref{Strategy}. \textsc{(ii)} We must derive improved estimates for $|v_\pm|$ and
$|(\dy v)_\pm|$ in the inner region, stated as \eqref{eq:vpmEst} below. These allow the first bootstrap
machine to function without the discrete symmetries imposed in our earlier work.

\subsection{Inputs}\label{SBMI}
The second bootstrap machine requires three sets of inputs.
\medskip

First we require that $u$ be a solution of equation~\eqref{MCF-u} satisfying the following condition:
\begin{itemize}
\item[{[Cd]}]
There exists $t_{\#}>0$ such that for $0\le t\le t_{\#}$, there exist $C^1$ functions $a(t)$, $b(t)$
and $\beta_{k}(t)$ $(k=0,\dots,4)$ such that $u$ admits a decomposition
\begin{align}\label{eqn:split2}
u(x,\theta,t) = &\lambda(t) v(y,\theta,\tau)\nonumber \\
 =& \lambda (t)
\Bigg[
\left(\frac{2+b(t)y^{2}}{2-2a(t)}\right)^{\frac12}
+\beta_0(t) y+\beta_1(t) \cos\theta+\beta_2(t) \sin\theta\\
&+\beta_3(t)y \cos\theta +\beta_4(t) y\sin\theta  + \phi (y,\theta,\tau)\Bigg] \nonumber
\end{align}
with the $L^2$ orthogonality properties
$$e^{-\frac{a y^2}{2}}\phi \perp\big\{1,\;y,\;1-ay^2,\;\cos\theta,\;\sin\theta,\;y\cos\theta\,\;y\sin\theta\big\},$$
where
\begin{align*}
	a(t)&:= -\lambda (t)\partial_t \lambda(t),\\
	y&:=\lambda^{-1}(t)x,\\
	\tau(t) &:= \int_0^t \lambda^{-2} (s)\,\mathrm{d}s,\\
	\beta_{k}(t_{\#})&=0,\quad(k=0,\dots,4),\\
	a(t_{\#})&=\frac{1}{2}-\frac{1}{4}b(t_{\#}).
\end{align*}
\end{itemize}
Condition~[Cd] follows from our Main Assumptions by an implicit function theorem argument
highly analogous to the decomposition result proved in \cite{GS09} and also used in \cite{GKS11}.
We omit the details here.
\medskip

To state the second set of inputs, we define majorizing functions to control the decay of the
quantities $\phi(y,\theta,\tau)$, $a(t(\tau))$, and $b(t(\tau))$ that appear in equation~\eqref{eqn:split2}.
These are:
\begin{align}
M_{m,n}(T) := &\displaystyle\max_{\tau \le T}
\beta^{-\frac{m+n}2 - \frac1{10}}  (\tau)
\|\phi(\cdot,\cdot,\tau)\|_{m,n}, \label{eq:majorMmn}\\
\noalign{\vskip6pt}
A(T) := &\displaystyle \max_{\tau\le T}  \beta^{-2} (\tau)
\left| a(t(\tau)) - \frac12 + \frac14b(t(\tau))\right|, \label{eq:majorA}\\
\noalign{\vskip6pt}
B(T) :=& \displaystyle\max_{\tau\le T} \beta^{-\frac32}  (\tau) |b(t(\tau)) - \beta(\tau)|,  \label{eq:majorB}
\end{align}
where $(m,n)\in\left\{(3,0),\ (11/10,0),\ (2,1), \ (1,1)\right\}$.
By standard regularity theory for quasilinear parabolic equations, if the
initial data satisfy the Main Assumptions [A1]--[A7] for $b_0$
and $c_0$ sufficiently small, then, making $t_{\#}>0$ smaller if necessary,
the solution will satisfy the second set of inputs for this bootstrap argument, which are
contained in the following condition:
\begin{itemize}
\item[{[Cb]}] For any $\tau \le \tau(t_{\#})$, one has
\begin{align*}%\label{eq:assuMAB}
A(\tau)+B(\tau)+|M(\tau)| &\ls  \beta^{-\frac1{20}} (\tau),\\
|\beta_{k}(t)| & \ls  \beta^{2}(\tau),\ k=0,1,2,3,4,\nonumber
\end{align*}
\end{itemize}
where $M$ denotes the vector
\begin{equation*}%\label{eq:vector}
M := (M_{i,j}), \ (i,j)\in\left\{(3,0),\ (11/10,0),\ (2,1),\ (1,1)\right\}.
\end{equation*}

\begin{remark}\label{NoCircleHere}
Condition~[Cb] implies estimates on $v$ and $v^{-\frac12} \dy v$ in the
inner region $\beta y^2\le  20$, which in turn imply Condition~[C0i], the estimate
$|\dy v|\ls\beta^{\frac12} v^{\frac12}$ of Condition~[C1i], and the estimate
$\langle y\rangle^{-1} |\dy v |\leq M^{\frac14}\beta$ of Condition~[Cg], for
$M$ sufficiently large.
\end{remark}
\medskip

The final inputs to this bootstrap machine are the outputs of the first bootstrap
machine, which are listed in Section~\ref{FirstOutput}. They follow from
Theorems~\ref{FirstOrderEstimates}--\ref{SmallnessEstimates}. In the sequel,
we give ourselves the freedom to use all estimates detailed in Section~\ref{FirstOutput}.
For the reader's convenience, we list here the implications of those estimates
most commonly used in this machine: for any $\tau \in [0, \tau(t_{\#})]$, the quantity $v$ satisfies
\begin{align}\label{eq:lower}
v(y,\theta,\tau) \ge 1;
\end{align}
and there exist constants $\epsilon_0\ll 1$ and $C$, independent of $\tau(t_{\#})$, such that
\begin{align}
|\dy  v| &\le  C,	\label{eq:upperBounddV}\\
|v^{-1} \dz^2 v|+ |\dy \dz^2 v| &\le \epsilon_0 \ll 1, \label{eq:smallness}\\
v^{-1} |\dy v| \le C\beta^{\frac12},\qquad\qquad |\dy^2 v| \le  C\beta^{\frac{13}{20}},
\qquad\qquad|\dy^3  v|  &\le  C\beta^{\frac{23}{20}}, \label{eq:yDe}\\
v^{-2} |\dy \dz^2 v|+v^{-1} |\dy \dz v|+v^{-1}|\dy^2 \dz v|
+v^{-2}|\dz^2 v|+v^{-3}|\dz^3 v|&\le C \beta^{\frac{33}{20}}.\label{eq:thetaDe}
\end{align}

\subsection{Outputs}		\label{SecondOutput}
We now state the main outputs of the second bootstrap machine. They serve two main purposes.
\textsc{(i)} The estimates stated below improve some of those in the inputs. Together with the local
well-posedness of the solution, this fact enables us to continue making bootstrap arguments in larger
time intervals. \textsc{(ii)} Most importantly, the estimates below provide almost sharp control on 
various components of the solution, and hence a clear understanding of the geometry and
asymptotic behavior of the evolving surface.

\begin{theorem}	\label{THM:aprior}
Suppose that Conditions~[Cd] and [Cb] and estimates~\eqref{eq:lower}--\eqref{eq:thetaDe}
hold in the interval $\tau\in [0, \tau(t_{\#})]$. Then there exists a constant $C$ independent of $\tau(t_{\#})$ such that  in the same time interval:

\textsc{(i)} The parameters $a$, $b$, and $\beta_{k}$, $(k=0,\dots,4)$ satisfy the estimates
\begin{align*}
A(\tau)+B(\tau) &\leq C, \\ %\label{EstABM}\\
|\beta_0(t)| &\leq   C\beta^{\frac{13}{5}},\ \nonumber\\
|\beta_1(t)|+|\beta_2(t)| &\leq   C\beta^{\frac{18}{5}}, \\ %\label{eq:parameterBeta}\\
|\beta_3(t)|+|\beta_4(t)| & \leq   C\beta^{\frac{13}{5}}.\nonumber
\end{align*}

\textsc{(ii)} The function $\phi$ satisfies improved estimates implied by the bound
\begin{equation*} %\label{eq:boundM}
|M(\tau)|\le C.
\end{equation*}

\textsc{(iii)} In the region $\beta y^2\leq 20$, one can estimate that
\begin{align}
|v_{\pm}|+|\partial_{y}v_{\pm}|\leq C \beta^{\frac{21}{10}}\label{eq:vpmEst}.
\end{align}
\end{theorem}

The theorem will be reformulated in Section~\ref{SlowlyDecompose} below. In that
section, we provide a heuristic outline of the ideas used in its proof, with an emphasis
on those that differ from what we used in \cite{GKS11}. The technical details of the
proof appear in the appendices.
\medskip

For the reader's convenience, we reformulate some implications of the estimates
proved in Theorem~\ref{THM:aprior} and Lemma~\ref{LM:traj} (below) in forms that
are most useful for our subsequent applications.

\begin{corollary}	\label{thm:estimateT0}
Suppose that the assumptions of Section~\ref{SBMI} hold in a time interval
$0\leq\tau\leq\tau(t_{\#})$. Then for that same time interval, the parameters
$a$ and $b$ satisfy
\begin{align*}
a(t)&=\frac12-\frac{b(t)}{4}+\mathcal O\big(b^2(t)\big),\\
b(t)&=\big(1+o(1)\big)\frac{1}{b_0^{-1}+\tau(t)},
\end{align*}
while the ``small component'' $\phi$ of the decomposition~\eqref{eqn:split2}
satisfies
\[
\|\langle y \rangle^{-3}\phi\|_{L^\infty}
	+\|\langle y \rangle^{-2}\dy\phi\|_{L^\infty}
	+\|v^{-2}\dz^2 v\|_{L^\infty}\ls b^{\frac85}.
\]

\end{corollary}

Another consequence of the estimates in Theorem~\ref{THM:aprior} and Lemma~\ref{LM:traj}
gives control on the sequence of optimal coordinate systems we construct.

\begin{corollary}\label{prop:differenceCoor}
Let the assumptions of Section~\ref{SBMI} hold for times $0\leq t\leq t_{\#}$, and
suppose that $(x_0,x_1,x_2)_n$ and $(x_0,x_1,x_2)_{n+1}$ are 
optimal coordinates at times $t_n<t_{n+1}$ respectively, where
$[t_n,t_{n+1}]\subseteq[0,t_{\#}]$.
Then there exist $\Phi\in\bR^3$ and $\Psi\in\mathrm O(3)$ such that
$(x_0,x_1,x_2)_{n+1}=\Phi+(x_0,x_1,x_2)_n\Psi$, where
\[
 |\Phi|+|\Psi-I|\ls b_\opt^2(t_n).
\]
Furthermore, for $t\in[t_n,t_{n+1}]$, the parameters $b$ and $\lambda$ satisfy the estimates
\begin{align*}
b(t)&=b_\opt(t_n)\big\{1+\cO(b_\opt^2(t_n))\big\},\\
\lambda(t)&=\lambda_\opt(t_n)\big\{1+\cO(b_\opt^2(t_n))\big\}.
\end{align*}
\end{corollary}

\section{Improved estimates for the decomposition}\label{SlowlyDecompose}
We begin by putting the equation into a better form. Instead of studying the evolution equation
for $\phi$, we fix a gauge by means of a weight factor $e^{-\frac{a}4 y^2}$ chosen so that the
linearization, to be obtained below, becomes a self-adjoint operator.

By Condition~[Cd] in Section~\ref{sec:asymp}, there exists a ($t$-scale) time
$0<t_{\#} \le \infty$ such that the gauge-fixed quantity
$$w(y,\theta,\tau) := v(y,\theta,\tau) e^{-\frac{a}4 y^2}$$ can be decomposed as
\begin{equation}
w = w_{ab}(y) +e^{-\frac{a}4 y^2}
[\beta_0 y+\beta_1 \cos\theta+\beta_2 \sin\theta+\beta_3 y \cos\theta +\beta_4 y \sin\theta ]
+\xi(y,\theta,\tau).	 \label{eqn:split3}
\end{equation} 
Here $w_{ab}:=V_{a,b}e^{-\frac{a}4 y^2}$, and
\begin{align}	\label{eq:orthoxi}
e^{-\frac{a}4 y^2}\xi\perp&\big\{1,\;y,\; y^2,\; \cos\theta,\; \sin\theta,\; y\cos\theta,\; y\sin\theta\big\},
\end{align}
where the orthogonality is with respect to the $L^2(\bS^1\times\bR)$ inner product,\footnote{In
the remainder of this paper, the inner product we use is $\langle\cdot,\cdot\rangle$. As
defined in Section~\ref{Notation}, this is the standard (unweighted)  inner product for
$L^2\equiv L^2(\mathbb{S}^1\times\mathbb{R})$.} and
\begin{align}
\beta_{k}(t_{\#}) =&0,\quad(k=0,1,2,3,4),\label{eq:boundary}\\
a(t_{\#})=&\frac{1}{2}-\frac{1}{4}b(t_{\#})\nonumber.
\end{align}
The parameters $a$, $b$ and $\beta_{k},\ k=0,1,\cdots,4$, are $C^1$ functions of $t$.\footnote{To
simplify notation, we mildly abuse notation by writing $a(\tau)$, $b(\tau)$, and $\beta_{k}(\tau)$ for $a(t(\tau))$,
$b(t(\tau))$, and $\beta_{k}(t(\tau))$, respectively. This will not cause confusion, as the original functions
$a(t)$, $b(t)$, and $\beta_{k}(t)$ are not needed in what follows.}
\medskip

We now derive the \textsc{pde} satisfied by the  ``small component'' $\xi(y,\theta,\tau)$. Below,
we use this to derive the \textsc{ode} satisfied by the parameters $a(\tau)$, $b(\tau)$, and  $\beta_{k}(\tau)$.
Using equations~\eqref{eqn:split3} and \eqref{MCF-v}, one finds that $\xi$ evolves by
\begin{equation}\label{eq:xi}
\dt \xi = - L(a,b) \xi + F(a,b,\beta) + N_1 (a,b,\xi) +N_2 (a,b,\xi) + N_3 (a,b,\xi),
\end{equation}
where $L(a,b)$ is the linear operator
$$L(a,b) := - \dy^2 + \frac{a^2 + \partial_\tau a }4 y^2 -\frac{3a}2- \frac{2-2a}{2+b y^2} - \frac12 \dz^2.$$
The remaining terms on the \textsc{rhs} of the evolution equation~\eqref{eq:xi} are as follows.

The quantity $F$ has two parts, $F(a,b,\beta) :=F_1(a,b)+F_{2}(a,b,\beta)$, which are
\begin{align*}
F_1(a,b) & := e^{-\frac{ay^2}{4}} \left[
\frac{1}{\sqrt{2-2a}}\frac{2b}{(2+by^2)^{\frac{3}{2}}}-\frac{aby^2}{\sqrt{2-2a}\sqrt{2+by^2}}+a\sqrt{\frac{2+by^2}{2-2a}} \right.\\
&\qquad \qquad \left. 
-\sqrt{\frac{2-2a}{2+by^2}}-\frac{\frac{1}{2}b_{\tau}y^2}{\sqrt{2-2a}\sqrt{2+by^2}}-\frac{\sqrt{2+by^2}}{(2-2a)^{\frac{3}{2}}}a_{\tau}\right].
\end{align*}
and
\begin{align*}
F_{2}(a,b,\beta)& := e^{-\frac{a y^2}4}\left[\partial_{y}^2+V_{a,b}^{-2}\partial_{\theta}^2-ay\partial_{y}+a+V_{a,b}^{-2}-\partial_{\tau}\right]\\
&\qquad 
\left[\beta_0 y+\beta_1 \cos\theta+\beta_2 \sin\theta+\beta_3 y \cos\theta +\beta_4 y \sin\theta \right].
\end{align*}

The first nonlinear term $N_1$ is 
\begin{align}
N_1 (a,b,\xi) := & - \frac1v   \frac{2-2a}{2+by^2} e^{\frac{ay^2}4} \tilde\xi^2 ,
\label{eq:defN1}
\end{align}
where
 $$\tilde\xi:= e^{-\frac{ay^2}4} [e^{\frac{ay^2}4}\xi+\beta_0 y+\beta_1 \cos\theta+\beta_2 \sin\theta+\beta_3 \cos\theta y+\beta_4 \sin\theta y].$$

The second nonlinear term $N_2$ is
\begin{align*}
N_2 (a,b,\xi) := & - e^{-\frac{ay^2}4} \frac{p^2}{1+p^2+q^2} \dy^2 v . %\label{eq:defN2}
\end{align*}

The final nonlinear term $N_3$ is
\begin{align}\label{eq:difN3}
N_3 (a,b,\xi) :=  N_{3,1}(a,b,\xi)+N_{3,2}(a,b,\xi),
\end{align}
where $$N_{3,1}:=\left[V_{a,b}^{-2}-\frac{1}{2}\right]\partial_{\theta}^2\xi,$$
and
\begin{multline*}
 N_{3,2}:=  \left[v^{-2} \frac{1+p^2}{1+q^2+p^2} - V_{a,b}^{-2} \right]
\dz^2 v\ e^{-\frac{ay^2}4}\\
- e^{-\frac{ay^2}4}  v^{-1} \frac{2pq}{1+p^2+q^2}\dz  \dy v
+ e^{-\frac{ay^2}4}  v^{-2}  \frac{q}{1+p^2+q^2} \dz v.
\end{multline*}

\subsection{The finite-dimensional part of the decomposition}
We next derive and estimate \textsc{ode} for the parameters $a$, $b$ and $\beta_k$, $(k=0,\cdots, 4)$ by using
the evolution equation~\eqref{eq:xi} for $\xi$ and the orthogonality conditions~\eqref{eq:orthoxi}. Before providing
the details of the technical proof, we illustrate the key ideas involved by heuristically outlining how one derives and
estimates an \textsc{ode} for $\beta_0$.

Taking the $L^2(\bS^1\times\bR)$ inner product of~\eqref{eq:xi} with the function $y e^{-\frac{a}{4}y^2}$ yields
$$
\left\langle y e^{-\frac{a}{4}y^2},\ \partial_{\tau}\xi\right\rangle
=- \left\langle  y e^{-\frac{a}{4}y^2},\ L(a,b) \xi \right\rangle
+ \left\langle y e^{-\frac{a}{4}y^2},\ F(a,b,\beta)\right\rangle+\cdots.
$$ 
The fact that $\xi\perp y e^{-\frac{a}{4}y^2}$ in ~\eqref{eq:orthoxi} implies that
\begin{align*}
\left|\left\langle y e^{-\frac{a}{4}y^2},\ \partial_{\tau}\xi \right\rangle \right|
= \left|\partial_{\tau}\left\langle y e^{-\frac{a}{4}y^2},\ \xi\right\rangle
+\frac{a_{\tau}}{4} \left\langle y^3e^{-\frac{a}{4}y^2},\ \xi \right\rangle \right|\ls |a_{\tau}|\,\|\langle y\rangle^{-3} e^{\frac{a}{4}y^2}\xi\|_{L^\infty}.
\end{align*}

Noting that $y e^{-\frac{a}{4}y^2}$ is an eigenfunction of the self-adjoint operator
\[
-\partial_{y}^2-\frac{1}{2}\partial_{\theta}^2+\frac{a^2}{4} y^2-2a -\frac{1}{2}
\]
that constitutes the ``main part" of the linear operator $L(a,b)$ appearing in ~\eqref{eq:xi}, we obtain
$$\left|\left\langle y e^{-\frac{a}{4}y^2},\ L(a,b)\xi\right\rangle\right|
\ls\big(|a_{\tau}|+b\big)\,\|\langle y\rangle^{-3} e^{\frac{a}{4}y^2}\xi\|_{L^\infty}.$$
Because $y e^{-\frac{a}{4}y^2}$ is an odd function of $y$ and is independent of $\theta$, we then get
$$
\left\langle y e^{-\frac{a}{4}y^2},\ F(a,b,\beta)\right\rangle
= 2\pi \left[\Omega_1(a,b)\beta_0-\partial_{\tau}\beta_0\right]
\int_{-\infty}^{\infty} y^2 e^{-\frac{a}{2}y^2}\,\mathrm{d}y,
$$
where $\Omega_1$ is the positive scalar function
\begin{align}\label{eq:Omega1}
\Omega_{1}(a,b):=&\frac{\int_{-\infty}^{\infty}y^2 V_{a,b}^{-2} e^{-\frac{a}{2}y^2}\,\mathrm{d}y}
{ \int_{-\infty}^{\infty}y^2  e^{-\frac{a}{2}y^2}\,\mathrm{d}y}\approx \frac{1}{2}.
\end{align}
Here we used the fact that $V_{a,b}=\sqrt{\frac{2+by^2}{2-2a}}\approx \frac{1}{2}$, which is implied by
$a\approx \frac{1}{2}$ and the fact that $b$ is positive and small.
Collecting the estimates above yields the desired equation for $\beta_0,$ which has the form
$$\partial_{\tau}\beta_0-\Omega_1(a,b)\beta_0=\cO\big((|a_\tau|+|b|) \|\langle y\rangle^{-3}
e^{\frac{a}{4}y^2}\xi\|_{L^\infty}\big)+\cdots.$$

Finally, we sketch how one derives an estimate for $\beta_0$ based on this \textsc{ode}. Using the
boundary condition $\beta_0(\tau(t_{\#}))=0$ from \eqref{eq:boundary}, we can rewrite the equation
for $\beta_0$ at any $\tau\leq \tau(t_{\#})$ as
\begin{equation*}
\beta(\tau)=\int_{\tau}^{\tau(t_{\#})}\Big\{e^{-\int^{\tau(t_{\#})}_s \Omega_1(a, b)(s_1)\ ds_1}\;\cO
\big((|a_\tau|+|b|) \|\langle y\rangle^{-3} e^{\frac{a}{4}y^2}\xi\|_{L^\infty}\big)(s)+\cdots\Big\}\,\mathrm{d}s.
\end{equation*}
Using the fact that $\Omega_1\approx \frac{1}{2}$ and the smallness of $|a_\tau|$, $b$, and
$\|\langle y\rangle^{-3} e^{\frac{a}{4}y^2}\xi\|_{L^\infty}$, this implies the desired estimate.
\medskip

Using the method outlined above, we  obtain the following estimates.
\begin{lemma}\label{LM:traj}
For all times that the assumptions of Section~\ref{SBMI} hold, one has
\begin{align}
-\frac{2}{1-a}\partial_{\tau}a+2b+4(2a-1)  & = \cO(\beta^{\frac{5}{2}}),\label{eq:a}\\
\partial_{\tau}b+b^2  & = \cO(\beta^{\frac{5}{2}}),\label{eq:b}\\
\partial_{\tau}\beta_0  & = \Omega_{1}(a,b) \beta_0+\cO(\beta^{\frac{5}{2}}),\label{eq:26in}\\ \nonumber\\
\partial_{\tau}\beta_1  & = a \beta_1+(\mathrm{Small})_1,\label{eq:27in}\\
\partial_{\tau}\beta_2  & = a \beta_2+(\mathrm{Small})_2,\nonumber\\
\partial_{\tau}\beta_3  & = (\mathrm{Small})_3, \nonumber \\  % \label{eq:29in}\\
\partial_{\tau}\beta_4  & = (\mathrm{Small})_4, \nonumber
\end{align}
where the terms $(\mathrm{Small})_{\ell}$, $(\ell=1,\dots,4)$, satisfy
$$|(\mathrm{Small})_{\ell}|\ls\min\Big\{\beta^{\frac{16}{5}}, \beta^{\frac{18}{5}}
+\beta^{\frac{31}{20}}\big\|\langle y\rangle^{-5} \partial_{\theta}\xi e^{\frac{ay^2}{4}}\big\|_{L^\infty}\Big\}.$$

Furthermore, one has
\begin{align}	\label{eq:abab}
\Big|2a-1+\frac{1}{2}b \Big| \ls \beta^2\qquad\text{and}\qquad
|b-\beta|\ls  \beta^{\frac{3}{2}},
\end{align}
(i.e.~$A+B\ls1$) and
\begin{align}
|\beta_0| & \ls  \beta^{\frac{5}{2}}\label{eq:estbeta0},\\
|\beta_1|+\ |\beta_2|& \ls \min\Big\{\beta^{\frac{16}{5}},\ (1+M_4) \beta^{\frac{18}{5}}\Big\},\label{eq:beta01234}\\
|\beta_3|+\ |\beta_4|& \ls  \min\Big\{\beta^{\frac{11}{5}},\ (1+M_4) \beta^{\frac{13}{5}}\Big\}.\nonumber
\end{align}
\end{lemma}

The function $M_{4}$ above is defined as 
$$
M_{4}(\tau):=\max_{0\leq s\leq \tau}\beta^{-\frac{21}{10}}
\Big\|\langle y\rangle^{-5}\partial_{\theta}\xi(\tau) e^{\frac{ay^2}{4}} \Big\|_{L^\infty}.
$$
We prove the lemma in Appendix~\ref{Sec:Splitting}.

\subsection{The infinite-dimensional part of the decomposition}	\label{SmallComponent}
We next show how to control the infinite-dimensional part $\xi$. Many of the arguments here
are virtually identical to their counterparts in \cite{GKS11}. Before providing details, we
again begin with a heuristic outline of the main ideas behind the proof.

We apply Duhamel's principle to equation~\eqref{eq:xi} for $\xi$ to obtain
$$
\xi(\tau)=U(\tau,0)\xi(0)+\int_{0}^{\tau} U(\tau,\sigma)\ \cdots\,\mathrm{d}\sigma,
$$
where $U(\tau,\sigma)$ is the propagator generated by the linear operator $-L(a,b)$ from time $\sigma$ to $\tau$.
To prove the desired decay estimates for $\xi$, we exploit the decay of $U(\tau,\sigma)$ in a suitable subspace of $L^{\infty}$. To achieve this, we must overcome two difficulties. \textsc{(i)} The generator of $U(\tau,\sigma)$
namely $L(a,b)$, has a seven-dimensional nonpositive eigenspace, spanned by $e^{-\frac{a}4 y^2}$, 
$ye^{-\frac{a}4 y^2}$, $(ay^2-1)e^{-\frac{a}4 y^2}$, $(\cos\theta)e^{-\frac{a}4 y^2}$,
$(\sin\theta)e^{-\frac{a}4 y^2}$, $(y\cos\theta)e^{-\frac{a}4 y^2}$, and $(y\sin\theta)e^{-\frac{a}4 y^2}.$
This implies that $U(\tau,\sigma)$ may grow in these directions. \textsc{(ii)} The operator $L(a,b)$ is not
autonomous, so $U(\tau,\sigma)\not= e^{-tL(a,b)}$.

We overcome the first difficulty by using the fact that $\xi$ is orthogonal to the unstable eigenspace.
At least intuitively, one expects $U(\tau,\sigma)$ to decay exponentially fast in directions orthogonal
to that space. To overcome the second difficulty, we make another gauge change, reparameterizing
the function $\xi$ to obtain a new function whose evolution is dominated by an autonomous operator.
\medskip

Using these ideas, one first proves the following.

\begin{lemma}\label{Prop:Gam12}
For all times that the assumptions of Section~\ref{SBMI} hold, one has
\[
|M_{m,n}|\ls 1,\quad \text{for}\quad (m,n)\in\big\{(3,0), \ (11/10,0),\ (2,1),\ (1,1)\big\}.
\]
\end{lemma}
The proof is essentially identical to that of Proposition~9.2 from~\cite{GKS11}, because it relies only
on the outputs of the first bootstrap machine, hence does not need the discrete symmetry assumptions
that were in force there. So we omit it here.
\medskip

Next we turn to  estimating the functions $v_{\pm}$ and $(\dy v)_{\pm}$ defined by
formula~\eqref{DefinePM}. These estimates represent a departure from \cite{GKS11}.
In what follows, we present the difficulties and the ideas used to overcome them when
estimating $v_{\pm}.$ The arguments for  $(\dy v)_{\pm}$ are very similar, hence omitted.
Our objective is to prove the following.

\begin{lemma}	\label{LemmaB}
For all times that the assumptions of Section~\ref{SBMI} hold, $M_4\ls1$, i.e.
\begin{equation}\label{eq:higherWeight}
\Big\|\langle y\rangle^{-5}\partial_{\theta}\xi e^{\frac{ay^2}{4}}\Big\|_{L^\infty}\ls\beta^{\frac{21}{10}}.
\end{equation}
\end{lemma}

\begin{lemma}	\label{LemmaC}
For all times that the assumptions of Section~\ref{SBMI} hold, in the region $\beta y^2\leq 20$,
one has
\begin{equation}	\label{New-xi-estimate}
\Big|e^{\frac{ay^2}{4}}\xi_{\pm}\Big| + \Big|(\partial_{y}e^{\frac{ay^2}{4}}\xi)_{\pm}\Big|
\ls \beta^{\frac{21}{10}}.
\end{equation}
\end{lemma}

Lemmas~\ref{LemmaB} and \ref{LemmaC} are proved in Appendices~\ref{sec:higherWei} and \ref{sec:proof723},
respectively.
\medskip

The difficulties encountered in their proofs are as follows. The decomposition of $v$ implies that
$$
|v_{\pm}|\leq |\beta_1|+|\beta_2|+|\beta_3y |+|\beta_4y|+|\xi_{\pm}|e^{\frac{ay^2}{4}},
$$
where the $\theta$-independent functions $\xi_{\pm}$ are defined using ~\eqref{DefinePM}.
For the purposes of Lemmas~\ref{LemmaB} and \ref{LemmaC}, the estimates
$|\beta_3|+|\beta_4|=\cO(\beta^{\frac{11}{5}})$ from ~\eqref{eq:beta01234} are obviously not good enough.
In the region $by^2\leq 20$, they only give $|\beta_{k}y|\ls\beta^{\frac{8}{5}}$, which is significantly slower than the 
desired $\beta^{\frac{21}{10}}$ decay. The estimates for $\xi$ in Lemma~\ref{Prop:Gam12} are not good enough 
either. As noted above, the reason we want to prove stronger decay is so that the first bootstrap machine can function 
without needing discrete symmetry hypotheses. So we must derive improved estimates for  $|\partial_{\theta}\xi|$.
But it is not difficult to see that $\partial_{\theta}\xi$ admits better decay estimates than does $\xi$. 
Differentiation allows us to  remove the slowly-decaying $\theta$-independent components of the
solution, and thereby to obtain improved estimates.

In the remainder of this section, we outline the main ideas used in proving estimate~\eqref{New-xi-estimate} in
Appendix~\ref{sec:proof723}. The steps used to prove \eqref{eq:higherWeight} in Appendix~\ref{sec:higherWei}
are similar, hence will not be discussed here.

Using~\eqref{eq:xi} and recalling definition \eqref{DefinePM} , we compute that
$$
\partial_{\tau}\xi_{\pm}=-L_{0}(a)\xi_{\pm}+\cdots,
$$
where 
$$L_0(a) := - \dy^2 + \frac{a^2 + \partial_\tau a }4 y^2 -\frac{3a}2.$$
Our strategy for estimating $\xi_{\pm}$ is similar to that used to estimate $\xi$ in the proof of Lemma~\ref{Prop:Gam12}, 
and hence has similar difficulties: the linear operator $L_{0}(a)$ is time-dependent and has nonpositive eigenvalues
with eigenvectors $e^{-\frac{ay^2}{4}}$ and $ye^{-\frac{ay^2}{4}}.$ To take care of these eigendirections, we use the 
orthogonality conditions imposed on $\xi$ in ~\eqref{eq:orthoxi} to see that
$\xi_{\pm}\perp\mathrm{span}\{ e^{-ay^2/4},\ ye^{-ay^2/4}\}$.
The intuition behind the argument is that we obtain good estimates for $\xi_{\pm}$ by applying $L_0(a)$ to
the orthogonal complement of the finite-dimensional unstable subspace. We make this rigorous in
Appendix~\ref{sec:proof723}. Once that work is done, our final result follows readily:

\begin{proof}[Proof of estimate~\eqref{eq:vpmEst} in Theorem ~\ref{THM:aprior}]

We use the decomposition of $v$ in ~\eqref{eqn:split2} to relate $\xi_{\pm}$ to $v_{\pm}$, obtaining
\begin{equation*}
|v_{\pm}|\leq |\beta_1|+|\beta_2|+|\beta_3y |+|\beta_4y|+|\xi_{\pm}|e^{\frac{ay^2}{4}}.
\end{equation*}
The estimate for $\xi_{\pm}$ in Lemma~\ref{LemmaC} and the estimates for $\beta_k$ $(k=1,\dots,4)$ in
Lemma~\ref{LM:traj} then imply that
$|v_{\pm}(y,\theta,\tau)|\ls \beta^{\frac{21}{10}}(\tau)$.

The estimate for  $(\partial_{y}v)_{\pm}$ is obtained similarly.
\end{proof}

\section{Proof of the Main Theorem}	\label{BigFinish}
In this section, we collect the remaining arguments needed to complete the proof of our Main Theorem,
modulo the technical details collected in the appendices.
\begin{proof}[Proof of the Main Theorem]
By placing an Angenent self-similarly shrinking torus around the (approximate) center of the neck, one
sees easily that the solution must become singular before some time $T^*<\infty$. So suppose that $[0,T_*)$
is the maximal time interval such that for any time $t\in[0,T_*)\subseteq[0,T^*)$, we can construct an optimal
coordinate system in which
\begin{equation}	\label{OptimalDecomposition}
u(x,\theta,t)=\lambda_\opt(t)\left\{\sqrt{\frac{2+b_\opt(t) y^2}{1+\frac{1}{2}b_\opt(t)}}
+\phi_\opt(y,\theta,t)\right\},
\end{equation}
where $y=\lambda_\opt^{-1}x$, and $\phi$ satisfies the orthogonality conditions of
Definition~\ref{OptimalCoordinates}, along with the estimate
$$
\big\|\langle y\rangle^{-3} \phi_\opt(\cdot,\cdot,t)\big\|_{L^\infty}\ls b_\opt^{8/5}(t).
$$
In this case, we claim that for any sequence of times $t_n\nearrow T_*$ at which
we construct optimal coordinate systems, one has $b_\opt(t_n)\rightarrow0$ and
$\lambda_\opt(t_n)\rightarrow 0$ as $n\rightarrow\infty$. By the estimates
in Theorem~\ref{THM:aprior} (see also Corollaries~\ref{thm:estimateT0}--\ref{prop:differenceCoor})
for the components of the decomposition~\eqref{OptimalDecomposition}, this
implies that the surface must become singular as $t\nearrow T_*$.

We prove the claim by contradiction, showing that if either quantity $b_\opt(t_n)$
or $\lambda_\opt(t_n)$ has a positive lower bound, then the other does also, which
implies that the solution can be extended past $T_*$, contradicting the assumption that
$T_*$ is maximal. We provide a detailed argument for the case that there exists a constant $c_\infty>0$
such that $b_\opt(t_n)\geq c_\infty$. (An analogous argument works if
 $\lambda_\opt(t_n)\geq c'_\infty$.) 

With respect to an optimal coordinate system constructed at $t_n$, there exists a time interval
$[t_{n-1},t_n]\subseteq[0,t_n]$ in which the solution $v=\lambda^{-1}u$ can be parameterized
as in \eqref{tilde-phi}. Using the fact that $\dt\log\lambda=-a$, and the boundary conditions
stipulated in Definition~\ref{OptimalCoordinates} that ensure that $\lambda_\opt=\lambda$
at any times at which we construct optimal coordinates, we see that
$$
\lambda_\opt(t_n)=e^{-\int_{t_{n-1}}^{t_n}a(\hat\tau)\,\mathrm d\hat\tau}\,\lambda_\opt(t_{n-1}).
$$
Then using the upper bound for $b$ in Corollary~\ref{thm:estimateT0}, the consequence of
estimate~\eqref{eq:abab} that $a=\frac12+\mathcal O(b)$ is bounded independently of $n$,
and the upper bound $T_*\leq T^*$, we conclude that there exists $c>0$
independent of $n$ such that $\lambda_\opt(t_n)\geq c\lambda_\opt(0)$. Because
$n$ was arbitrary, this contradicts the maximality of $T_*$ and proves the claim.
%% As a consequence, one sees that our bootstrap iteration may be continued up to the singular time.

Part~\textsc{(i)} of the theorem follows directly from the claim.

Part~\textsc{(ii)} then
follows from Corollary~\ref{prop:differenceCoor}, because the claim implies that
we can construct optimal coordinate systems up to the singular time.

Proving Part~\textsc{(iii)} takes more work.
Obtaining asymptotics for the sequential parameters $\lambda_\opt(t_n)$ and $b_\opt(t_n)$
that determine the ``main components'' of the solution's decomposition is complicated by the
fact that we do not have \textsc{ode} for them; we only have \textsc{ode} for the quantities
$\lambda$ and $b$ that depend smoothly on the sequential choices of $\lambda_\opt(t_n)$
and $b_\opt(t_n)$. So we proceed as follows.
Working in a time interval $[t_n,t_m]$, we use the relation $\lambda\partial_t\lambda=-a$
to see that
$$
\lambda^2(t)=\lambda_\opt^2(t_m)+\int_t^{t_m}2a(s)\,\mathrm ds.
$$
By the estimates for $a$ in Corollary~\ref{thm:estimateT0} and for $b$ in
Corollary~\ref{prop:differenceCoor}, this implies that
$$
\lambda_\opt^2(t_n)=\big(1+\mathcal O(b_\opt(t_m)\big)
\big\{\lambda_\opt^2(t_m)+(t_m-t_n)\big\}.
$$
Letting $m\rightarrow\infty$ and using the fact proved above that
$\lambda_\opt(t_m)\searrow0$ as $m\rightarrow\infty$, we 
conclude that the asymptotic behavior of $\lambda_\opt(t_n)$ as $n\rightarrow\infty$ is
\begin{equation}	\label{lambdaopt-asymptotics}
\lambda_\opt(t_n)=\big(1+o(1)\big)\sqrt{T-t_n}.
\end{equation}
%%To obtain the asymptotic behavior of $b_\opt(t_n)$, we first analyze the $\tau$ timescale.
Next, by the estimate for $a$ in Corollary~\ref{thm:estimateT0} and the fact that
$\frac{d\tau}{d\lambda}=-\frac{1}{a}$, we find that as $n\rightarrow\infty$, one has
\[
\tau(t_n)=\big(1+o(1)\big)\log\frac{1}{T-t_n}.
\]
Then using the estimates for $b$ and $b_\opt$ in Corollaries~\ref{thm:estimateT0}--\ref{prop:differenceCoor},
it follows that as $n\rightarrow\infty$,
\begin{equation}	\label{bopt-asymptotics}
b_\opt(t_n)=\big(1+o(1)\big)\left(\log\frac{1}{T-t_n}\right)^{-1}.
\end{equation}
Equations~\eqref{lambdaopt-asymptotics} and \eqref{bopt-asymptotics}
establish Part~\textsc{(iii)} of the theorem.

Part~\textsc{(iv)} of the theorem follows directly from the outputs of the first and
second bootstrap machines, as stated in Sections~\ref{FirstOutput} and \ref{SecondOutput}, respectively.

The proof is complete.
\end{proof}

\appendix

\section{Proof of Lemma~\ref{LM:traj}}\label{Sec:Splitting}
In this appendix, we prove Lemma ~\ref{LM:traj}. To avoid unenlightening repetition,
we provide detailed arguments only for the $a$, $b$, $\beta_0$, and $\beta_3$
evolution equations and their estimates. The arguments for the others are almost identical.

\subsection{Proofs of estimates~\eqref{eq:a} and \eqref{eq:b}}
The derivations of \eqref{eq:a} and ~\eqref{eq:b} are almost identical to
those in our previous work \cite{GKS11}, hence are only sketched here.

We rewrite $F_{1}(a,b)$ in the form
\begin{align*}
F_{1}(a,b) & = e^{-\frac{ay^2}{4}}\frac{1}{\sqrt{2-2a}\sqrt{2+by^2}}
\left[\frac{2b}{2+by^2}+4a-2-\frac{1}{2}b_{\tau}y^2-\frac{1}{1-a}a_{\tau}-\frac{a_{\tau}b y^2}{2-2a}\right]\\
& = e^{-\frac{ay^2}{4}}\frac{1}{\sqrt{2-2a}\sqrt{2+by^2}} 
\left[\Gamma_1-\frac{1}{2}\Gamma_2 y^2+\Gamma_3\right],
\end{align*}
with $\Gamma_{k}$ $(k=1,2,3)$ defined as
\begin{align*}
\Gamma_1&:=b+4a-2-\frac{1}{1-a} a_{\tau},\\
\Gamma_2&:=b^2+\partial_{\tau}b,\\
\Gamma_3&:=\frac{b^3 y^4}{2(2+by^2)}-\frac{a_{\tau}b y^2}{2-2a}.
\end{align*}
Then we take an inner product of \eqref{eq:xi} with the functions $e^{-\frac{a}{4}y^2}$ and $(ay^2-1)e^{-\frac{a}{4}y^2}$, 
applying the orthogonality conditions of \eqref{eq:orthoxi} to obtain
\begin{align}\label{eq:gamma1}
|\Gamma_1|+|\Gamma_2|\ls \beta^{\frac{5}{2}}
\end{align} 
Recalling the definitions of $\Gamma_1$ and $\Gamma_2$ above, it is easy to see that this estimate
implies \eqref{eq:a} and \eqref{eq:b}.

\subsection{Proof of the estimates in \eqref{eq:abab}}

\begin{proof}[Proof of the estimates in \eqref{eq:abab}]
We write $\Gamma_1$ as
\begin{equation*}
\Gamma_1  =\frac{1}{1-a} \left[2\tilde\Gamma_1-\partial_{\tau}\tilde\Gamma_1
+\frac{1}{4}\partial_{\tau}b-2\tilde\Gamma_1^2+\frac{1}{8}b^2\right],
\end{equation*} 
with $\tilde\Gamma_1:=a-\frac{1}{2}+\frac{1}{4} b$.
To estimate the various components, we use the assumptions
$$
\max_{s\leq \tau}|b(s)-\beta(s)|\beta^{-\frac{3}{2}}(s)=B(\tau)\ls \beta^{-\frac{1}{20}}(\tau)
$$
and
$$
\max_{s\leq \tau}\left|a-\frac{1}{2}+\frac{1}{4}b\right|\beta^{-\frac{5}{2}}(s)
$$
to obtain $b\leq 2\beta$ and $|\tilde\Gamma_1 |\ls \beta^2.$ Hence we get
\begin{equation}\label{eq:BGamma2}
|\partial_{\tau}\tilde\Gamma_1-2\tilde\Gamma_1|\ls \beta^2.
\end{equation}  

By the boundary condition $a(\tau(t_{\#}))=\frac{1}{2}-\frac{1}{4}b(\tau(t_{\#}))$ in \eqref{eq:boundary}, we have
\begin{align*}
\tilde\Gamma_1(\tau_{\#})=0,
\end{align*} where we denote $\tau(t_{\#})$ by $\tau_{\#}$.
We rewrite \eqref{eq:BGamma2} as
\begin{equation}\label{eq:gamma2}
\begin{split}
|\tilde\Gamma_1(\tau)| &\ls  \int_{\tau}^{\tau_{\#}} e^{2(\tau-s)}\beta^{2}(s)\,\mathrm ds\\
&\ls  \beta^{2}(\tau)\int_{\tau}^{\tau_{\#}}e^{2(\tau-s)}\,\mathrm ds\\
&\ls  \beta^{2}(\tau),
\end{split}
\end{equation}
and then we use the definition of $A$ in \eqref{eq:majorA} to obtain the first estimate
in \eqref{eq:abab}.

The argument used to estimate $b$ is identical to the corresponding argument from~\cite{GKS11}.
We rewrite $\Gamma_2$ and use estimate~\eqref{eq:gamma1} to see that
\begin{align*}
\Big|\partial_{\tau}\frac{1}{b}-1\Big| = \Big|\partial_{\tau}\Big[\frac{1}{b}-\frac{1}{\beta}\Big]\Big|
\ls \beta^{\frac{1}{2}}.
\end{align*} 
Recall that $\beta:=\big(b_0^{-1}+\tau\big)^{-1}$.
It is clear that $\big[\frac{1}{b}-\frac{1}{\beta}\big]_{\tau=0}=0,$ and consequently that
\begin{align*}
\Big|\frac{1}{b}-\frac{1}{\beta}\Big|(\tau)
\ls \int_0^{\tau} \beta^{\frac{1}{2}}(\tau)\ d\tau\ls \beta^{-\frac{1}{2}}(\tau).
\end{align*} 
Finally, recalling the definition of $B$ in ~\eqref{eq:majorB}, we obtain
$|b(\tau)-\beta(\tau)|\ls \beta^{\frac{3}{2}}(\tau)$ and $B\ls1$.
\end{proof}

The following facts will be used frequently in the rest of the paper.

\begin{lemma} 
For all times that the assumptions in Section~\ref{SBMI} hold, one has
\begin{equation}\label{eq:ataubtau}
|a_{\tau}|+|b_{\tau}|\ls \beta^2,
\end{equation}
and 
\begin{equation}   \label{eq:bbetaRatio}
1-\beta^{\frac{1}{4}}\leq \frac{b}{\beta}\leq 1+\beta^{\frac{1}{4}}.
\end{equation}
\end{lemma}

\begin{proof}
To obtain estimate~\eqref{eq:ataubtau}, we combine the estimate for $b_{\tau}$ from \eqref{eq:gamma1}, the estimate
for $a_{\tau}$ from \eqref{eq:BGamma2}, and the estimate $\tilde\Gamma_1=\cO(\beta^2)$ from \eqref{eq:gamma2}.

Estimate~\eqref{eq:bbetaRatio} follows from the observation that
$$
\max_{s\leq \tau}|b(s)-\beta(s)|\beta^{-\frac{3}{2}}(s)=B(\tau)\ls 1.
$$
\end{proof}

\subsection{Proofs of estimates~\eqref{eq:26in} and ~\eqref{eq:estbeta0}}		\label{Cancellations}

We start with proving ~\eqref{eq:26in}, using methods outlined in the discussion
before Lemma~\ref{LM:traj}.

\begin{proof}[Proof of estimate~\eqref{eq:26in}]
Following the approach outlined in our introduction to Lemma~\ref{LM:traj}, we take the
inner product of \eqref{eq:xi} with $ye^{-\frac{ay^2}{4}}$  to obtain
\begin{equation}\label{eq:beta0}
\left\langle ye^{-\frac{ay^2}{4}},\ \partial_{\tau}\xi\right\rangle = \sum_{k=1}^{5}A_{k},
\end{equation}
where the terms on the \textsc{rhs} are defined by
\begin{align*}
A_{1}:=&- \left\langle ye^{-\frac{ay^2}{4}},\ \left[-\partial_{y}^2+\frac{a^2+a_{\tau}}{4}y^2-\frac{3a}{2}-\frac{1}{2}\partial_{\theta}^2-\frac{\frac{1}{2}+a}{2+by^2} \right]\xi \right\rangle
+ \left\langle ye^{-\frac{ay^2}{4}},\ N_{3,1} \right\rangle\\
=&- \left\langle ye^{-\frac{ay^2}{4}},\ \left[-\partial_{y}^2+\frac{a^2+a_{\tau}}{4}y^2-\frac{3a}{2}
-V_{a,b}^{-2}\partial_{\theta}^2-\frac{\frac{1}{2}+a}{2+by^2}\right]\xi \right\rangle,
\end{align*}
and
\begin{align*}
A_2 & := \Big\langle ye^{-\frac{ay^2}{4}}, \ F(a,b,\beta) \Big\rangle,\\
A_3 & := \Big\langle y e^{-\frac{ay^2}{4}}, \ N_1 \Big\rangle,\\
A_4 & := \Big\langle ye^{-\frac{ay^2}{4}}, \ N_2 \Big\rangle,\\
A_5 & :=  \Big\langle ye^{-\frac{ay^2}{4}}, N_{3,2} \Big\rangle.
\end{align*}
We now show how the terms in equation~\eqref{eq:beta0} are controlled.

For the \textsc{lhs} of \eqref{eq:beta0}, we use the orthogonality condition
$\xi\perp  ye^{-\frac{ay^2}{4}}$ in \eqref{eq:orthoxi} to obtain
\begin{align*}
\Big\langle ye^{-\frac{ay^2}{4}},\ \partial_{\tau}\xi \Big\rangle
&= \partial_{\tau} \Big\langle ye^{-\frac{ay^2}{4}},\ \xi \Big\rangle
+ \frac{a_{\tau}}{4}\Big\langle y^3e^{-\frac{ay^2}{4}},\ \xi\Big\rangle\\
&=\cO\Big(|a_{\tau} |\|\langle y\rangle^{-3}e^{\frac{ay^2}{4}}\xi\|_{L^\infty}\Big)\\
&= \cO(\beta^{\frac{71}{20}}).
\end{align*}
Then we use the estimate $|a_{\tau}|=\cO(\beta^2)$ from \eqref{eq:ataubtau}, and the assumption that
$M_{3,0}\ls \beta^{-\frac{1}{20}}$ to obtain
$$\Big\|\langle y\rangle^{-3} e^{\frac{ay^2}{4}}\xi\Big\|_{L^\infty}
\ls \beta^{\frac{8}{5}}M_{3,0}\ls \beta^{\frac{31}{20}}.$$

In $A_1$, many terms cancel because of the fact that $\xi\perp ye^{-\frac{ay^2}{4}}$, where
$ye^{-\frac{ay^2}{4}}$ is an eigenfunction  of the self-adjoint operator
$-\partial_{y}^2+\frac{a^2}{4}y^2-V_{a,b}^{-2}\partial_{\theta}^2$. We compute that
\begin{align*}
A_1 & = - \Big\langle ye^{-\frac{ay^2}{4}},\ [-\partial_{y}^2 
+\frac{a^2}{4}y^2-\frac{3a}{2}-V_{a,b}^{-2}\partial_{\theta}^2
-\frac{\frac{1}{2}+a}{2}]\xi \Big\rangle\\
&\qquad - \Big\langle ye^{-\frac{ay^2}{4}},\ \frac{a_{\tau}}{4}y^2\xi \Big\rangle
+ \Big\langle ye^{-\frac{ay^2}{4}}, \Big[\frac{\frac{1}{2}+a}{2+by^2}
-\frac{\frac{1}{2}+a}{2}\Big] \xi \Big\rangle\\
& = - \Big\langle y e^{-\frac{ay^2}{4}},\ \frac{a_{\tau}}{4}y^2\xi \Big\rangle
- \Big(\frac{1}{2}+a\Big)b \Big\langle y e^{-\frac{ay^2}{4}},\ \frac{y^2}{2(2+by^2)}\xi \Big\rangle.
\end{align*}
Hence we conclude that
$$
|A_1|\ls (|a_{\tau}|+b)\|\langle y\rangle^{-3} e^{\frac{ay^2}{4}}\xi\|_{L^\infty}\ls \beta^{\frac{51}{20}}.
$$

The expression for $A_2$ can be simplified by observing that $ye^{-\frac{ay^2}{4}}$ is odd in $y$ and independent of
$\theta$, hence is orthogonal to the functions even in $y$ or $\theta$-dependent. Thus we compute that
\begin{align*}
A_2 &=  \Big\langle ye^{-\frac{ay^2}{4}}, \ F_1(a,b)+F_2(a,b,\beta) \Big\rangle\\
& = \Big\langle ye^{-\frac{ay^2}{4}}, \ e^{-\frac{ay^2}{4}} \Big[\partial_{y}^2-ay\partial_{y} + a
+ V_{a,b}^{-2}-\partial_{\tau}\Big]\beta_0 y\Big\rangle.
\end{align*}
Consequently, we get
$$
A_2
= 2\pi \big[\Omega_1(a,b)\beta_0-\partial_{\tau}\beta_0\big]\int_{-\infty}^{\infty}y^2e^{-\frac{ay^2}{2}}\,\mathrm{d}y,
$$
where $\Omega_{1}(a,b)$ is the constant defined in \eqref{eq:Omega1}.

For $A_3$, we use the definition of $\tilde{\xi}$ in \eqref{eq:defN1} and the assumptions in Condition~[Cb]
that $|\beta_{k}|\ls\beta^2$ for $k=0,1,\dots, 4$ to obtain
\begin{equation}\label{eq:estA3}
\begin{split}
A_3
&\ls  \Big\|\langle y\rangle^{-6} e^{\frac{ay^2}{2}}\tilde{\xi}^2\Big\|_{L^\infty} \\
&\ls  \Big\|\langle y\rangle^{-3} e^{\frac{ay^2}{2}}\xi^2\Big\|_{L^\infty}^2+\sum_{k=0}^4|\beta_{k}|^2 \\
&\ls  \beta^{3}.
\end{split}
\end{equation} 

For $A_4$, direct calculation yields
$$|A_{4}|\ls \|\langle y\rangle^{-1}p\|_{L^\infty}^2 \|\partial_{y}^2 v\|_{L^\infty}
=\|\langle y\rangle^{-1}\partial_{y}v\|_{L^\infty}^2 \|\partial_{y}^2 v\|_{L^\infty}. $$
The decomposition of $v$ implies that
$$
\partial_{y}v=\frac{by}{\sqrt{2-2a}\sqrt{2+by^2}}+\beta_0+\beta_3\cos\theta
+\beta_{4}\sin\theta+ \partial_{y}\Big\{e^{\frac{ay^2}{2}}\xi\Big\}.
$$
Hence we get
\begin{equation}\label{eq:y-1yv}
\big\|\langle y\rangle^{-1}\partial_{y}v\big\|_{L^\infty}
\leq  b+|\beta_0|+|\beta_3|+|\beta_{4}| + \Big\|\langle y\rangle^{-1}\partial_{y}
\Big[e^{\frac{ay^2}{2}}\xi \Big] \Big\|_{L^\infty}
\ls  \beta.
\end{equation}
Recall that $\|\langle y\rangle^{-1}\partial_{y}[e^{\frac{ay^2}{2}}\xi]\|_{L^\infty}=\|e^{\frac{ay^2}{2}}\xi\|_{1,1}$, 
and that by assumption on $M_{1,1} $ we have
$\|e^{\frac{ay^2}{4}}\xi\|_{1,1}\ls \beta^{\frac{21}{20}}$.
This, together with the estimate $|\partial_{y}^2 v|=\cO(\beta^{\frac{13}{20}})$ in ~\eqref{eq:yDe} implies that
\begin{equation*}
|A_{4}| \ls \beta^{\frac{53}{20}}.
\end{equation*}

We use the definition of the quantity $N_{3,2}$ introduced in \eqref{eq:difN3} to 
decompose $A_5$ into four terms,
\begin{equation}\label{eq:A5}
\begin{split}
A_5 
& = \Big\langle ye^{-\frac{ay^2}{4}},\ \Big[v^{-2}-V_{a,b}^{-2}\Big] \partial_{\theta}^2 
\xi e^{-\frac{ay^2}{4}}\Big\rangle\\
&\qquad 
- \Big\langle ye^{-\frac{ay^2}{4}},\ v^{-2}\frac{q^2}{1+p^2+q^2} \partial_{\theta}^2 \xi e^{-\frac{ay^2}{4}}
\Big\rangle\\
&\qquad 
+ \Big\langle ye^{-\frac{ay^2}{4}},\ v^{-1}\frac{2pq}{1+p^2+q^2}\partial_{\theta}
\partial_{y}v e^{-\frac{ay^2}{4}}\Big\rangle\\
&\qquad 
+ \Big\langle ye^{-\frac{ay^2}{4}},\ v^{-2}\frac{q}{1+p^2+q^2}\partial_{\theta}v e^{-\frac{ay^2}{4}}\Big\rangle\\
& = \sum_{\ell=1}^{4}A_{5,\ell},	
\end{split}
\end{equation}
where the various terms $A_{5,\ell}$ $(\ell=1,\dots,4)$ are naturally defined.

For $A_{5,1}$, we integrate by parts in $\theta$ to remove the slowly decaying $\theta$-independent
components in $v^{-2}-V_{a,b}^{-2}$. Thus we compute that
$$A_{5,1}=2\Big\langle ye^{-\frac{ay^2}{2}},\ v^{-3}\partial_{\theta}v \partial_{\theta}\xi \Big\rangle$$
and hence can estimate
\begin{equation*}
|A_{5,1}|\ls \|v^{-2} \partial_{\theta}v\|_{L^\infty}   \Big\|\langle y\rangle^{-3}e^{\frac{ay^2}{4}}
\partial_{\theta} \xi \Big\|_{L^\infty}.
\end{equation*}
We relate $e^{\frac{ay^2}{4}}\partial_{\theta} \xi$ to $\partial_{\theta}v$ using the decomposition
of $v$ in \eqref{eqn:split2}, obtaining
\begin{equation}\label{eq:xieta}
\begin{split}
\langle y\rangle^{-2}e^{\frac{ay^2}{4}} |\partial_{\theta} \xi | 
&\leq  \langle y\rangle^{-2}|\partial_{\theta} v|+\sum_{k=1}^{4}|\beta_k| \\
& \ls   v^{-2}|\partial_{\theta} v|+\sum_{k=1}^{4}|\beta_k| \\
& \ls   \beta^{\frac{33}{20}}.
\end{split}
\end{equation}
Here we used the estimate $v^{-2}|\partial_{\theta}v| \ls \beta^{\frac{33}{20}}$ from \eqref{eq:thetaDe}
along with the fact that
\begin{equation}\label{eq:y-1}
\langle y\rangle^{-1}\ls v^{-1},\qquad\text{equivalently,}\qquad v\ls \langle y\rangle.
\end{equation}
This, in turn, is implied by three facts: the computation that
$$
|v(y,\theta,\tau)|\leq |v(0,\theta,\tau)|+\int_{0}^{y}|\partial_{z}v(z,\theta,\tau)|\,\mathrm{d}z,
$$
the consequence  $v(0,\theta,\tau)\ls 1$ of our input assumption that $M_{3,0}\ls\beta^{-\frac{1}{20}}$,
and the assumption that $|\partial_{y}v|\ls 1$ from \eqref{eq:upperBounddV}.
We collect the estimates above to obtain
\begin{equation*}
|A_{5,1}|\ls \beta^{\frac{33}{10}}.
\end{equation*}

Turning to $\sum_{\ell=2}^{4}A_{5,\ell}$, we observe that each of these terms
contains a rapidly decaying factor $q=v^{-1}\partial_{\theta}v$. We apply ~\eqref{eq:y-1} again to get
\begin{equation*}  %\label{eq:qest}
\|\langle y\rangle^{-1}q\|_{L^\infty}\ls \|v^{-2}\partial_{\theta}v\|_{L^\infty} \ls \beta^{\frac{33}{20}}.
\end{equation*}
This, together with the estimates in \eqref{eq:thetaDe}, implies that
\begin{equation*}
\bigg|\sum_{\ell=2}^{4} A_{5,\ell}\bigg|  \ls \beta^{\frac{16}{5}}.
\end{equation*}

Collecting the estimates above completes the proof of estimate~\eqref{eq:26in}.
\end{proof}

\begin{proof}[Proof of estimate~\eqref{eq:estbeta0}]
Let $\tau_{\#}$ denote $\tau(t_{\#})$, and recall the boundary condition
$\beta_{0}(\tau_{\#})=0$ from \eqref{eq:boundary}. We use \eqref{eq:26in} to write
\begin{equation*}
|\beta_0(\tau)|\ls \int_{\tau}^{\tau_{\#}}
e^{-\int_{\kappa}^{\tau_{\#}}\Omega(a,b)(s)\,\mathrm{d}s}\;
\beta^{\frac{13}{5}}(\kappa)\,\mathrm{d}\kappa
\end{equation*}
Then we use the consequence of \eqref{eq:Omega1} that $\Omega(a,b)\geq1/2 $ to conclude that
\begin{equation*}
|\beta_0(\tau)|\ls \beta^{\frac{13}{5}}(\tau).
\end{equation*}
\end{proof}

\subsection{Proofs of estimates~\eqref{eq:27in} and \eqref{eq:beta01234}}

\begin{proof}[Proof of estimate~\eqref{eq:27in}]
As we did in deriving \eqref{eq:beta0}, we take the
inner product of \eqref{eq:xi} with $\cos\theta e^{-\frac{ay^2}{4}}$  to obtain
\begin{equation*}
\Big\langle \cos\theta e^{-\frac{ay^2}{4}},\ \partial_{\tau}\xi \Big\rangle
= \sum_{k=1}^{5}\tilde{A}_{k},
\end{equation*}
where the terms $\tilde{A}_{k}$ are defined like those in \eqref{eq:beta0}, replacing $y$ by $\cos\theta$ where needed.

The estimates for the various terms are very similar to those in the proof of estimate~\eqref{eq:26in}.
The only difference is that the presence of the factor $\cos\theta$ here allows us to integrate by parts in
$\theta$ to remove some slowly-decaying $\theta$-independent parts. In what follows, we estimate 
$\tilde{A}_1$, $\tilde{A}_2$, and $\tilde{A}_4$ in detail to illustrate the main ideas. We omit detailed
proofs of the estimates for the other terms.

For $\tilde{A}_1$, we have
\begin{align*}
\tilde{A}_{1}
& = - \Big\langle \cos\theta\ e^{-\frac{ay^2}{4}}, \ \Big[-\partial_{y}^2+\frac{a^2+a_{\tau}}{4}y^2-\frac{3a}{2}-V_{a,b}^{-2}\partial_{\theta}^2-V_{a,b}^{-2}\Big] \xi \Big\rangle\\
& = - \frac{a_{\tau}}{4} \Big\langle \cos\theta\ y^2 e^{-\frac{ay^2}{4}}, \ \xi \Big\rangle\\
& = \frac{a_{\tau}}{4} \Big\langle \sin\theta\ y^2 e^{-\frac{ay^2}{4}}, \ \partial_{\theta}\xi \Big\rangle.
\end{align*}
In the second step, we used the simple observation
$[V_{a,b}^{-2}\partial_{\theta}^2+V_{a,b}^{-2}]\cos\theta=0$;
and in the last step, we integrated by parts in $\theta$. This calculation directly implies that
\begin{align*}
|\tilde{A}_1|  & \ls |a_{\tau}| \min\Big\{ \Big\|  \langle y\rangle^{-3} e^{\frac{ay^2}{4}}\xi \Big\|_{L^\infty},\;
\Big\|\langle y\rangle^{-5} e^{\frac{ay^2}{4}}\partial_{\theta}\xi \Big\|_{L^\infty}\Big\}\\
& \ls \min\Big\{\beta^{\frac{7}{2}},\;\beta^2  \Big\|  \langle y\rangle^{-5} e^{\frac{ay^2}{4}}
\partial_{\theta}\xi \Big\|_{L^\infty}\Big\}.
\end{align*} 

For $\tilde{A}_2$, one can use symmetries to cancel many terms, as we did in Section~\ref{Cancellations} in the
proof of estimate~\eqref{eq:26in}. Here we get
\begin{align*}
\tilde{A}_2
& = \Big\langle\cos\theta \ e^{-\frac{ay^2}{2}},\ \Big[\partial_{y}^2+V_{a,b}^{-2}\partial_{\theta}^2-ay\partial_{y}+a+V_{a,b}^{-2}-\partial_{\tau}\Big] \beta_1\cos\theta \Big\rangle\\
& = [a\beta_1-\partial_{\tau}\beta_1] \int_0^{2\pi}\int_{-\infty}^{\infty}\cos^2\theta \ e^{-\frac{ay^2}{2}}
\,\mathrm{d}y\,\mathrm{d}\theta.
\end{align*}

For $\tilde{A}_4$, we integrate by parts in $\theta$ to remove the $\theta$-independent components, yielding
$$\tilde{A}_{4} :=  \Big \langle\cos\theta e^{-\frac{ay^2}{4}},\ N_{2} \Big\rangle
= -  \Big\langle\sin\theta e^{-\frac{ay^2}{4}},\ \partial_{\theta}N_2  \Big\rangle.$$
Using the definition of $N_2$, this becomes
\begin{align*}
\tilde{A}_{4} & =-\Big\langle\sin\theta e^{-\frac{ay^2}{4}},\quad e^{-\frac{ay^2}{4}}\frac{2 p\partial_{\theta}p}{1+p^2+q^2}\partial_{y}^2 v-e^{-\frac{ay^2}{4}}\frac{2p^2[p\partial_{\theta}p+q\partial_{\theta}q]}{(1+p^2+q^2)^2}\partial_{y}^2 v\\
& \qquad\qquad\qquad\qquad+e^{-\frac{ay^2}{4}}\frac{p^2}{1+p^2+q^2}\partial_{y}^2\partial_{\theta}v\Big\rangle\\
& = K_1+K_2+K_3,
\end{align*} 
where the terms $K_{\ell}$ $(\ell=1,2,3)$ are naturally defined. 

It is easy to estimate $K_2$ and $K_3$. Indeed, using the estimates in ~\eqref{eq:thetaDe} and the
estimate $\langle y\rangle^{-1}\ls v^{-1}$  from \eqref{eq:y-1}, we get
\begin{equation*}
|K_{\ell}|\ls b^{\frac{18}{5}}
\end{equation*}
for $\ell=2,3$. Estimating $K_1$ is more involved.
We begin by integrating by parts in the $y$ variable, using $2p\partial_{y}^2 v=\partial_{y}(\partial_{y}v)^2$ to obtain
\begin{equation*}
K_1 = \Big\langle \partial_{y} \Big[\sin\theta e^{-\frac{ay^2}{2}} \frac{\partial_{\theta}
\partial_{y}v}{1+p^2+q^2}\Big],\ (\partial_{y}v)^2 \Big\rangle.
\end{equation*}
We bound the term $\langle y\rangle^{-2}(\partial_{y}v)^2$ by $\beta^2$ as in \eqref{eq:y-1yv}.
Then we use \eqref{eq:thetaDe} to see that the terms in
$ \partial_{y}\big\{\sin\theta e^{-\frac{ay^2}{2}} \frac{\partial_{\theta}\partial_{y}v}{1+p^2+q^2}\big\}$ 
are of order $\beta^{\frac{8}{5}}$. Thus we get
\begin{equation*}
|K_1|\ls \beta^{\frac{18}{5}}.
\end{equation*}

We collect the estimates obtained above to conclude that
\begin{equation*}
|\tilde{A}_{4}|\ls \beta^{\frac{18}{5}}.
\end{equation*}

For $\tilde{A}_3$ and $\tilde{A}_5$, we apply similar methods to conclude that
\begin{equation*}
|\tilde{A}_{3}|+|\tilde{A}_5|
\ls \min\Big\{ \beta^{\frac{16}{5}},\;
\beta^{\frac{31}{20}}  \Big\| \langle y\rangle^{-5}e^{\frac{ay^2}{4}}\partial_{\theta}\xi \Big\|_{L^\infty}\Big\}.
\end{equation*}
This completes our proof of estimate~\eqref{eq:27in}.
\end{proof}

\begin{proof}[Proof of estimate~\eqref{eq:beta01234}]
By the boundary condition \eqref{eq:boundary}, we have
$\beta_1(\tau_{\#})=0$, where we again define $\tau_{\#}=\tau(t_{\#})$.
We rewrite estimate~\eqref{eq:27in} in the form
\begin{align*}
|\partial_{\tau}\beta_{1}-a\beta_{1}|\ls \min \Big\{\beta^{\frac{16}{5}},\;(M_4+1)\beta^{\frac{18}{5}}\Big\}.
\end{align*} 
This implies that
\begin{equation*}
|\beta_1(\tau)|\ls \int_{\tau}^{\tau_{\#}} e^{-\int_{\kappa}^{\tau_{\#}}a(s)\,\mathrm{d}s}
\cdot \min \Big\{\beta^{\frac{16}{5}},\;(M_4+1)\beta^{\frac{18}{5}}\Big\}(\kappa)\,\mathrm{d}\kappa.
\end{equation*}
Using the fact that $a\geq1/4$, we thus conclude that
\begin{equation*}
|\beta_{1}(\tau)| \ls \min \Big\{\beta^{\frac{16}{5}}(\tau),\;(M_4(\tau)+1)\beta^{\frac{18}{5}}(\tau)\Big\}.
\end{equation*}
\end{proof}

\section{Proof of Lemma~\ref{LemmaB}}
\label{sec:higherWei}

\begin{proof}[Proof of Lemma~\ref{LemmaB}]
By differentiating equation \eqref{eq:xi}, we find that  $\partial_{\theta}\xi$ evolves by
\begin{equation}\label{eq:thetaXi}
\partial_{\tau}(\dz\xi)=-L_{0}(a)\big[\dz\xi\big]+\sum_{k=1}^{3}D_{k},
\end{equation}
where the terms on the \textsc{rhs} are defined by
\begin{align*}
L_0(a)& := - \dy^2 + \frac{a^2 + \partial_\tau a }4 y^2 -\frac{a}2
- 1 - \frac12 \dz^2,\\
D_{1} & := \left[1-a-\frac{2-2a}{2+by^2}\right]\dz\xi, \\
D_2 & := \partial_{\theta}F(a,b,\beta),\\
D_3 & := \sum_{\ell=1}^{3}\dz N_{\ell}.
\end{align*}
Recall that the nonlinear terms $N_\ell$ appearing above are defined in Section~\ref{SlowlyDecompose}.
\medskip

Our first observations are the following key facts.
\begin{lemma}\label{LM:higherWei}
For all times that  the assumptions in Section~\ref{SBMI} hold, one has the orthogonality condition
\begin{equation}\label{eq:ortho10}
 e^{-\frac{a}4 y^2}\dz\xi\perp \{1,\ y,\  y^2,\cos\theta,\sin\theta,y\cos\theta,y\sin\theta\},
\end{equation}
along with estimates
\begin{equation}\label{eq:higherwei}
\sum_{k=1}^{3}\big\| \langle y\rangle^{-5} e^{\frac{a}{4} y^2} D_{k} \big\|_{L^\infty} \ls \beta^{\frac{13}{5}}.
\end{equation}
\end{lemma}
This lemma is proved in Section~\ref{subsec:remainder} below.
\medskip

In what follows, we pursue the strategies outlined in Section~\ref{SlowlyDecompose}.
Our next step is to put \eqref{eq:thetaXi} into a more convenient form by removing the time
dependency of the linear operator $L_0(a)$, which is what makes it difficult to estimate the
propagator. The parametrization that follows is identical to that used in \cite{GKS11}.

Recall that $\tau(t) := \int_{0}^t  \lambda^{-2} (s)\,\mathrm{d}s$ for any $\tau\ge  0$,
and $a(\tau) := - \lambda (t(\tau)) \partial_t  \lambda(t(\tau))$.
Let $t(\tau)$ be the inverse function of $\tau(t)$, and fix a ($\tau$-timescale) constant $T_1>0$.
We approximate $\lambda (t(\tau))$ on the interval  $0\le\tau\le T_1$ by a new gauge
$\lambda_1 (t(\tau))$, chosen so that $\lambda_1 (t(T_1)) = \lambda (t(T_1))$, and
$\alpha := - \lambda_1 (t(\tau)) \partial_t \lambda_1  (t(\tau)) = a(T_1)$ is constant.

Then we introduce a new spatial variable $z(x,t) := \lambda_1^{-1}(t)x$ and a new time variable
$\sigma(t) := \int_{0}^t  \lambda_1^{-2}(s)\,\mathrm{d}s$, together with a new function
$\eta(z,\theta,\sigma)$ defined by
\begin{equation}\label{NewFun}
\lambda_1 (t) e^{\frac{\alpha}4 z^2} \eta (z,\theta,\sigma)
    :=  \lambda (t) e^{\frac{a(\tau)}4 y^2} \xi (y,\theta,\tau)
    \equiv\lambda (t)\phi(y,\theta,\tau).
\end{equation}
One should keep in mind that the variables $z$ and $y$ are related by
$z/y= \lambda(t) / \lambda_1(t)$.
To relate the time scales $\sigma(t) := \int_{0}^t  \lambda_1^{-2} (s)\,\mathrm{d}s$ and
$\tau = \int_{0}^t \lambda^{-2} (s)\,\mathrm{d}s$, we note that one may for all
$t(\tau)<T_1$ regard $\sigma$ as a function of $\tau$ given by
\begin{equation}\label{eq:scaledTime}
\sigma(\tau) := \int_0^{t(\tau)} \lambda_1^{-2} (s)\,\mathrm{d}s.
\end{equation}
Observing that the function $\sigma$ is invertible, we denote its inverse by $\tau(\sigma)$.

Now we derive an equation for $\dz\eta,$ which is related to $\dz\xi$ by ~\eqref{NewFun}.
One has
\begin{equation}\label{eq:eta}
\partial_\sigma\big[\dz\eta(\sigma)\big]
\ =  - L_0(\alpha) \big[\dz\eta (\sigma)\big]+\tilde{D}_1 +\tilde{D}_2+\tilde{D}_3,
\end{equation}
where the operator $L_0(\alpha)$ is linear and autonomous,
$$ L_0(\alpha) :=
- \partial_z^2  + \frac{\alpha^2}4 z^2 - \frac{\alpha}2 - 1 - \frac12 \dz^2.$$
The remaing terms on the \textsc{rhs} are
\begin{subequations}		\label{eq:difTidD2}
\begin{align}
\tilde{D}_1 & :=  \left[- \frac{\lambda^2}{\lambda_1^2}\
\frac{2-2a}{2+b(\tau(\sigma)) y^2}+ 1-\alpha\right] \partial_{\theta}\eta,\\
\tilde{D}_2 & :=   \frac{\lambda_1}{\lambda}  e^{-\frac{\alpha}4 z^2}
e^{\frac{a}4 y^2}D_2,\\
\tilde{D}_3 & :=   \frac{\lambda_1}{\lambda}  e^{-\frac{\alpha}4 z^2}
e^{\frac{a}4 y^2}D_3.
\end{align}
\end{subequations}
We estimate those terms in the following.

\begin{lemma}	\label{TildeLM:traj}
If the assumptions of Section~\ref{SBMI} hold, then there exists a 
constant $C$ such that one has
\begin{equation*}
\frac{\tau}{C}\leq \sigma(\tau)\leq C\tau,
\end{equation*}
\begin{equation}\label{eq:appTra}
|a-\alpha|+\left|\frac{\lambda_1(t)}{\lambda(t)}-1\right|\leq C \beta(\tau(t)),
\end{equation}
and
\begin{equation}\label{eq:compari}
\frac{1}{C}\leq \frac{\langle y\rangle}{\langle z\rangle}\leq C
\end{equation}
in the time interval $\tau\in [0,T_1].$ Moreover,
\begin{equation}\label{eq:D123}
\sum_{k=1}^{3} \Big\| \langle z\rangle^{-5} e^{\frac{\alpha}{4} z^2} \tilde{D}_{k}\Big\|_{L^\infty}
\ls \beta^{\frac{13}{5}}(\tau(\sigma)).
\end{equation}
\end{lemma}
This lemma is proved in Section~\ref{TildeProofs}.
\medskip

Now we apply Duhamel's principle to the evolution equation~\eqref{eq:eta}, obtaining
\begin{equation*}
\dz\big[\eta(\sigma(T_1))\big]=e^{-\sigma(T_1) L_0(\alpha)} \partial_{\theta}\eta_0
+\int_0^{\sigma(T_1)} e^{-(\sigma(T_1)-\sigma_1)L_0(\alpha)}
\sum_{k=1}^{3}\tilde{D}_{k}(\sigma_1)\,\mathrm{d}\sigma_1.
\end{equation*}
The definition of $\eta$ implies that at time $t=T_1$, one has $y=z$, $\alpha=a(\tau(T_1))$ and
$\eta=\xi$. Together with the orthogonality conditions for $\dz\xi$ in \eqref{eq:ortho10}, this implies that
\begin{equation*}
\begin{split}
\dz\eta(\sigma(T_1))&=\bP_{7}\big[\partial_{\theta}\eta(\sigma(T_1))\big]\\
& =\bP_7 \big[e^{-\sigma L(\alpha)} \partial_{\theta}\eta_0\big]
+\int_0^{\sigma(T_1)} e^{-(\sigma(T_1)-\sigma_1)L_0(\alpha)}
\bP_7 \left\{\sum_{k=1}^{3}\tilde{D}_{k}(\sigma_1)\right\}\mathrm d\sigma_1
\end{split}
\end{equation*}
where $\bP_7$ denotes orthogonal projection onto the subspace orthogonal to the seven
functions in \eqref{eq:ortho10}.

To prove that $\eta$ is sufficiently small, we use the following propagator estimates.

\begin{lemma}		\label{PRO:propagator}
Let the assumptions of Section~\ref{SBMI} hold, Then for any time $\sigma\geq 0$, one has
\begin{equation*}
\Big\| \langle z\rangle^{-5}e^{\frac{\alpha z^2}{4}}e^{-\sigma L(\alpha)}\bP_{7}[\eta]\Big\|_{L^\infty}
\ls e^{-\alpha \sigma} \Big\| \langle z\rangle^{-5}e^{\frac{\alpha z^2}{4}}\eta \Big\|_{L^\infty}.
\end{equation*}
\end{lemma}

This lemma is proved in Section~\ref{subsec:propa} below.
\medskip

Now we use estimate~\eqref{eq:D123} to obtain
\begin{align*}
\Big\|\langle z\rangle^{-5} e^{\frac{\alpha z^2}{4}}\dz\eta(\sigma(T_1)) \Big\|_{L^\infty}
&\ls e^{-\alpha \sigma} \Big\|\langle z\rangle^{-5} e^{\frac{\alpha z^2}{4}}\partial_{\theta}\eta(0)\Big\|_{L^\infty}\\
&\qquad + \int_0^{\sigma(T_1)} e^{-\alpha \sigma} \Big\|\langle z\rangle^{-5} e^{\frac{\alpha z^2}{4}}
\sum_{k=1}^{3}\tilde{D}_k \Big\|_{L^\infty}\,\mathrm d\sigma\\
&\ls e^{-\alpha \sigma} \beta^{\frac{21}{10}}(0)+\int_0^{\sigma(T_1)}
e^{-\alpha \sigma}\beta^{\frac{13}{5}}(\tau(\sigma))\,\mathrm d\sigma\\
&\ls \beta^{\frac{21}{10}}(\sigma(T_1)).
\end{align*}
Note that here we used the new condition
$ \|\langle x\rangle^{-5} \partial_{\theta}u_0\|_{L^\infty}\ls b^{\frac{21}{10}}(0)$
implied by Assumption~[A1]  to get
$ \|\langle y\rangle^{-5} e^{\frac{a(0) y^2}{4}}\dz\xi(0)\|_{L^\infty}\ls \beta^{\frac{21}{10}}(0)$.
Finally, we complete the proof by observing that
$$\langle z\rangle^{-5} e^{\frac{\alpha z^2}{4}}\partial_{\theta}\eta(\sigma(T_1))
=\langle y\rangle^{-5} e^{\frac{ay^2}{4}}\partial_{\theta}\xi(T_1).$$
The desired estimate follows, and the proof of Lemma~\ref{LemmaB} is complete.
\end{proof}

\subsection{Proof of Lemma~\ref{LM:higherWei}}	\label{subsec:remainder}
\begin{proof} [Proof of Lemma~\ref{LM:higherWei}]
We start by proving the orthogonality condition. By~\eqref{eq:orthoxi},
we see that  $\xi$ is orthogonal to seven functions. By a simple integration by parts in $\theta$,
we find that $\dz\xi$ is also orthogonal to these same functions, which proves \eqref{eq:ortho10}.

Next we turn to \eqref{eq:higherwei}. Three observations make it relatively straightforward
to estimate the various terms there: \textsc{(i)} Most of the terms are of higher orders in $\xi$,
$\partial_{y}v$, $\dz v$, and $\dy\dz v,$ which (appropriately weighted) are small in $L^{\infty}$.
\textsc{(ii)} The presence of the operator $\dz$ removes the slowly-changing $\theta$-independent
terms. And \textsc{(iii)}, the high weight $\langle y\rangle^{-5}$ here allows us to be generous when
doing estimates.

We begin with the term $D_2=\partial_{\theta}F=\partial_{\theta}[F_1(a,b)+F_2(a,b,\beta)]$.
The fact that $F_1$ is independent of $\theta$ makes $\partial_{\theta}F=\partial_{\theta}F_2$.
Moreover the $\beta_0$-component, which is also $\theta$-independent, will vanish after
applying $\partial_{\theta}.$
Together with the estimates
$\partial_{\tau}\beta_{k}-a\beta_k=\mathcal{O}(\beta^{\frac{16}{5}})$, $k=1,2$,
and $\partial_{\tau}\beta_{\ell}=\mathcal{O}(\beta^{\frac{16}{5}})$, $\ell=3,4$,
from~\eqref{eq:27in} of Lemma~\ref{LM:traj}, this implies that
\begin{equation*}
\Big\|\langle y\rangle^{-5} e^{\frac{ay^2}{4}} D_2 \Big\|_{L^\infty} \ls \beta^3.
\end{equation*}

For $D_1$, we relate $e^{\frac{ay^2}{4}}\xi$ to $v$ using \eqref{eqn:split2}, and we bound
$\langle y\rangle^{-2}\left(\frac{\frac{1}{2}+a}{2}-\frac{\frac{1}{2}+a}{2+by^2}\right)$ by $\beta$, obtaining
\begin{equation*}
\Big\|\langle y\rangle^{-5}e^{\frac{ay^2}{4}} D_{1} \Big\|_{L^\infty}\ls
\beta \bigg\{\|v^{-2}\partial_{\theta}v\|_{L^\infty}+\sum_{k=1}^{4}|\beta_{k}|\bigg\}
\ls  \beta^{\frac{5}{2}}.
\end{equation*}
Here, we used the estimate $\langle y\rangle^{-1}\ls v^{-1} $ from \eqref{eq:y-1} and
an assumption on $v^{-2}|\partial_{\theta}v|$ from the outputs of the first bootstrap
machine in Section~\ref{FirstOutput}.

We decompose $D_5$ as
\begin{align*}
\Big|\langle y\rangle^{-5}\partial_{\theta}N_3 e^{\frac{ay^2}{4}}\Big|
&\ls \langle y\rangle^{-2} \Big|V_{a,b}^{-2}-\frac{1}{2}\Big| 
\bigg[v^{-3}|\partial_{\theta}^3v|+\sum_{k=1}^{k}|\beta_{k}|\bigg]\\
 &\quad+\langle y\rangle^{-3} |v^{-2}-V_{a,b}^{-2}|  
 \bigg[v^{-2}|\partial_{\theta}^3v|+\sum_{k=1}^{k}|\beta_{k}|\bigg]\\
 &\quad+\langle y\rangle^{-3} \Big|\partial_{\theta} \Big[v^{-2}\frac{1+p^2}{1+p^2+q^2}\Big]\Big| \ v^{-2}|\partial_{\theta}^2 v|\\
 &\quad+\langle y\rangle^{-5} \Big|\partial_{\theta} \Big[v^{-1}\frac{2pq}{1+p^2+q^2}\partial_{\theta}\partial_{y}v\Big] \Big|\\
 &\quad+\langle y\rangle^{-5} \Big|\partial_{\theta} \Big[v^{-2}\frac{q}{1+p^2+q^2}\partial_{\theta}v\Big] \Big|\\
 &= \sum_{l=1}^{5} K_{\ell},
\end{align*}
with $K_{\ell}$ $(\ell=1,\dots,5)$ naturally defined.

For $K_1$, we use the assumption $A\ls\beta^{-\frac{1}{20}}$ for the quantity
$A$ defined in ~\eqref{eq:majorA} to estimate the factor
\begin{equation*}
\langle y\rangle^{-2} \Big|V_{a,b}^{-2}-\frac{1}{2}\Big| 
\ls \langle y\rangle^{-2}\Big( |\frac{1}{2}-a|+by^2\Big)  \ls \beta.
\end{equation*}
By \eqref{eq:thetaDe}, the quantity $v^{-2}|\dz^3 v|$ is bounded by $\beta^{\frac{33}{20}}$.
Combining these estimates with the consequence $|\beta_{k}|\ls \beta^2$ of Lemma~\ref{LM:traj}
implies that
\begin{equation*}   %\label{eq:k1beta}
|K_{1}|\ls \beta^{\frac{13}{5}}.
\end{equation*}

For $K_2$, we use ideas similar to those we used to control $K_1$ to obtain
\begin{equation*}
\Big\|\langle y\rangle^{-5}e^{\frac{ay^2}{4}}K_2\Big\|_{L^\infty} \ls \beta^{\frac{16}{5}}.
\end{equation*}

The presence of factors $q=v^{-1}\partial_{\theta}v$ and $\partial_{\theta}\partial_{y}v$ 
makes it relatively easy to control the terms in $K_3$ and $K_4$. Indeed, one may apply
the estimates in \eqref{eq:thetaDe} to obtain
\begin{equation*}
\Big\|\langle y\rangle^{-5} e^{\frac{ay^2}{4}}[K_3+K_4]\Big\|_{L^\infty} \ls \beta^3.
\end{equation*}

The claimed estimate for $D_5$ follows when one combines the estimates above.
\end{proof}

\subsection{Proof of Lemma~\ref{TildeLM:traj}}	\label{TildeProofs}
\begin{proof}[Proof of Lemma~\ref{TildeLM:traj}]
Because the first three estimates are obtained exactly as in the corresponding part of ~\cite{GKS11},
we do not repeat those arguments here.

For ~\eqref{eq:D123}, we begin by decomposing $\tilde{D}_1$ into three terms,
\begin{equation*}
\tilde{D}_1 = \Big[1-\frac{\lambda^2}{\lambda_1^2}\Big] \frac{2-2a}{2+b(\tau(\sigma))y^2} 
\partial_{\theta}\eta+\frac{2a-2\alpha}{2+b(\tau(\sigma))y^2}\partial_{\theta}\eta
+ \Big(2-2\alpha\Big)
\frac{b(\tau(\sigma))y^2}{2(2+b(\tau(\sigma))y^2)}\partial_{\theta}\eta.
\end{equation*}
The factor $by^2$ in the last term on the \textsc{rhs} is controlled by $\langle y\rangle^{-2}by^2\leq b\lesssim \beta$,
which together with the estimates for $a-\alpha$ and $\lambda/\lambda_1-1$ in \eqref{eq:appTra}
implies that
\begin{equation*}
\Big\|\langle y\rangle^{-5} e^{\frac{\alpha z^2}{4}}\tilde{D}_1\Big\|_{L^\infty}
\ls  b \Big\|\langle z\rangle^{-3} e^{\frac{\alpha}{4}z^2}\partial_{\theta}\eta(\sigma)\Big\|_{L^\infty}
\ls  \beta \Big\|\langle y\rangle^{-3} e^{\frac{a}{4}y^2}\partial_{\theta}\xi(\tau(\sigma))\Big\|_{L^\infty}.
\end{equation*}
Note that in the last step, we used \eqref{eq:compari} and the definition of $\eta.$ Using~\eqref{eq:xieta},
which relates $\partial_{\theta}\xi$ to $\partial_{\theta}v$, we get
\begin{equation*}
\Big\|\langle y\rangle^{-5} e^{\frac{\alpha z^2}{4}}\tilde{D}_1\Big\|_{L^\infty}\ls \beta^{\frac{51}{20}}.
\end{equation*}

We relate the quantities $\tilde{D}_2$ and $\tilde{D}_3$ to $D_2$ and $D_3$, respectively,
using \eqref{eq:appTra}. The we apply estimate~\eqref{eq:higherwei} to obtain
\begin{equation*}
\begin{split}
\Big\|\langle z\rangle^{-5} e^{\frac{\alpha}{4}z^2}\tilde{D}_{k}(\sigma)\Big\|_{L^\infty}
& \ls \Big\|\langle y\rangle^{-5} e^{\frac{a}{4}y^2}D_{k}(\tau(\sigma))\Big\|_{L^\infty} \\
& \ls \beta^{\frac{13}{5}}(\tau(\sigma)). 
\end{split}
\end{equation*}

The proof is complete.
\end{proof}

\subsection{Proof of Lemma~\ref{PRO:propagator} }		\label{subsec:propa}
\begin{proof}[Proof of Lemma~\ref{PRO:propagator} ]
In our earlier work \cite{GKS11}, we proved a similar result but with the weight $\langle z\rangle^{-3}$.
To prove an estimate with the weight $\langle z\rangle^{-5}$, we slightly generalize a result proved in
\cite{DGSW08}, which in turn is motivated by a result from \cite{BrKu94}.

In \cite{DGSW08}, the functions are $\theta$-independent. So before adapting the arguments there,
we decompose $\eta$ as
\begin{equation*}
\eta(z,\theta)=\sum_{k=-\infty}^{\infty} e^{ik\theta} f_{k}(z),
\end{equation*}
with $f_{k}(z):=\frac{1}{2\pi} \int_{\bS^1} e^{-ik\theta} \eta(z,\theta)\,\mathrm d\theta.$
The explicit form of $\bP_7$ allows us to write 
\begin{align}\label{eq:decomP7}
\bP_7[\eta]=\bP_3[f_0](z)+ e^{i\theta}\bP_2[f_{1}](z)
+e^{-i\theta}\bP_2[f_{-1}](z)+\sum_{|k|>2} e^{ik\theta} f_{k}(z)
\end{align}
and 
\begin{align*}
e^{-\sigma L(\alpha)}\bP_7[\eta]&=
e^{-\sigma L_{0}(\alpha)}\bP_3[f_0](z)
+ e^{i\theta}e^{-\sigma [L_{0}(\alpha)+\frac{1}{2}]}\bP_2[f_{1}](z)\\
&\quad+e^{-i\theta}e^{-\sigma [L_{0}(\alpha)+\frac{1}{2}]}\bP_2[f_{-1}](z)
+\sum_{|k|>2} e^{ik\theta}e^{-\sigma [L_{0}(\alpha)+\frac{k^2}{2}]} f_{k}(z).
\end{align*}
Here, $\bP_3$ denotes orthogonal projection onto the subspace orthogonal to $e^{-\frac{\alpha z^2}{4}}$,
$ze^{-\frac{\alpha z^2}{4}}$, and $(\alpha z^2-1)e^{-\frac{\alpha z^2}{4}}$, while $\bP_2$ denotes orthogonal
projection onto the subspace orthogonal to $e^{-\frac{\alpha z^2}{4}}$ and $ze^{-\frac{\alpha z^2}{4}}$.

In what follows, we focus on the first term on the \textsc{rhs} above; the remaining terms are estimated
similarly. To generalize Lemma~17 from~\cite{DGSW08}, we note that in the present situation, one has
$\beta=0$ and $V=0$, which will make our analysis easier.

Following~\cite{DGSW08}, we derive an integral kernel for $e^{-\sigma L_{0}(\alpha)}$ such that
\begin{align}\label{eq:kernel10}
e^{-\sigma L_{0}(\alpha)}\bP_3[f_0](z)=e^{\frac{\alpha}{4}|z|^2}\int_{-\infty}^{\infty} U_{0}(z,y) f(y)\,\mathrm{d}y,
\end{align}
where
$$f(y):=e^{-\frac{\alpha}{4}|y|^2}\bP_3[f_0](y),$$ 
and the integral kernel $U_0$ is
$$U_{0}(z,y):=4\pi (1-e^{-2\alpha \sigma})^{-\frac{1}{2}}
\sqrt{\alpha}\;e^{2\alpha\sigma} e^{-\frac{\alpha(z-e^{-\alpha \sigma}y)^2}{2(1-e^{-2\alpha \sigma})}}.$$

To obtain the decay estimate, we integrate by parts in $y$ three times, exploiting the fact
that $\partial_{y}U_{0}(z,y)$ has a factor $e^{-\alpha \sigma}$. To prepare for this, namely to ensure 
that the functions we consider are in appropriate spaces, we use the fact that
$\bP_3[f_0]\perp\{e^{-\frac{\alpha}{4}|z|^2}, \,ze^{-\frac{\alpha}{4}|z|^2},\,z^2 e^{-\frac{\alpha}{4}|z|^2}\}$ 
to see that $f\perp \{1,\ z,\ z^2\}$, hence that
\begin{equation}\label{eq:threeInParts}
\begin{split}
\int_{-\infty}^{\infty} f(y)\,\mathrm dy&=0, \\
\int_{-\infty}^{\infty}\int_{-\infty}^{y} f(y_1)\,\mathrm dy_1\,\mathrm dy
	&=\int_{-\infty}^{\infty} y f(y)\,\mathrm dy=0, \\
\int_{-\infty}^{\infty}\int_{-\infty}^{y}\int_{-\infty}^{y_1} f(y_2)\,\mathrm dy_2
	\,\mathrm dy_1\,\mathrm dy&=\int_{-\infty}^{\infty} y^2 f(y)\,\mathrm dy=0.
\end{split}
\end{equation}
Furthermore, there exists a constant $C$, independent of $y$, such that
\begin{equation}\label{eq:f-3}
|f^{(-3)}(y)|\leq C e^{-\frac{\alpha y^2}{2}}(1+|y|^2)\|\langle z\rangle^{-5} e^{\frac{\alpha z^2}{4}}f_0\|_{L^\infty},
\end{equation}
where
$$f^{(-3)}(y):=\int_{-\infty}^{y}\int_{-\infty}^{y_1}\int_{-\infty}^{y_2} f(y_3)\,\mathrm dy_3
	\,\mathrm dy_2\,\mathrm dy_1
=-\int^{\infty}_{y}\int^{\infty}_{y_1}\int^{\infty}_{y_2} f(y_3)
	\,\mathrm dy_3\,\mathrm dy_2\,\mathrm dy_1.
$$
To obtain estimate~\eqref{eq:f-3}, we repeatedly used the fact that the inequality
$$\int_{y}^{\infty} (1+z)^{m}e^{-\frac{\alpha z^2}{2}}\,\mathrm dz\ls
(1+y)^{m-1}e^{-\frac{\alpha y^2}{2}}
$$
holds for $y\geq 0$ and $m\in \mathbb{R}$.

The work that follows is different from \cite{DGSW08}.
Returning to ~\eqref{eq:kernel10}, we integrate by parts three times in $y$ to find that
\begin{equation*}
e^{-\sigma L_{0}(\alpha)}\bP[f_0](z)=-e^{\frac{\alpha z^2}{4}}
	 \int \partial_{y}^3 U_{0}(z,y) f^{(-3)}(y)\,\mathrm dy.
\end{equation*}
Then we calculate directly and apply estimate~\eqref{eq:f-3} to obtain
\begin{equation}\label{eq:lastTep}
\begin{split}
\Big| \langle z\rangle^{-5} &e^{\frac{\alpha z^2}{4}}e^{-\sigma L_{0}(\alpha)}\bP_3[f_0](z)\Big| \\
&\ls \frac{ e^{-3\alpha\sigma} }{(1-e^{-2\alpha \sigma})^3}e^{\frac{\alpha z^2}{2}}\langle z\rangle^{-5} \int \big(|z|+|y|+1\big)^{3}\,
	U_{0}(z,y) \big| f^{(-3)}(y)\big|\,\mathrm dy \\ 
&\ls\frac{ e^{-3\alpha\sigma} }{(1-e^{-2\alpha \sigma})^3}\langle z\rangle^{-5} e^{\frac{\alpha z^2}{2}}
\Big\|\langle z\rangle^{-5} e^{\frac{\alpha z^2}{4}}f_0\Big\|_{L^\infty}
\int \big(|z|+|y|+1\big)^{3}\,U_{0}(z,y) e^{-\frac{\alpha y^2}{2}}\big(1+|y|^2\big)\,\mathrm dy \\
&\ls \frac{ e^{-3\alpha\sigma} }{(1-e^{-2\alpha \sigma})^3} 
	\Big\|\langle z\rangle^{-5} e^{\frac{\alpha z^2}{4}}f_0\Big\|_{L^\infty}
	\sum_{k=1,2,3,4,5} \langle z\rangle^{-k} e^{\frac{\alpha z^2}{2}}
	\int U_{0}(z,y) e^{-\frac{\alpha y^2}{2}} \big(1+|y|\big)^{k}\,\mathrm dy.	
\end{split}
\end{equation}
Here in the second step, we use the simple observation
\begin{align*}
\langle z\rangle^{-5} (|z|+|y|+1)^{3} (1+|y|^2)\lesssim \sum_{k=1,2,3,4,5}\langle z\rangle^{-k} (1+|y|)^k
\end{align*}
to show that
\begin{align}\label{eq:k12345}
\sum_{k=1,2,3,4,5} \langle z\rangle^{-k} e^{\frac{\alpha z^2}{2}}\int U_{0}(z,y) 
e^{-\frac{\alpha y^2}{2}}\big(1+|y|\big)^{k}
	\,\mathrm dy \ls e^{2\alpha \sigma},
\end{align}
Then we apply the same arguments as in the proof of Lemma~16 of \cite{DGSW08}, where the cases 
$k=1,\dots, 4$ were verified. Because this adaptation is straightforward and simple, it does not need
to be detailed here.

Now for $\sigma\geq1$, estimates~\eqref{eq:lastTep} and \eqref{eq:k12345} together imply the
desired decay estimate. On the other hand, for small $\sigma$, we apply \eqref{eq:k12345} directly
to \eqref{eq:kernel10} to obtain a uniform bound. This completes the proof.
\end{proof}

\section{Proof of Lemma~\ref{LemmaC}}	\label{sec:proof723}

\begin{proof}[Proof of Lemma~\ref{LemmaC}]
We first derive evolution equations for $\xi_\pm$ and
$(e^{-\frac{ay^2}{4}}\partial_{y}e^{\frac{ay^2}{4}}\xi)_\pm$.
To simplify notation, we introduce new functions $\psi_{\pm}$ defined by
\begin{equation*}
\psi_{\pm}:=e^{-\frac{ay^2}{4}}\partial_{y} \Big(e^{\frac{ay^2}{4}}\xi_{\pm}\Big).
\end{equation*}
Next we make the following observations.

\begin{lemma}\label{LM:vpm}
The functions $\xi_{\pm}$ and  $\psi_{\pm}$ satisfy the orthogonality conditions
\begin{align}\label{eq:xipmOrtho}
\xi_{\pm}\perp\Big\{e^{-\frac{ay^2}{4}},\ ye^{-\frac{ay^2}{4}}\Big\}
\qquad\text{and}\qquad
\psi_{\pm}\perp\Big\{e^{-\frac{ay^2}{4}}\Big\},
\end{align}
and evolve by
\begin{equation}\label{eq:eqPhiPm}
\begin{split}
\partial_{\tau}\xi_{\pm} = & -L_{0}(a)\xi_{\pm}+e^{-\frac{ay^2}{4}}[G_1+G_2],\\
\partial_{\tau}\psi_{\pm} = & -[L_{0}(a)+a]\psi_{\pm}+e^{-\frac{ay^2}{4}}[G_3+G_4],
\end{split}
\end{equation}
where
$$L_0(a) := - \dy^2 + \frac{a^2 + \partial_\tau a }4 y^2 -\frac{3a}2.
$$
The terms $G_{k}$ $(k=1,\dots,4)$ on the \textsc{rhs} above satisfy the estimates
\begin{align*}
\|\langle y\rangle^{-2}G_1\|_{L^\infty}\ls \beta^{\frac{33}{10}},\qquad&\ \|\langle y\rangle^{-3}G_2\|_{L^\infty}
	\ls\beta^{\frac{73}{20}},\\
\|\langle y\rangle^{-1}G_3\|_{L^\infty}\ls \beta^{\frac{53}{20}}, \qquad&\  \|\langle y\rangle^{-2}G_4\|_{L^\infty}
	\ls\beta^{\frac{63}{20}}.
\end{align*}
\end{lemma}
The orthogonality conditions ~\eqref{eq:xipmOrtho} follow directly from the corresponding properties
for $\xi$. The remainder of the lemma is proved in Section~\ref{eq:phipm}.
\medskip

In what follows, we focus on estimating $\xi_{+}$. The remaining estimates are proved similarly.
For $\xi_{+}$, after going through the same procedures as we followed when deriving ~\eqref{eq:eta},
 we obtain
\begin{align}\label{eq:etaplus}
\partial_{\sigma}[\eta_{+}]=-L_{0}(\alpha)[\eta_{+}]+\tilde{G}_{1}(\sigma)+\tilde{G}_2(\sigma),
\end{align}
where $L_{0}(\alpha)$ is the autonomous linear operator
$$L_{0}(\alpha):=- \partial_{z}^2 + \frac{\alpha^2 }4 z^2 -\frac{3\alpha}2,$$
and $\tilde{G}_{k}$ $(k=1,2)$ are reparametrizations of $G_{k}$ $(k=1,2)$ defined 
similarly to \eqref{eq:difTidD2}.
The terms appearing in the evolution equation~\eqref{eq:etaplus} satisfy the following estimates.

\begin{lemma}
If the assumptions of Section~\ref{SBMI} hold, then for any $\tau\leq T_1$ and any weight $\ell\geq 0$,
one has
\begin{equation}	\label{eq:WEta}
\big\|\langle z\rangle^{-\ell} e^{\frac{\alpha z^2}{4}}\eta_{+}(\sigma)\big\|_{L^\infty}
	\ls \big\|\langle y\rangle^{-\ell} e^{\frac{a y^2}{4}}\xi_{+}(\tau(\sigma))\big\|_{L^\infty},
\end{equation}
along with
\begin{equation} \label{eq:differentWei}
\big\|\langle z\rangle^{-2} \tilde{G}_{1}\big\|_{L^\infty}\ls  \beta^{\frac{33}{10}}(\tau(\sigma))
\quad\text{and}\quad
\big\|\langle z\rangle^{-3}\tilde{G}_{2}\big\|_{L^\infty}\ls \beta^{\frac{73}{20}}(\tau(\sigma)).
\end{equation}
\end{lemma}

The proofs of estimates~\eqref{eq:WEta} and \eqref{eq:differentWei} are almost identical
to those that appear in Lemma~\ref{LM:traj} above, hence are not repeated here.
\medskip

Returning to equation~\eqref{eq:etaplus}, we apply Duhamel's principle to obtain
\begin{align}\label{eq:etapplus}
\eta_{+}(\sigma)=e^{-\sigma L_{0}(\alpha)} \eta_{+}(0)+\int_0^{\sigma}
e^{-(\sigma-\sigma_1)L_{0}(\alpha)} \big[\tilde{G}_1(\sigma_1)+\tilde{G}_2(\sigma_1)\big]\,\mathrm d\sigma_1.
\end{align}
We again rely on a propagator estimate to prove the decay of $\eta_{+}$.
Observe that the quantum harmonic oscillator $L_{\alpha}$ has nonpositive eigenvalues with
eigenfunctions $e^{-\frac{\alpha}{4}z^2}$ and $ze^{-\frac{\alpha}{4}z^2}$, which might make $\eta$
grow. To control these eigenvectors, we use the orthogonality properties of $\eta_{+}.$ 

Recall \eqref{eq:scaledTime} and note that at time $\sigma=\sigma(T_1)$, namely $\tau=T_1$,
we have $\eta_{+}(\sigma)=\xi_{+}(T_1)$ and $e^{\alpha z^2/4}=e^{a(\tau)y^2/4}$.
Hence by \eqref{eq:xipmOrtho}, we have
\begin{equation*}
\bP_2 [\eta_{+}](\sigma(T_1))=\eta_{+}(\sigma(T_1)),
\end{equation*}
where $\bP_2$ denotes the orthogonal projection onto the subspace orthogonal to the span
of $\{e^{-\alpha z^2/4},\;ze^{-\alpha z^2/4}\}$ (i.e.~orthogonal to the unstable subspace of
$L_{\alpha}$). We apply $\bP_2$ to both sides of ~\eqref{eq:etapplus}, obtaining
\begin{equation}\label{A123decomposition}
\begin{split}
\eta_{+}(\sigma(T_1))
&=\bP_2 \big[e^{-\sigma(T) L_{0}(\alpha)} \eta_{+}\big](0)\\
&\qquad + \int_0^{\sigma(T_1)} e^{-(\sigma(T)-\sigma_1)L_{0}(\alpha)}
\bP_2 \big[\tilde{G}_1(\sigma_1)+\tilde{G}_2(\sigma_1)\big]\,\mathrm d\sigma_1 \\
&=A_1+A_2+A_3,
\end{split}	
\end{equation}
where the terms $A_1,A_2,A_3$ are naturally defined.

We estimate $\bP_2 e^{-\sigma(T)}L_{0}(\alpha)$ as follows.
\begin{lemma}	\label{Nested}
For all times that the assumptions of Section~\ref{SBMI} hold, for any smooth function $g$,
any such time $\sigma>0$, and any weight $k=2,3$, one has
\begin{equation}\label{eq:wei23}
\Big\|\langle z\rangle^{-k}e^{\frac{\alpha z^2}{4}}\bP_2 \big[e^{-\sigma L_{0}(\alpha)} g\big] \Big\|_{L^\infty}
\ls e^{-\alpha\sigma} \Big\|\langle z\rangle^{-k}e^{\frac{\alpha z^2}{4}} g\Big\|_{L^\infty};
\end{equation}
while for the chosen weight $\langle z\rangle^{-\frac{11}{10}}$, one has
\begin{equation}\label{eq:we1110}
\Big\|\langle z\rangle^{-\frac{11}{10}}e^{\frac{\alpha z^2}{4}}\bP_2 \big[e^{-\sigma L_{0}(\alpha)} g\big]
\Big\|_{L^\infty}
\ls e^{-\frac{\alpha}{10}\sigma} \Big\|\langle z\rangle^{-\frac{11}{10}}e^{\frac{\alpha z^2}{4}} g\Big\|_{L^\infty}.
\end{equation}
\end{lemma}

\begin{proof}[Proof of Lemma~\ref{Nested}]
Both cases of Estimate~\eqref{eq:wei23} are proved by following the arguments in the proof
of  Lemma~\ref{PRO:propagator}, \emph{mutatis mutandis}.

Estimate~\eqref{eq:we1110} is obtained via an interpolation technique, as in Proposition~11.5
from~\cite{GS09}. Here, one interpolates between 
\begin{align*}
\Big\|\langle z\rangle^{-1}e^{\frac{\alpha z^2}{4}}\bP_2[e^{-\sigma L_{0}(\alpha)} g]\Big\|_{L^\infty}
&= \Big\|\langle z\rangle^{-1}e^{\frac{\alpha z^2}{4}}\bP_1\big[e^{-\sigma L_{0}(\alpha)} \bP_2 g\big]\Big\|_{L^\infty} \\
&\ls \Big\|\langle z\rangle^{-1}e^{\frac{\alpha z^2}{4}} g\Big\|_{L^\infty}
\end{align*}
and the $k=2$ case of estimate~\eqref{eq:wei23}. Note that $\bP_1$ denotes orthogonal projection
onto the subspace orthogonal to $e^{-\frac{\alpha y^2}{4}}.$ We omit further details.
\end{proof}

Continuing with the proof of Lemma~\ref{LemmaC}, we now apply~\eqref{eq:WEta} to estimate the
first term $A_1$ in the decomposition~\eqref{A123decomposition} by
\begin{align*}
\Big\|\langle z\rangle^{-\frac{11}{10}}e^{\frac{\alpha z^2}{4}}A_1\Big\|_{L^\infty}
&\ls e^{-\frac{\alpha}{10}\sigma(T)} \Big\|\langle z\rangle^{-\frac{11}{10}} e^{\frac{\alpha z^2}{4}}\eta_{+}(0)
\Big\|_{L^\infty}\\
&\ls e^{-\frac{\alpha}{10}\sigma(T)} \Big\|\langle y\rangle^{-\frac{11}{10}} e^{\frac{a y^2}{4}}\xi_{+}(0)
\Big\|_{L^\infty}\\
&\ls \beta^{\frac{53}{20}}(T_1).
\end{align*}
Note that here we use the stronger\footnote{That is to say, stronger than we needed in \cite{GKS11}.}
restrictions
$\|\langle x\rangle^{-\frac{11}{10}} (u_0)_{\pm}\|_{L^\infty}\ls b^{\frac{53}{20}}(0)=\beta^{\frac{53}{20}}(0)$ and
$\|\langle x\rangle^{-\frac{11}{10}}\dx(u_0)_{\pm}\|_{L^\infty}\ls b^{\frac{53}{20}}(0)=\beta^{\frac{53}{20}}(0)$ 
on the initial data contained in Assumption~[A1] to get the estimate
$\|\langle y\rangle^{-\frac{11}{10}} e^{\frac{a y^2}{4}}\xi_{\pm}(0)\|_{L^\infty}\ls\beta^{\frac{53}{20}}(0)$
and the estimate
$\|\langle y\rangle^{-\frac{11}{10}} \dy e^{\frac{a y^2}{4}}\xi_{\pm}(0)\|_{L^\infty}\ls\beta^{\frac{53}{20}}(0)$.
It follows that in the region $\beta y^2\leq20$,
\begin{equation}\label{eq:a1Inner}
\big| e^{\frac{\alpha z^2}{4}}A_1\big| \ls \beta^{\frac{21}{10}}(T_1).
\end{equation}
Here we also used the fact that $y=z$ at time $\tau=T_1$.

To estimate the second term in the decomposition~\eqref{A123decomposition}, we use
\eqref{eq:differentWei}, obtaining
\begin{align*}
\Big\|\langle z\rangle^{-2}e^{\frac{\alpha z^2}{4}}A_2\Big\|_{L^\infty}
\ls\int_0^{\sigma(T)}e^{-\frac{1}{8}(\sigma(T_1)-\sigma_1)}
\beta^{\frac{33}{10}}(\tau(\sigma_1))\,\mathrm d\sigma
\ls\beta^{\frac{33}{10}}(T_1).
\end{align*}
Hence in the region $\beta y^2\leq 20$, one has
\begin{align}\label{eq:a2Inner}
\big|e^{\frac{\alpha z^2}{4}}A_2\big| \ls \beta^{\frac{23}{10}}(T_1).
\end{align}

To estimate the final term $A_3$ in the decomposition~\eqref{A123decomposition}, we use
a different norm: here we apply \eqref{eq:differentWei} again to get
\begin{equation*}
\Big\|\langle z\rangle^{-3}e^{\frac{\alpha z^2}{4}}A_3\Big\|_{L^\infty}
\ls \int_0^{\sigma(T_1)}e^{-\frac{1}{8}(\sigma(T_1)-\sigma_1)}
\beta^{\frac{19}{5}}(\tau(\sigma_1))\,\mathrm d\sigma\ls\beta^{\frac{19}{5}}(T_1).
\end{equation*}
Thus in the region $\beta y^2\leq 20$, one has
\begin{align}\label{eq:a3Inner}
\big|e^{\frac{\alpha z^2}{4}}A_3\big| \ls \beta^{\frac{23}{10}}(T_1).
\end{align}

Collecting estimates~\eqref{eq:a1Inner}--\eqref{eq:a3Inner} above yields
$$
\Big|e^{\frac{a(\tau)y^2}{4}}\xi_{+}(T_1)\Big|
= \Big| e^{\frac{\alpha z^2}{4}}\eta_{+}(\sigma(T_1))\Big| \ls\beta^{\frac{23}{10}}(T_1)
\qquad\text{if}\quad \beta y^2\leq 20.
$$
Because $T_1\geq 0$ is arbitrary, this completes the proof of the estimate
for $\xi_+$ in Lemma~\ref{LemmaC}, modulo the proof of Lemma~\ref{LM:vpm}
that appears below. The remaining estimates are obtained in a wholly analogous manner.
\end{proof}

\subsection{Proof of  Lemma~\ref{LM:vpm}}\label{eq:phipm}
\begin{proof}[Proof of  Lemma~\ref{LM:vpm}] The proof is in two parts.

\textsc{Part (i)}
We start by deriving the first evolution equation in~\eqref{eq:eqPhiPm} and its associated
estimates. Instead of using equation~\eqref{eq:xi} for $\xi$, it is more convenient to work
directly from equation~\eqref{MCF-v} for $v$, in order to see and exploit certain fortuitous cancellations.

Using notation defined in Section~\ref{Notation}, we take an inner product
$\langle e^{i\theta},\ \cdot\rangle_{\bS^1}$ with both sides of equation~\eqref{MCF-v} to obtain
\begin{equation}\label{eq:vplus}
\begin{split}
2\pi\,\dt v_+ & =\dt \big\langle e^{i\theta},\ v\big\rangle_{\bS^1}\\
& = \big\langle e^{i\theta},\ F_{1}(p,q)\partial_{y}^2 v\big\rangle_{\bS^1}
+ \big\langle e^{i\theta},\ v^{-2}F_{2}(p,q)\partial_{\theta}^2 v\big\rangle_{\bS^1}\\
&\quad + \big\langle e^{i\theta},\ v^{-1}F_{3}(p,q)\partial_{\theta}\partial_{y} v\big\rangle_{\bS^1}
+ \big\langle e^{i\theta},\ v^{-2}F_{4}(p,q)\partial_{\theta} v \big\rangle_{\bS^1}\\
&\quad - \big\langle e^{i\theta},\ ay\partial_{y}v \big\rangle_{\bS^1}
+ \big\langle e^{i\theta},\ av\big\rangle_{\bS^1} - \big\langle e^{i\theta},\ v^{-1}\big\rangle_{\bS^1}.
\end{split}
\end{equation}
To transform this equation into a form similar to \eqref{eq:eqPhiPm}, we have to allocate appropriate
terms to $G_1$ and $G_2.$ For this purpose, we decompose various terms on the \textsc{rhs} of
equation~\eqref{eq:vplus}.

We decompose the first term on the \textsc{rhs} of \eqref{eq:vplus} into two terms,
\begin{align*}
\big\langle e^{i\theta},\ F_{1}(p,q)\partial_{y}^2 v \big\rangle_{\bS^1}
& = \big\langle e^{i\theta},\ \partial_{y}^2 v\big\rangle_{\bS^1}
+ \big\langle e^{i\theta},\ [F_{1}(p,q)-1]\partial_{y}^2 v\big\rangle_{\bS^1}\\
& = 2\pi\,\partial_{y}^2 v_{+} + i \big\langle e^{i\theta},\ \partial_{\theta}[(F_{1}(p,q)-1)\partial_{y}^2 v]
\big\rangle_{\bS^1},
\end{align*}
with the final term above obtained by integrating by parts in $\theta$.
Direct computation yields
\begin{multline*}
\partial_{\theta}\big[(F_{1}(p,q)-1)\partial_{y}^2 v\big]
= - \frac{p^2}{1+p^2+q^2}\partial_{\theta}\partial_{y}^2 v\\
- 2 \frac{\partial_{y}\partial_{\theta}v \partial_{y}v}{1+p^2+q^2}\partial_{y}^2 v
+ 2 \frac{p^2 (p\partial_{\theta}p+q\partial_{\theta}q)}{(1+p^2+q^2)^2} \partial_{y}^2 v.
\end{multline*}

Now we group various terms on the \textsc{rhs} of \eqref{eq:vplus} by introducing two functions
$\tilde{G}_1$ and $\tilde{G}_2$, defined by
\begin{align}
2\pi\,\tilde{G}_1&:=\Big\langle e^{i\theta},\ -\frac{2\partial_{y}\partial_{\theta}v \partial_{y}v}{1+p^2+q^2}
	\partial_{y}^2 v\Big\rangle_{\bS^1}
+2\Big\langle e^{i\theta},\ \frac{p^2 (p\partial_{\theta}p+q\partial_{\theta}q)}{(1+p^2+q^2)^2}
	\partial_{y}^2 v\Big\rangle_{\bS^1}\\
&\quad+
\big\langle e^{i\theta},\ v^{-2}F_{2}(p,q)\partial_{\theta}^2 v\big\rangle_{\bS^1}
-\big\langle e^{i\theta},\ v^{-1}\big\rangle_{\bS^1}\nonumber\\
&\quad
+\big\langle e^{i\theta},\ v^{-1}F_{3}(p,q)\partial_{\theta}\partial_{y} v\big\rangle_{\bS^1}
+\big\langle e^{i\theta},\ v^{-2}F_{4}(p,q)\partial_{\theta} v\big\rangle_{\bS^1}\nonumber\\
\nonumber\\
2\pi\,\tilde{G}_2&:=-\frac{p^2}{1+p^2+q^2}\partial_{\theta}\partial_{y}^2 v.
\end{align}
This lets us write
\begin{equation}	\label{eq:tildeG12}
\partial_{\tau}v_{+}= \big[\partial_{y}^2-ay\partial_{y}+a\big] v_{+}+\tilde{G}_1+\tilde{G}_2,
\end{equation}
where $\tilde{G}_1$ and $\tilde{G}_2$ satisfy the following estimates.

\begin{lemma}\label{LM:tidG12}
For all times that the assumptions of Section~\ref{SBMI} hold, one has
\[
\|\langle y\rangle^{-2}\tilde{G}_1\|_{L^\infty}\ls\beta^{\frac{33}{10}}\qquad\text{and}\qquad
\|\langle y\rangle^{-3}\tilde{G}_2\|_{L^\infty}\ls\beta^{\frac{73}{20}}.
\]
\end{lemma}

\begin{proof}[Proof of Lemma~\ref{LM:tidG12}]
We first estimate $\tilde{G}_2$, computing directly to find that
\begin{align}\label{onlyG2}
\Big|\langle y\rangle^{-3}\frac{p^2}{1+p^2+q^2}\partial_{\theta}\partial_{y}^2 v\Big|
\ls\langle y\rangle^{-2} |p|^2 \ \langle y\rangle^{-1}|\partial_{\theta}\partial_{y}^2 v|\ls\beta^{\frac{73}{20}}.
\end{align}
In the final step above, we used the estimates $\langle y\rangle^{-1}\ls v^{-1}$ from~\eqref{eq:y-1} and
$v^{-1}|\partial_{\theta}\partial_{y}^2 v|\ls\beta^{\frac{33}{20}}$ from~\eqref{eq:thetaDe}, as well as the estimate
\begin{equation}\label{eq:yyv}
\begin{split}
\langle y\rangle^{-1} |p|  &\equiv \langle y\rangle^{-1}|\partial_{y}v|
\leq \langle y\rangle^{-1} |\partial_{y}V_{a,b}|\\
&\qquad +\sum_{k=0}^{4} |\beta_{k}|+|\partial_{y}\phi|
\leq \beta+\beta^{\frac{11}{10}}M_{1,1}\ls \beta
\end{split}
\end{equation}
implied by the assumptions on $\beta_{k}$ $(k=0,\cdots,4)$ and the assumption $M_{1,1}\ls 1.$

Now we turn to $\tilde{G}_1.$
The assumptions in ~\eqref{eq:thetaDe}, which are outputs of the first bootstrap machine,
show that its first two terms admit the estimate
$$\langle y\rangle^{-2} |\cdots|\ls \beta^{\frac{33}{10}}.
$$
We treat the third and fourth terms of \eqref{eq:vplus} together to exploit certain
cancellations. By integrating by parts in $\theta$, we obtain
\begin{align*}
\big\langle e^{i\theta},&v^{-2}F_{2}(p,q)\partial_{\theta}^2 v\big\rangle_{\bS^1}-\big\langle e^{i\theta},\ v^{-1}\big\rangle_{\bS^1}\\
& = -i \big\langle e^{i\theta},\ v^{-2}F_{2}(p,q)\partial_{\theta} v\big\rangle_{\bS^1}
+ i \big\langle e^{i\theta},\ v^{-2}\partial_{\theta} v \big\rangle_{\bS^1}
- \big\langle e^{i\theta},\ \partial_{\theta}[v^{-2}F_{2}(p,q)]\ \partial_{\theta}v\big\rangle_{\bS^1}\\
&= -i \big\langle e^{i\theta},\ v^{-2}[F_{2}(p,q)-1]\partial_{\theta} v\big\rangle_{\bS^1}
- \big\langle e^{i\theta},\ \partial_{\theta}[v^{-2}F_{2}(p,q)]\ \partial_{\theta}v\big\rangle_{\bS^1}.
\end{align*}
Recall that $F_2=\frac{1+p^2}{1+p^2+q^2}.$ Controlling the terms that appear here is not hard,
when one takes advantage of the presence of the operator $\partial_{\theta}$, the presence of
sufficiently many factors of $v^{-1}$ (helped by the estimate $\langle y\rangle^{-1}\ls v^{-1}$),
and employs estimates that have been proved and used frequently above. In this way, one finds
that the third and fourth terms are bounded by
$$\|\langle y\rangle^{-2}(\cdots)\|_{L^\infty}\ls \beta^{\frac{33}{10}}.$$
In the same way, we find that the fifth and sixth terms of $\tilde G_1$ also admit the estimate
$$\|\langle y\rangle^{-2}(\cdots)\|_{L^\infty}\ls\beta^{\frac{33}{10}}.$$
Collecting the estimates above, completes the proof.
\end{proof}

Returning to the proof of Lemma~\ref{LM:vpm}, we go back to equation~\eqref{eq:tildeG12}.
We must transform it further, because our objective is to derive an equation for $\xi_{+}$.
The decomposition of $v$ in equation~\eqref{eqn:split2} implies that
\begin{multline*}
\partial_{\tau}\xi_{+}
=-L_0(a)\xi_{+}+e^{-\frac{ay^2}{4}}[a-\partial_{\tau}]\left(\frac{1}{2}\beta_1+\frac{1}{2i}\beta_2\right)\\
	-y e^{-\frac{ay^2}{4}}\partial_{\tau}\left(\frac{1}{2}\beta_3 +\frac{1}{2i}\beta_4\right)
+e^{-\frac{ay^2}{4}}\tilde{G}_1+e^{-\frac{ay^2}{4}}\tilde{G}_2.
\end{multline*}
To derive the simple form above, we repeatedly used the fact that $e^{-\frac{ay^2}{4}}$ and 
$y e^{-\frac{ay^2}{4}}$ are eigenvectors of $L_{0}(a)$.
By the equations for $\beta_{k}$ $(k=1,\dots,4)$ in Lemma ~\ref{LM:traj}, we have 
\begin{equation*}
\big|[-\partial_{\tau}+a][\beta_1-i\beta_2]\big|
+\big|\partial_{\tau}[\beta_3 -i\beta_4  ]\big|\ls \beta^{\frac{33}{10}}.
\end{equation*} 
This leads us to write
$$\partial_{\tau}\xi_{+}=-L_0(a)\xi_{+}+G_1+G_2,$$
where
$$
G_1:=e^{-\frac{ay^2}{4}}[-\partial_{\tau}+a]\left(\frac{1}{2}\beta_1+\frac{1}{2i}\beta_2\right)
-e^{-\frac{ay^2}{4}}\partial_{\tau}\left(\frac{1}{2}\beta_3 y+\frac{1}{2i}\beta_4 y\right)
+e^{-\frac{ay^2}{4}}\tilde{G}_1,
$$
and $$G_2:=e^{-\frac{ay^2}{4}}\tilde{G}_2.$$
To conclude the first part of the proof of Lemma~\ref{LM:vpm}, we apply Lemma~\ref{LM:tidG12} to obtain
$$
\Big\|\langle y\rangle^{-2}e^{\frac{ay^2}{4}}G_1\Big\|_{L^\infty} 
\ls\beta^{\frac{33}{10}}\qquad\text{and}\qquad
\Big\|\langle y\rangle^{-3}e^{\frac{ay^2}{4}}G_2\Big\|_{L^\infty}\ls \beta^{\frac{73}{20}}.
$$
\medskip

\textsc{Part (ii)} We now derive the second evolution equation in~\eqref{eq:eqPhiPm} and prove
its associated estimates.

Our first task is to derive an evolution equation for
$(\dy v)_{+}=\frac{1}{2\pi}\langle e^{i\theta},\ \partial_{y}v\rangle_{\bS^1}$.
(For $(\dy v)_{-}$, it suffices to observe that $\dy v_{-}=\overline{(\dy v)_{+}}$.)
We compute directly to get
$$\dt \dy v_{+}=\partial_{y}^3 v_{+}-ay\partial_{y}^2 v_{+}+\tilde{G}_3+\tilde{G}_4,$$
where
\begin{align*}
\tilde{G}_3
& := \big\langle e^{i\theta},\ [F_1(p,q)-1]\partial_{y}^3 v\big\rangle_{\bS^1}
+ \big\langle e^{i\theta}, \partial_{y}F_{1}(p,q)\partial_{y}^2v\big\rangle_{\bS^1}\\
&\quad + \big\langle e^{i\theta},\ v^{-2}[F_2(p,q)-1]\partial_{y}\partial_{\theta}^2 v\big\rangle_{\bS^1}
+ \big\langle e^{i\theta},\ \partial_{y}[v^{-2}F_{2}(p,q)]\partial_{\theta}^2 v\big\rangle_{\bS^1}\\
&\quad + \big\langle e^{i\theta},\ \partial_{y}[v^{-1}F_{3}(p,q)\partial_{\theta}\partial_{y}v]\big\rangle_{\bS^1}
+ \big\langle e^{i\theta},\ \partial_{y}[v^{-2}F_{4}(p,q)\partial_{\theta}v]\big\rangle_{\bS^1}.
\end{align*}
and
$$
\tilde{G}_4:=\langle e^{i\theta},\ v^{-2}\partial_{\theta}^2\partial_{y}v\rangle_{\bS^1}+\langle e^{i\theta}, v^{-2}\partial_{y}v\rangle_{\bS^1}.
$$
We now proceed to estimate these quantities.

To control $\tilde{G}_4$, we twice integrate by parts in $\theta$, obtaining
\begin{align*}
\tilde{G}_4&=
\big\langle \partial_{\theta}^2[v^{-2}e^{i\theta}],\ \partial_{y}v\big\rangle_{\bS^1}
+\big\langle e^{i\theta},\ v^{-2}\partial_{y}v\big\rangle_{\bS^1}\\
&=\big\langle e^{i\theta} \partial_{\theta}^2 v^{-2}+ie^{i\theta} \partial_{\theta}v^{-2},\ \partial_{y}v\big\rangle_{\bS^1}.
\end{align*}
Using the estimate
$\|\langle y\rangle^{-1}\partial_{y}v\|_{L^\infty}=\cO(\beta)$ coming from ~\eqref{eq:yyv} and the estimate
$v^{-2}|\partial_{\theta}^2 v|=\cO(\beta^{\frac{33}{20}})$ coming from assumption~\eqref{eq:thetaDe},
which is an output of the first bootstrap machine, we conclude that
$$\|\langle y\rangle^{-1} \tilde{G}_4\|_{L^\infty}\ls\beta^{\frac{53}{20}}.$$ 

To control $\tilde{G}_3$, we combine assumptions in~\eqref{eq:yDe} and \eqref{eq:thetaDe} with
the estimate for $\langle y\rangle^{-1}|\partial_{y}v|$ coming from \eqref{eq:yyv} to get
\begin{equation*}
\|\langle y\rangle^{-2}\tilde{G}_3\|_{L^\infty}\ls\beta^{\frac{63}{20}}.
\end{equation*}

To conclude the second part of the proof, we again use the decomposition of $v$ in equation~\eqref{eqn:split2},
this time to decompose the evolution equation for $\partial_{y}\psi_{+}$. The arguments used here are virtually
identical to those which appear in part \textsc{(i)} of this proof. Hence we omit further details.
\end{proof}

\end{document}